\documentclass[draft]{amsart}

\usepackage[draft=false,pagebackref]{hyperref}
\usepackage{amssymb,esint,mathtools}
\newtheorem{thm}[equation]{Theorem}
\newtheorem{lem}[equation]{Lemma}
\newtheorem{cor}[equation]{Corollary}
\theoremstyle{definition}
\newtheorem{defn}[equation]{Definition}

\numberwithin{equation}{section}
\newcommand\abs[2][empty]{\csname#1\endcsname \lvert{#2}\csname#1\endcsname\rvert}
\newcommand\doublebar[2][empty]{\csname#1\endcsname \lVert{#2}\csname#1\endcsname\rVert}

\newcommand\mat[1]{\boldsymbol{#1}}
\newcommand\arr[1]{\boldsymbol{\dot{#1}}}
\newcommand\dist{\mathop{\mathrm{dist}}\nolimits}
\newcommand\Div{\mathop{\mathrm{div}}\nolimits}
\newcommand\Tr{\mathop{\smash{\boldsymbol{\rlap{$\arr{\phantom{T}}$}\mathrm{Tr}}}\vphantom{T}}\nolimits}
\newcommand\Trace{\mathop{\mathrm{Tr}}\nolimits}
\newcommand\M{\mathop{\smash{\arr{\mathrm{M}}}\vphantom{M}}\nolimits}
\newcommand\MM{\mathop{\smash{{\mathrm{M}}}\vphantom{M}}\nolimits}
\newcommand\supp{\mathop{\mathrm{supp}}\nolimits}
\newcommand\re{\mathop{\mathrm{Re}}\nolimits}
 \let\R\RR

 \let\N\NN
\newcommand\1{\mathbf{1}}
\newcommand\D{\mathcal{D}}
\newcommand\s{\mathcal{S}}

\newcommand\pureH{\parallel}

\newcommand\dmn{{n+1}}
\newcommand\pdmn{{(n+1)}}
\newcommand\dmnMinusOne{n}

\usepackage{xcolor}
\definecolor{mygreen}{rgb}{0,0.55,0}
\definecolor{darkred}{rgb}{0.7,0,0}
\definecolor{historyred}{rgb}{0.4,0,0}
\definecolor{purple}{rgb}{0.55,0,0}

\makeatletter\def\HyPsd@CatcodeWarning#1{}\makeatother

\begin{document}

\title[Nontangential estimates and the Neumann problem]{Nontangential estimates on layer potentials and the Neumann problem for higher order elliptic equations}

\author{Ariel Barton}
\address{Ariel Barton, Department of Mathematical Sciences,
			309 SCEN,
			University of Ar\-kan\-sas,
			Fayetteville, AR 72701}
\email{aeb019@uark.edu}
%\thanks{}

\author{Steve Hofmann}
\address{Steve Hofmann, 202 Math Sciences Bldg., University of Missouri, Columbia, MO 65211}
\email{hofmanns@missouri.edu}
\thanks{Steve Hofmann is partially supported by the NSF grant DMS-1361701.}

\author{Svitlana Mayboroda}
\address{Svitlana Mayboroda, Department of Mathematics, University of Minnesota, Minneapolis, Minnesota 55455}
\email{svitlana@math.umn.edu}
\thanks{Svitlana Mayboroda is partially supported by the  NSF CAREER Award DMS 1056004,  the NSF INSPIRE Award DMS 1344235, the Simons Fellowship,  and the Simons Collaborations in Mathematics and the Physical Sciences.}

\begin{abstract}
We solve the Neumann problem, with nontangential estimates, for higher order divergence form elliptic operators with variable $t$-independent coefficients. Our results are accompanied by nontangential estimates on higher order layer potentials.
\end{abstract}



\maketitle 

\tableofcontents

\section{Introduction}

%\cite{BarM16B}

% Maybe add some more references after the referee is assigned.

Consider the higher order elliptic differential operator $L$ given by
\begin{equation}\label{eqn:divergence}
Lu = (-1)^m \sum_{\abs{\alpha}=\abs{\beta}=m} \partial^\alpha (A_{\alpha\beta} \partial^\beta u)\end{equation}
for $m\geq 1$ a positive integer. The study of such operators, for $2m\geq 4$, is still fairly new. However, some results are known in the case of constant coefficients; see, for example, \cite{DahKPV97,Ver05,She07B,MitM13A} for some results related to those of the present paper. 

%Some constant coefficient results that are closely related to the present paper are well posedness of the Neumann problem for the biharmonic operator $\Delta^2$ established in \cite{Ver05,She07B}, well posedness of the Dirichlet problem in \cite{PipV95A}, and the layer potentials in \cite{MitM13A}.

In this paper we will consider coefficients $\mat A$ that are variable, but are bounded, elliptic, and $t$-independent in the sense that
\begin{equation}\label{eqn:t-independent}\mat A(x,t)=\mat A(x,s)=\mat A(x) \quad\text{for all $x\in\R^n$ and all $s$, $t\in\R$}.\end{equation}
Such $t$-independent coefficients have been studied extensively in the second order case. In particular, layer potentials have been used extensively in this case; see \cite{HofKMP15B,HofMayMou15,HofMitMor15} for some recent examples. In \cite{BarHM17pA,Bar17} we generalized layer potentials to the higher order case; we will continue to use them in the present paper.

The main result of the present paper (see Theorem~\ref{thm:Neumann:2:N} below) is existence of solutions to the Neumann problem
\begin{equation}
\label{eqn:neumann:regular:N}
\left\{\begin{aligned}
Lw&=0 \text{ in }\R^\dmn_+
,\\
\M_{\mat A}^+ w &= \arr g,
\\
\doublebar{\widetilde N_+(\nabla^m w)}_{L^2(\R^n)}
&\leq C\doublebar{\arr g}_{L^2(\R^\dmnMinusOne)}
\end{aligned}\right.\end{equation}
and the rough Neumann problem (or subregularity problem)
\begin{equation}
\label{eqn:neumann:rough:N}
\left\{\begin{aligned}
Lv&=0 \text{ in }\R^\dmn_+
,\\
\M_{\mat A}^+ v &= \arr h,
\\
\doublebar{\widetilde N_+(\nabla^{m-1} v)}_{L^2(\R^n)}
&\leq C\doublebar{\arr h}_{\dot W_{-1}^2(\R^\dmnMinusOne)}
\end{aligned}\right.\end{equation}
where $\M_{\mat A}^+$ denotes the Neumann boundary value operator (given in the second order case by $-\vec e_\dmn\cdot \mat A\nabla$, and by formula~\eqref{eqn:Neumann} below or by \cite[formula~(2.16)]{Bar17pB} in the general case), and where $\widetilde N_+$ denotes the modified nontangential maximal operator; this is the natural sharp estimate on solutions to boundary value problems. 
This work builds on our earlier results \cite{BarHM17,BarHM17pA,BarHM18}, in which we established well posedness in terms of the Lusin area integral. We will solve the problems (\ref{eqn:neumann:regular:N}--\ref{eqn:neumann:rough:N}) by establishing nontangential bounds on the double layer potential; we will in the process establish nontangential bounds on the single layer potential.

\subsection{Solutions to the Neumann problem}

%The main result of this paper concerns solutions to the Neumann problem. 

We begin by reviewing the history of the Neumann problem with $L^2$ or $\dot W^2_{-1}$ boundary data. In the case of harmonic functions, solutions to the problem
\begin{equation*}\Delta u=0 \text{ in }\Omega, \quad \nu\cdot\nabla u=g \text{ on }\partial\Omega, \quad \doublebar{N_\Omega(\nabla u)}_{L^2(\partial\Omega)} \leq C\doublebar{g}_{L^2(\partial\Omega)}\end{equation*}
for an arbitrary bounded %Lipschitz domain~$\Omega$ was shown to be well posed by Jerison and Kenig in \cite{JerK81B}. 
$C^1$ domain~$\Omega$ were constructed using the method of layer potentials in \cite{FabJR78}.
Here $\nu$ denotes the unit outward normal vector to~$\partial\Omega$, and $N_\Omega$ denotes the standard nontangential maximal function $N_\Omega F(X)= \sup\{\abs{F(Y)}: Y\in\Omega, \> \dist(Y,\partial\Omega)<2\abs{X-Y}\}$. 

By the divergence theorem, if $\nabla u$ is continuous up to the boundary, $\nu\cdot \nabla u=g$ on~$\partial\Omega$, and $\Delta u=\nabla\cdot\nabla u=0$ in~$\Omega$, then
\begin{equation}\label{eqn:neumann:harmonic}
\int_{\partial\Omega}\varphi\,g\,d\sigma
=
\int_\Omega \nabla\varphi\cdot \nabla u 
\end{equation}
for any smooth test function~$\varphi$. The left hand side converges provided only that $\nabla u$ is integrable up to the boundary; thus, if $\Delta u=0$ in~$\Omega$ then we say that $\nu\cdot \nabla u=g$ on $\partial\Omega$ in the weak sense if the above equation is satisfied for all nice test functions~$\varphi$.

The $L^2$ Neumann problem for harmonic functions was shown to be well posed in bounded Lipschitz domains in \cite{JerK81B}, and in \cite{Ver84} it was shown that the solution to the Neumann problem may be written as a single layer potential. In \cite{KenP93}, Kenig and Pipher solved the Neumann problem in starlike Lipschitz domains for operators of the form~\eqref{eqn:divergence} of second order, (that is, with $2m=2$,) with real symmetric radially independent coefficients; essentially the same argument establishes well posedness of the Neumann problem for operators with real symmetric $t$-independent coefficients in the domain above a Lipschitz graph.

In the case of second order equations (but not higher order equations), a simple change of variables allows one to pass from the the half space $\R^\dmn_+$ to a domain above a Lipschitz graph. This change of variables preserves $t$-independence. Thus, much recent work in the second order case has considered the Neumann problem in the half space
\begin{equation}\label{eqn:Neumann:L2:second}\nabla\cdot \mat A\nabla u=0 \text{ in }\R^\dmn_+,\quad \MM_{\mat A}^+ u=g,\quad \doublebar{\widetilde N_+(\nabla u)}_{L^2(\R^n)}\leq C\doublebar{g}_{L^2(\R^n)}\end{equation}
where the Neumann boundary values $\MM_{\mat A}^+ u$ of a solution $u$ to $\Div \mat A\nabla u=0$ are given by
\begin{equation}\label{eqn:Neumann:second}
\int_{\R^\dmnMinusOne} \varphi(x,0)\MM_{\mat A}^+ u(x)\,dx = \int_{\R^\dmn_+} \nabla\varphi\cdot \mat A\nabla u\quad\text{for all $\varphi\in C^\infty_0(\R^\dmn)$.}\end{equation}
If $u$ and $\mat A$ are sufficiently smooth, then $\MM_{\mat A}^+ u=-\vec e_\dmn\cdot\mat A\nabla u=\nu\cdot \mat A\nabla u$. 

If the coefficients $\mat A$ are real, symmetric, and $t$-independent, recall that well posedness of the problem~\eqref{eqn:Neumann:L2:second} was established in \cite{KenP93}. 
This problem is also known to be well posed for certain classes of complex $t$-independent coefficients.
If $\mat A$ is complex and constant, then the problem may be solved using the Fourier transform.
Well posedness of the Neumann problem in $\R^\dmn_+$ in the case where $\mat A$ is a $t$-independent matrix in ``block'' form follows from the resolution of the Kato square root problem in~$\R^n$ established in \cite{AusHLMT02}. See \cite[Remark 2.5.6]{Ken94} and \cite{AlfAAHK11,AusAH08} for a discussion of block matrices. Well posedness was extended from the case of block matrices to that of block lower triangular matrices in \cite{AusMM13}.

Furthermore, well posedness of the $L^2$ Neumann problem in the half-space was shown in \cite{AlfAAHK11,AusAM10A} to be stable under $t$-independent $L^\infty$ perturbations; in particular, if $\mat A$ is $t$-independent and close in $L^\infty$ to a constant matrix, a variable self-adjoint matrix, or a variable block or block lower triangular matrix, then the Neumann problem is well posed. 

The Neumann problem with $L^2$ boundary data is known not to be well posed for arbitrary real nonsymmetric coefficients; see the appendix to~\cite{KenR09}. (One of the two main results of \cite{KenR09} is that the Neumann problem with $L^p$ boundary data, for $p>1$ sufficiently small, is well posed in $\R^2_+$.) 

We now turn to the case of the Neumann problem with boundary data in a negative Sobolev space $\dot W^2_{-1}$, that is, the dual space to the space $\dot W^2_1$ of functions whose gradient is square integrable. 

In \cite[Proposition~4.2]{Ver05}, Verchota established well posedness of the Neumann subregularity problem for harmonic functions in Lipschitz domains
\begin{equation*}\Delta u=0 \text{ in }\Omega,\quad \MM_I^\Omega u=h,\quad \doublebar{N_\Omega u}_{L^2(\partial\Omega)}\leq C\doublebar{h}_{\dot W^2_{-1}(\partial\Omega)}.\end{equation*}
For subregular solutions, (that is, for $u$ with $N_\Omega u\in L^2(\partial\Omega)$ rather than $N_\Omega(\nabla u)\in L^2(\partial\Omega)$,) the definition of Neumann boundary values $\MM_I^\Omega u$ must be modified, as the integral on the right hand side of formula~\eqref{eqn:neumann:harmonic} need not converge for all test functions $\varphi\in C^\infty_0(\R^\dmn)$; we refer the reader to \cite[Definition 4.1]{Ver05} for a precise definition.

In \cite[Section~11]{AusM14}, it was shown that if $\mat A$ is real or if the ambient dimension $\dmn=3$, (or, more generally, if all solutions $u$ to either $\Div\mat A\nabla u=0$ or $\Div\mat A^*u=0$ satisfy the De Giorgi-Nash-Moser condition of local H\"older continuity,) then solvability of the $L^2$ Neumann problem~\eqref{eqn:Neumann:L2:second} implies solvability of the $\dot W^2_{-1}$-Neumann problem
\begin{equation*}
\nabla\cdot\mat A\nabla u=0 \text{ in }\R^\dmn_+,\quad \MM_{\mat A}^+ u=h,\quad \doublebar{\mathcal{A}_2^+(t\nabla u)}_{L^2(\R^n)}\leq C\doublebar{h}_{\dot W^2_{-1}(\R^n)}
.\end{equation*}
Note that in this case, the estimates on solutions are given not in terms of nontangential maximal functions, but in terms of area integral estimates. (See formula~\eqref{eqn:lusin} below for a definition of~$\mathcal{A}_2^+$.) In a few cases, (see also the Dirichlet problem (D2) in \cite{AlfAAHK11} and Dirichlet problems in \cite[Theorem~6.6]{AusMM13} and \cite[Section~11]{AusM14},) it has proven more convenient to solve boundary value problems posed with area integral estimates; however, it is much more common to phrase well posedness in terms of nontangential estimates.

We now turn to the higher order $L^2$ Neumann problem. Higher order Neumann boundary values may be constructed as a generalization of the second order Neumann boundary values given by formula~\eqref{eqn:Neumann:second}.

In \cite{CohG85,Ver05,Agr07,Ver10,MitM13B}, the Neumann boundary values of a solution $u$ to $Lu=0$ in~$\Omega$, with $\nabla^m u$ locally integrable up to the boundary, were essentially given by
\begin{equation}\label{eqn:Neumann:Ver05}
\sum_{j=0}^{m-1} \int_{\partial\Omega} \partial_\nu^j \varphi\,(\vec \MM\vphantom{M}_{\mat A}^\Omega u)_j\,d\sigma = \!\!\!\!\sum_{\abs\alpha=\abs\beta=m}\int_\Omega \partial^\alpha\varphi\,A_{\alpha\beta}\,\partial^\beta u\quad\text{for all $\varphi\in C^\infty_0(\R^\dmn)$.}
\end{equation}
An integration by parts argument gives a precise formula for $\vec \MM\vphantom{M}_{\mat A}^\Omega u$ in the case where $u$, $\mat A$ and $\Omega$ are sufficiently smooth; see \cite{CohG85,Ver05} in the case where $L=\Delta^2$ is the biharmonic operator, and \cite[Proposition~4.3]{MitM13A} for general constant coefficients.

The biharmonic $L^2$-Neumann problem
\begin{equation}
\label{eqn:Ver05}
\left\{\begin{aligned}\Delta^2 u&=0\text{ in }\Omega,\quad \vec \MM\vphantom{M}_{\mat A}^\Omega u = (\Lambda,f)\text{ on }\partial\Omega,\\
\doublebar{N(\nabla^2 u)}_{L^2(\partial\Omega)}&\leq C\doublebar{f}_{L^2(\partial\Omega)} + C\doublebar{\Lambda}_{\dot W^2_{-1}(\partial\Omega)}\end{aligned}\right.
\end{equation}
was shown to be well posed in \cite{CohG85} in planar $C^1$ domains, and in \cite{Ver05} in Lipschitz domains of arbitrary dimension. Even in the case of general constant coefficients the Neumann problem is not known to be well posed for $L^2$ boundary data (although see \cite{Ver10} for some general discussion of this case and \cite{Agr07,MitM13A,MitM13B,Bar17pB} for the case of boundary data in a Besov space).

In \cite{BarHM17,BarHM18,Bar17,Bar17pB}, we used the formulation 
\begin{equation}
\label{eqn:Neumann}
\int_{\R^n} \sum_{\abs\gamma=m-1}\partial^\gamma\varphi(x,0)\,(\M_{\mat A}^+u(x))_\gamma\,dx = \sum_{\abs\alpha=\abs\beta=m}\int_{\R^\dmn_+} \partial^\alpha\varphi\,A_{\alpha\beta}\,\partial^\beta u
\end{equation}
for all $\varphi\in C^\infty_0(\R^\dmn)$, provided $\nabla^m u$ is locally integrable up to the boundary. (We must modify the definition if $\nabla^m u$ satisfies weaker estimates; see \cite{BarHM17pB} for a precise definition and for further discussion.)

It is generally much more difficult to construct a formula for $\M_{\mat A}^+ u$ in this case (although see \cite[formula~(2.13)]{BarHM17} in the biharmonic case). In fact, observe that $\M_{\mat A}^+ u$ is an operator on $\{\nabla^{m-1}\varphi\big\vert_{\partial\R^\dmn_+}:\varphi\in C^\infty_0(\R^\dmn)\}$. By equality of mixed partials, the components of  $\nabla^{m-1}\varphi$ must satisfy certain compatibility conditions; thus, this is not a dense subspace of the set of arrays of distributions, and so $\M_{\mat A}^+ u$ as given by formula~\eqref{eqn:Neumann} lies in a quotient space of the space of tempered distributions.

However, we have preferred the formulation~\eqref{eqn:Neumann} to the formulation~\eqref{eqn:Neumann:Ver05} because the different components of $\M_{\mat A}^+u$ given by formula~\eqref{eqn:Neumann} generally have the same order of smoothness, and the different components of $\vec \MM\vphantom{M}_{\mat A}^\Omega u$ given by formula~\eqref{eqn:Neumann:Ver05} do not; observe the presence of both a $L^2$ norm and a $\dot W^2_{-1}$ norm in the problem~\eqref{eqn:Ver05}.

In {\cite{BarHM18}}, we established well posedness of the $L^2$ and $\dot W^2_{-1}$ Neumann problems. Specifically, suppose that $L$ is an operator of the form~\eqref{eqn:divergence} associated to coefficients~$\mat A$ that are bounded, $t$-independent in the sense of formula~\eqref{eqn:t-independent}, self-adjoint in the sense that $A_{\alpha\beta}(x)=\overline{A_{\beta\alpha}(x)}$ for all $\abs\alpha=\abs\beta=m$, and satisfy the ellipticity condition 
\begin{equation}
\label{eqn:elliptic:slices}
\re\langle \nabla^m \varphi(\,\cdot\,,t),\mat A\nabla^m\varphi(\,\cdot\,,t)\rangle_{\R^n} \geq C\doublebar{\nabla^m \varphi(\,\cdot\,,t)}_{L^2(\R^n)}^2
\end{equation}
for all $\varphi$ smooth and compactly supported in~$\R^\dmn$ and all $t\in\R$.
Then by {\cite[Theorems~1.7 and~1.11]{BarHM18}}, the Neumann problem
\begin{equation}
\label{eqn:neumann:regular:2}
\left\{\begin{gathered}\begin{aligned}
Lw&=0 \text{ in }\R^\dmn_+
,\\
\M_{\mat A}^+ w &\owns \arr g,
\end{aligned}\\
\doublebar{\mathcal{A}_2^+(t\nabla^m \partial_t w)}_{L^2(\R^n)} + \sup_{t>0}\doublebar{\nabla^m w(,\cdot\,,t)}_{L^2(\R^n)}
\leq C\doublebar{\arr g}_{L^2(\R^\dmnMinusOne)}
\end{gathered}\right.\end{equation}
and the rough Neumann problem
\begin{equation}
\label{eqn:neumann:rough:2}
\left\{\begin{aligned}
Lv&=0 \text{ in }\R^\dmn_+
,\\
\M_{\mat A}^+ v &\owns \arr h,
\\
\doublebar{\mathcal{A}_2^+(t\nabla^m v)}_{L^2(\R^n)}
&\leq C\doublebar{\arr h}_{\dot W_{-1}^2(\R^\dmnMinusOne)}
.\end{aligned}\right.\end{equation}
are well posed. That is, if $\arr g\in L^2(\R^n)$ or $\arr h\in \dot W_{-1}^2(\R^\dmnMinusOne)$, then there is a solution $w$ or~$v$ to the problem \eqref{eqn:neumann:regular:2} or~\eqref{eqn:neumann:rough:2}, and this solution is unique up to adding polynomials of degree~$m-1$.

By $\M_{\mat A}^+ w \owns \arr g$ we mean that $\arr g$ is a representative of the equivalence class $\M_{\mat A}^+ w$ given by formula~\eqref{eqn:Neumann}; that is, if we replace $\M_{\mat A}^+ w$ by $\arr g$, then formula~\eqref{eqn:Neumann} is true for all $\varphi\in C^\infty_0(\R^\dmn)$.

In the problem~\eqref{eqn:neumann:rough:2}, the gradient $\nabla^m v$ of the solution $v$ is not assumed to be locally integrable up to the boundary; it is only assumed to satisfy
\begin{equation*}\int_{\R^n}\mathcal{A}_2^+(t\nabla^m v)^2 = c_n\int_{\R^n} \int_0^\infty \abs{\nabla^m v(x,t)}^2\,t\,dt\,dx<\infty.\end{equation*} As mentioned above, in this case the notion of Neumann boundary value of formula~\eqref{eqn:Neumann} must be modified somewhat; we refer the reader to \cite[formula~(2.16)]{BarHM17pB} for the necessary generalization.

We remind the reader that it is somewhat unusual to formulate boundary value problems in terms of area integrals; of the results mentioned above, \cite{FabJR78,JerK81B,Ver84,KenP93,KenR09,Ver05} formulated solutions in terms of nontangential maximal estimates, while \cite{AusAM10A} established both square function and area integral estimates, and \cite{AlfAAHK11,AusMM13,AusM14} formulated solutions for some problems in terms of nontangential estimates and others in terms of square function estimates.

Thus, one of the two main results of this paper is the addition of nontangential estimates to the higher order Neumann problem.

\begin{thm}\label{thm:Neumann:2:N}
Suppose that $L$ is an elliptic operator of the form~\eqref{eqn:divergence} of order~$2m$ associated with coefficients $\mat A$ that satisfy $\doublebar{\mat A}_{L^\infty(\R^\dmnMinusOne)}= \Lambda<\infty$ and the ellipticity condition~\eqref{eqn:elliptic:slices}, are $t$-independent in the sense of formula~\eqref{eqn:t-independent}, and are self-adjoint, that is, satisfy $A_{\alpha\beta}(x)=\overline{A_{\beta\alpha}(x)}$.

Let $\arr g\in L^2(\R^n)$ and $\arr h\in \dot W^2_{-1}(\R^n)$, and let $w$ and $v$ be the solutions to the problems \eqref{eqn:neumann:regular:2} and~\eqref{eqn:neumann:rough:2}, respectively.

There is a constant $C$, depending only on $\Lambda$, the ambient dimension~$\dmn$, order $2m$ of the operator $L$, and the ellipticity constant $\lambda$ in the bound~\eqref{eqn:elliptic:slices}, such that
\begin{equation*}\doublebar{\widetilde N_+(\nabla^m w)}_{L^2(\R^n)}
\leq C\doublebar{\arr g}_{L^2(\R^\dmnMinusOne)}.\end{equation*}
Recall that $v$ is unique up to adding polynomials of degree~$m-1$.
There is some such additive normalization of~$v$ that satisfies
\begin{equation*}\doublebar{\widetilde N_+(\nabla^{m-1} v)}_{L^2(\R^n)}
\leq C\doublebar{\arr h}_{\dot W_{-1}^2(\R^\dmnMinusOne)}.\end{equation*}
\end{thm}

\subsection{Layer potentials}\label{sec:potentials:introduction}

The proof of Theorem~\ref{thm:Neumann:2:N} is as follows.
An examination of the proofs of \cite[Theorems~1.7 and~1.11]{BarHM18} in \cite[Section~7]{BarHM18} reveals that 
\begin{equation*}w=\D^{\mat A}\arr \varphi\quad\text{and}\quad v=\D^{\mat A}\arr f,\end{equation*}
where $\D^{\mat A}$ is the higher order double layer potential introduced in \cite{BarHM17,Bar17} (and defined in formula~\eqref{dfn:D:newton:+} below), and where $\arr \varphi=(\M_{\mat A}^+\D^{\mat A})^{-1}\arr g$ and $\arr f=(\M_{\mat A}^+\D^{\mat A})^{-1}\arr h$ lie in the Whitney spaces $\dot W\!A^2_{m-1,1}(\R^n)$ and $\dot W\!A^2_{m-1,0}(\R^n)$, respectively, used in \cite{BarHM18} (see Definition~\ref{dfn:Whitney} below). Theorem~\ref{thm:Neumann:2:N} then follows from the bounds
\begin{align*}
\doublebar{\widetilde N_+(\nabla^m \D^{\mat A}\arr \varphi)}_{L^2(\R^n)} 
&\leq C\doublebar{\arr \varphi}_{\dot W\!A^2_{m-1,1}(\R^n)}
,\\
\doublebar{\widetilde N_+(\nabla^{m-1} \D^{\mat A}\arr f)}_{L^2(\R^n)} 
&\leq C\doublebar{\arr f}_{\dot W\!A^2_{m-1,0}(\R^n)}
.\end{align*}

Thus, the double layer potential is of great interest in the theory of the higher order Neumann problem. The related single layer potential $\s^L$ is also of interest. It is often possible to use bounds on the single layer potential $\s^L$ to establish bounds on the double layer potential~$\D^{\mat A}$; see, for example, Section~\ref{sec:D:boundary} below. Bounds on the single layer potential were used to establish \cite[Theorem~1.6]{BarHM17pB}; this Fatou type result establishes existence of Dirichlet and Neumann boundary values of solutions. In the second order case, the single layer potential was used to construct solutions to the Dirichlet or Dirichlet regularity problem in \cite{FabJR78, Ver84, MayMit04A, %JerK95 I think is not actually layer potentials
AlfAAHK11,Bar13,BarM16A,HofKMP15B,HofMitMor15,HofMayMou15}; we hope that in future work, we may similarly use the single layer potential to solve the higher order Dirichlet problem.

Thus, nontangential bounds on layer potentials are of independent interest. 
The following theorem is the second main result of this paper; note that Theorem~\ref{thm:Neumann:2:N} follows from Theorem~\ref{thm:potentials}, and in particular from the bounds~\eqref{eqn:D:N:intro} and~\eqref{eqn:D:N:rough:intro}.

\begin{thm}\label{thm:potentials} Suppose that $L$ is an elliptic operator of the form~\eqref{eqn:divergence} of order~$2m$ associated with coefficients $\mat A$ that satisfy the ellipticity conditions \eqref{eqn:elliptic:bounded} and~\eqref{eqn:elliptic} and are $t$-independent in the sense of formula~\eqref{eqn:t-independent}.

Then there is an $\varepsilon>0$, depending only on the dimension $\dmn$, the order $2m$ of the operator~$L$, and the constants $\lambda$ and $\Lambda$ in the bounds \eqref{eqn:elliptic:bounded} and~\eqref{eqn:elliptic}, with the following significance.

If $2-\varepsilon<p<2+\varepsilon$, then there is a constant $C_p$ such that if 
$\arr g\in L^p(\R^n)$, $\arr h\in L^p(\R^n)$, $\arr f\in {\dot W\!A^p_{m-1,0}(\R^n)}$ and $\arr \varphi\in {\dot W\!A^p_{m-1,1}(\R^n)}$, then 
\begin{align}
\label{eqn:S:N:intro}
\doublebar{\widetilde N_+(\nabla^m \s^L\arr g)}_{L^p(\R^n)} 
&\leq C_p\doublebar{\arr g}_{L^p(\R^n)}
,&&2-\varepsilon<p<2+\varepsilon
,\\
\label{eqn:S:N:rough:intro}
\doublebar{\widetilde N_+(\nabla^{m-1} \s^L_\nabla\arr h)}_{L^p(\R^n)} 
&\leq C_p\doublebar{\arr h}_{L^p(\R^n)}
,&&2-\varepsilon<p<2+\varepsilon
,\\
\label{eqn:D:N:intro}
\doublebar{\widetilde N_+(\nabla^m \D^{\mat A}\arr \varphi)}_{L^p(\R^n)} 
&\leq C_p\doublebar{\arr \varphi}_{\dot W\!A^p_{m-1,1}(\R^n)}
,&&2-\varepsilon<p<2+\varepsilon
,\\
\label{eqn:D:N:rough:intro}
\doublebar{\widetilde N_+(\nabla^{m-1} \D^{\mat A}\arr f)}_{L^p(\R^n)} 
&\leq C_p\doublebar{\arr f}_{\dot W\!A^p_{m-1,0}(\R^n)}
,&&2-\varepsilon<p<2+\varepsilon
.\end{align}
\end{thm}
Here ${\dot W\!A^p_{m-1,1}(\R^n)}$ and ${\dot W\!A^p_{m-1,0}(\R^n)}$ are closed proper subsets of $\dot W^p_1(\R^n)$ and $L^p(\R^n)$, respectively; these subsets are the natural domain of $\D^{\mat A}$ in those spaces. See Definition~\ref{dfn:Whitney}. The modified single layer potential $\s^L_\nabla$ is the higher order analogue of the operator $\s^L\nabla$ used in \cite{AlfAAHK11,HofMitMor15,HofMayMou15}.
We will define $\s^L$ and $\s^L_\nabla$ in Section~\ref{sec:dfn:S}. Loosely speaking (see Lemma~\ref{lem:S:variant} below), we have that $\s^L_\nabla(h\arr e_\alpha)=-\s^L(\partial_j h\arr e_\zeta)$ %and $\s^L_\nabla(h\arr e_\alpha)=\partial_\dmn \s^L (h\arr e_\xi)$
whenever $1\leq j\leq \dmnMinusOne$ and $\alpha=\zeta+\vec e_j$; thus, formula~\eqref{eqn:S:N:rough:intro} gives a bound on the standard single layer potential with inputs in a negative smoothness space.

We now summarize the known bounds on higher order layer potentials. We will use these bounds to establish the nontangential estimates of Theorem~\ref{thm:potentials}. By definition (see formulas \eqref{dfn:D:newton:+} and \eqref{dfn:S} below), we have the bounds
\begin{equation*}\doublebar{\nabla^m \s^L\arr g}_{L^2(\R^\dmn)}\leq C\doublebar{\arr g}_{\dot B^{2,2}_{-1/2}(\R^n)},
\quad
\doublebar{\nabla^m \D^{\mat A}\arr f}_{L^2(\R^\dmn_\pm)}\leq C\doublebar{\arr f}_{\dot W\!A^2_{m-1,1/2}(\R^n)}
\end{equation*}
for all $\arr g\in {\dot B^{2,2}_{-1/2}(\R^n)}$ and all $\arr f\in {\dot W\!A^2_{m-1,1/2}(\R^n)}$.
The main result of \cite{BarHM17} is that the double and single layer potentials extend by density to operators that satisfy the bounds
\begin{align}
\label{eqn:D:lusin:2}
\doublebar{\mathcal{A}_2^+(t\nabla^m\partial_t\D^{\mat A}\arr \varphi)}_{L^2(\R^n)} 
&\leq C\doublebar{\arr \varphi}_{\dot W\!A^2_{m-1,1}(\R^n)}
,\\
\label{eqn:S:lusin:2}
\doublebar{\mathcal{A}_2^+(t\nabla^m\partial_t\s^L\arr g)}_{L^2(\R^n)} 
&
\leq C\doublebar{\arr g}_{L^2(\R^n)}
\end{align}
for all $\arr\varphi\in {\dot W\!A^2_{m-1,1}(\R^n)}$ and all $\arr g\in L^2(\R^n)$.

In \cite[Theorem~1.6]{BarHM17pA}, it was shown that if $\arr f\in {\dot W\!A^2_{m-1,0}(\R^n)}$, then
\begin{align}
\label{eqn:D:lusin:rough:2}
\doublebar{\mathcal{A}_2^+(t\nabla^m\D^{\mat A}\arr f)}_{L^2(\R^n)} 
&\leq C\doublebar{\arr f}_{\dot W\!A^2_{m-1,0}(\R^n)}
.\end{align}
Finally, in \cite[Theorem~1.13]{BarHM17pA}, the bound \eqref{eqn:S:lusin:2} was extended to $\arr g\in L^p$ for some $p<2$, and a bound on $\s^L_\nabla$ was established. Specifically, it was shown that there was some $\varepsilon>0$ such that, if $2-\varepsilon<p\leq 2$, then there is a $C_p$ such that for all $\arr g\in L^p(\R^n)$ and all $\arr h\in L^p(\R^n)$,
\begin{align}
\label{eqn:S:lusin:-}
\doublebar{\mathcal{A}_2^+(t\nabla^m\partial_t\s^L\arr g)}_{L^p(\R^n)} 
&
\leq C_p\doublebar{\arr g}_{L^p(\R^n)}, && 2-\varepsilon<p\leq 2
,\\
\label{eqn:S:lusin:rough:-}
\doublebar{\mathcal{A}_2^+(t\nabla^m\s^L_\nabla\arr h)}_{L^p(\R^n)} 
&\leq C_p\doublebar{\arr h}_{L^p(\R^n)}, && 2-\varepsilon<p\leq 2.
\end{align}

Observe that these known bounds all involve inputs in $L^p$ for $p=2$ or $p<2$.
In the course of proving Theorem~\ref{thm:potentials}, we will also establish the following area integral estimates for inputs in $L^p$ with $p>2$.
\begin{thm}\label{thm:lusin:intro} Let $L$ and $\mat A$be as in Theorem~\ref{thm:potentials}. Then there is an $\varepsilon>0$ such that, if $2<p<2+\varepsilon$, then there is a constant $C_p$ such that if 
$\arr g\in L^p(\R^n)$, $\arr h\in L^p(\R^n)$, $\arr f\in {\dot W\!A^p_{m-1,0}(\R^n)}$ and $\arr \varphi\in {\dot W\!A^p_{m-1,1}(\R^n)}$, then 
\begin{align}
\label{eqn:S:lusin:intro}
\doublebar{\mathcal{A}_2^+(t\nabla^m\partial_t\s^L\arr g)}_{L^p(\R^n)} 
&\leq C_p\doublebar{\arr g}_{L^p(\R^n)}
,&&2<p<2+\varepsilon
,\\
\label{eqn:S:lusin:rough:intro}
\doublebar{\mathcal{A}_2^+(t\nabla^m\s^L_\nabla\arr h)}_{L^p(\R^n)} 
&\leq C_p\doublebar{\arr h}_{L^p(\R^n)}
,&&2<p<2+\varepsilon
,\\
\label{eqn:D:lusin:intro}
\doublebar{\mathcal{A}_2^+(t\nabla^m\partial_t\D^{\mat A}\arr \varphi)}_{L^p(\R^n)} 
&\leq C_p\doublebar{\arr \varphi}_{\dot W\!A^p_{m-1,1}(\R^n)}
,&&2<p<2+\varepsilon
,\\
\label{eqn:D:lusin:rough:intro}
\doublebar{\mathcal{A}_2^+(t\nabla^m\D^{\mat A}\arr f)}_{L^p(\R^n)} 
&\leq C_p\doublebar{\arr f}_{\dot W\!A^p_{m-1,0}(\R^n)}
,&&2<p<2+\varepsilon
.\end{align}
\end{thm}

In a forthcoming paper, we will show that the bounds \eqref{eqn:D:lusin:intro} and~\eqref{eqn:D:lusin:rough:intro} extend to the case $2-\varepsilon<p<2$ (that is, we will establish the analogues to the bounds \eqref{eqn:S:lusin:-} and~\eqref{eqn:S:lusin:rough:-} for the double layer potential).

The theory of layer potentials for higher order operators is still relatively new, and thus to our knowledge the above represents a nearly comprehensive survey of bounds on layer potentials for operators of order $2m\geq 4$ with $t$-independent coefficients in the half-space. (Some additional bounds on $\partial_t^k\s^L$ and $\partial_t^k\s^L_\nabla$, for $k$ large enough, were established in \cite{BarHM17pA} and used in \cite{BarHM17pB}.) 

However, the theory of variable coefficient higher order operators builds on the extensive and well developed theory of second order operators (that is, the case $2m=2$) and the reasonably well developed theory of constant coefficient higher order operators. 

In the special case of constant coefficient operators (in particular, in the theory of harmonic functions) in Lipschitz domains, boundedness of layer potentials follows from boundedness of the Cauchy integral on Lipschitz curves; the Cauchy integral was famously bounded by Coifmann, McIntosh and Meyer in \cite{CoiMM82}. %In the higher order case, see, for example, \cite[Proposition~2.63]{MitM13A} or \cite[formula~(3.2)]{PipV92} \cite[Lemma~1.3]{DahKV86} \cite{CohG83,CohG85}
Layer potentials for the Laplacian $-\Delta$ were used in \cite{FabJR78,Ver84,DahK87,PipV92,FabMM98,Zan00}, for the biharmonic operator $\Delta^2$ in \cite{CohG83,CohG85,Ver05,MitM13B}, and for general higher order constant coefficient equations in \cite{Agm57,MitM13A}.

In the case of second order operators with variable $t$-independent coefficients, bounds on layer potentials were established in \cite{KenR09,Rul07,Bar13} in two dimensions for real (or almost real) coefficients. 

Turning to higher dimensions, in \cite{AlfAAHK11} the $p=2$ cases of the bounds~\eqref{eqn:S:N:intro} and \eqref{eqn:S:lusin:intro}  were established for operators of order $2m=2$ with real symmetric $t$-independent coefficients, and a stability result under $L^\infty$ perturbation was established. (The authors also established numerous more specialized  bounds on layer potentials.)
In \cite{Ros13}, R\`osen showed that layer potentials coincide with certain operators appearing in the theory of semigroups investigated in \cite{AusAH08,AusAM10A,AusA11}. In particular, numerous bounds in the $p=2$ case follow. % including the bound~\eqref{eqn:S:lusin:intro}.

The theory of boundary value problems and layer potentials for second order operators was subsequently investigated extensively in the case where $L=-\Div \mat A\nabla$ and $L^*=-\Div \mat A^* \nabla $ satisfy the De Giorgi-Nash-Moser condition; this condition is always satisfied if the ambient dimension $\dmn$ satisfies $\dmn=2$, if $\dmn=3$ and $\mat A$ is $t$-independent, or if $2m=2$ and $\mat A$ is real valued. In these cases, it is often possible to establish at least some bounds on layer potentials using the theory of Calder\'on-Zygmund operators with kernels that satisfy Littlewood-Paley  estimates. See, for example, \cite[Section~4]{KenR09} or \cite[Section~8]{AlfAAHK11}.

In particular, the $p=2=2m$ case of the bounds \eqref{eqn:S:lusin:intro} and \eqref{eqn:S:N:intro} on the single layer potential $\s^L$ were established in \cite{GraH16} directly using $Tb$ theorems, without recourse to the theory of semigroups used in \cite{Ros13}. Building on this bound, the $2m=2$ case of all eight of the bounds \eqref{eqn:S:N:intro}--\eqref{eqn:D:N:rough:intro} and \eqref{eqn:S:lusin:intro}--\eqref{eqn:D:lusin:rough:intro} may be found in \cite{AusM14,HofMitMor15,HofMayMou15,HofKMP15B} for a fairly broad range of~$p$.

Finally, returning to the theory of semigroups, if $2m=2$ then these eight bounds were established in \cite[Theorem~12.7]{AusS16} without assuming  the De Giorgi-Nash-Moser condition, that is, using only boundedness, ellipticity and $t$-independence of the coefficients.

\subsection{Outline}

The organization of this paper is as follows.

In Section~\ref{sec:dfn} we will define our terminology. In Section~\ref{sec:preliminaries} we will recall some known estimates on solutions that we will use extensively throughout the paper, and will prove a few lemmas involving the nontangential and area integral estimates of a general solution $u$ to $Lu=0$. In particular, given the known area integral estimates \eqref{eqn:D:lusin:2}--\eqref{eqn:S:lusin:rough:-} and the nontangential estimates of Theorem~\ref{thm:potentials}, most of the work involved in establishing the area integral estimates of Theorem~\ref{thm:lusin:intro} is contained in Lemma~\ref{lem:lusin:+}.

Section~\ref{sec:S} will be devoted to the nontangential bounds \eqref{eqn:S:N:intro} and~\eqref{eqn:S:N:rough:intro} on the single layer potential (and modified single layer potential). 
Section~\ref{sec:D} will mainly be concerned with establishing the nontangential estimate~\eqref{eqn:D:N:intro} on the double layer potential; the bound \eqref{eqn:D:N:rough:intro} (and the area integral bounds \eqref{eqn:D:lusin:intro} and~\eqref{eqn:D:lusin:rough:intro}) follow fairly quickly once this bound is established. We remark that we will establish area integral bounds \eqref{eqn:S:lusin:intro} and~\eqref{eqn:S:lusin:rough:intro} on the single layer potential in Section~\ref{sec:S} using the nontangential bounds\eqref{eqn:S:N:intro} and~\eqref{eqn:S:N:rough:intro}, and will use these nontangential bounds in order to establish preliminary estimates on the double layer potential.

\subsection*{Acknowledgements}
We would like to thank 
the American Institute of Mathematics for hosting the SQuaRE workshop on ``Singular integral operators and solvability of boundary problems for elliptic equations with rough coefficients,'' 
the Mathematical Sciences Research Institute for hosting a Program on Harmonic Analysis, 
the Instituto de Ciencias Matem\'aticas for hosting a Research Term on ``Real Harmonic Analysis and Its Applications to Partial Differential Equations and Geometric Measure Theory'', 
and 
the IAS/Park City Mathematics Institute for hosting a Summer Session with a research topic of Harmonic Analysis,
at which many of the results and techniques of this paper were discussed.

\section{Definitions}\label{sec:dfn}

In this section, we will provide precise definitions of the notation and concepts used throughout this paper. 

We will work with elliptic operators~$L$ of order~$2m$ in the divergence form \eqref{eqn:divergence} acting on functions defined on~$\R^\dmn$.

As usual, we let $B(X,r)$ denote the ball in $\R^n$ of radius $r$ and center $X$. We let $\R^\dmn_+$ and $\R^\dmn_-$ denote the upper and lower half-spaces $\R^n\times (0,\infty)$ and $\R^n\times(-\infty,0)$; we will identify $\R^n$ with $\partial\R^\dmn_\pm$.
If $Q\subset\R^n$ or $Q\subset\R^\dmn$ is a cube, we will let $\ell(Q)$ be its side-length, and we let $cQ$ be the concentric cube of side-length $c\ell(Q)$. If $E$ is a set of finite measure, we let 
\begin{equation*}\fint_E f(x)\,dx=\frac{1}{\abs{E}}\int_E f(x)\,dx.\end{equation*}
If $f\in L^1_{loc}(\R^n)$, then the Hardy-Littlewood maximal function $\mathcal{M}f$ is given by
\begin{equation*}\mathcal{M}f(x)=\sup_{Q\owns x} \fint_Q \abs{f}\end{equation*}
where the supremum is over all cubes $Q\subset\R^n$ with $x\in Q$.

If $E$ is a measurable set, we will let $\1_E$ denote the characteristic function of~$E$; that is, $\1_E(x)=1$ if $x\in E$ and $\1_E(x)=0$ if $x\notin E$. We will use $\1_\pm$ as a shorthand for $\1_{\R^\dmn_\pm}$.

\subsection{Multiindices and arrays of functions}

We will routinely work with multiindices in~$(\N_0)^\dmn$. (We will occasionally work with multiindices in $(\N_0)^\dmnMinusOne$.) Here $\N_0$ denotes the nonnegative integers. If $\zeta=(\zeta_1,\zeta_2,\dots,\zeta_\dmn)$ is a multiindex, then we define $\abs{\zeta}$ and $\partial^\zeta$ in the usual ways, as $\abs{\zeta}=\zeta_1+\zeta_2+\dots+\zeta_\dmn$ and $\partial^\zeta=\partial_{x_1}^{\zeta_1}\partial_{x_2}^{\zeta_2} \cdots\partial_{x_\dmn}^{\zeta_\dmn}$. 

We will routinely deal with arrays $\arr F=\begin{pmatrix}F_{\zeta}\end{pmatrix}$ of numbers or functions indexed by multiindices~$\zeta$ with $\abs{\zeta}=k$ for some~$k\geq 0$.
In particular, if $\varphi$ is a function with weak derivatives of order up to~$k$, then we view $\nabla^k\varphi$ as such an array.

The inner product of two such arrays of numbers $\arr F$ and $\arr G$ is given by
\begin{equation*}\bigl\langle \arr F,\arr G\bigr\rangle =
\sum_{\abs{\zeta}=k}
\overline{F_{\zeta}}\, G_{\zeta}.\end{equation*}
If $\arr F$ and $\arr G$ are two arrays of functions defined in a set $\Omega$ in Euclidean space, then the inner product of $\arr F$ and $\arr G$ is given by
\begin{equation*}\bigl\langle \arr F,\arr G\bigr\rangle_\Omega =
\int_\Omega \langle \arr F(X),\arr G(X)\rangle\,dX =
\sum_{\abs{\zeta}=k}
\int_{\Omega} \overline{F_{\zeta}(X)}\, G_{\zeta}(X)\,dX.\end{equation*}

We let $\vec e_j$ be the unit vector in $\R^\dmn$ in the $j$th direction; notice that $\vec e_j$ is a multiindex with $\abs{\vec e_j}=1$. We let $\arr e_{\zeta}$ be the ``unit array'' corresponding to the multiindex~$\zeta$; thus, $\langle \arr e_{\zeta},\arr F\rangle = F_{\zeta}$.

We will let $\nabla_\pureH$ denote either the gradient in~$\R^n$, or the $n$ horizontal components of the full gradient~$\nabla$ in $\R^\dmn$. (Because we identify $\R^n$ with $\partial\R^\dmn_\pm\subset\R^\dmn$, the two uses are equivalent.) If $\zeta$ is a multiindex with $\zeta_\dmn=0$, we will occasionally use the terminology $\partial_\pureH^\zeta$ to emphasize that the derivatives are taken purely in the horizontal directions.

\subsection{Elliptic differential operators}

Let $\mat A = \begin{pmatrix} A_{\alpha\beta} \end{pmatrix}$ be a matrix of measurable coefficients defined on $\R^\dmn$, indexed by multtiindices $\alpha$, $\beta$ with $\abs{\alpha}=\abs{\beta}=m$. If $\arr F$ is an array indexed by multiindices of length~$m$, then $\mat A\arr F$ is the array given by
\begin{equation*}(\mat A\arr F)_{\alpha} =
\sum_{\abs{\beta}=m}
A_{\alpha\beta} F_{\beta}.\end{equation*}

We let $L$ be the $2m$th-order divergence form operator associated with~$\mat A$. That is, we say that 
\begin{equation}
\label{eqn:weak}
Lu=0 \text{ in }\Omega \text{ in the weak sense if } \langle \nabla^m \varphi, \mat A\nabla^m u\rangle_\Omega=0 \text{ for all } \varphi\in C^\infty_0(\Omega).
\end{equation}

Throughout we consider coefficients that satisfy the bound
\begin{align}
\label{eqn:elliptic:bounded}
\doublebar{\mat A}_{L^\infty(\R^\dmn)}
&\leq
\Lambda
\end{align}
and the 
G\r{a}rding inequality
\begin{align}
\label{eqn:elliptic}
\re {\bigl\langle\nabla^m \varphi,\mat A\nabla^m \varphi\bigr\rangle_{\R^\dmn}}
&\geq
\lambda\doublebar{\nabla^m\varphi}_{L^2(\R^\dmn)}^2
\quad\text{for all $\varphi\in\dot W^2_m(\R^\dmn)$}
\end{align}
for some $\Lambda>\lambda>0$. (The stronger G\r{a}rding inequality~\eqref{eqn:elliptic:slices} will not be used in the proof of Theorem~\ref{thm:potentials} or~\ref{thm:lusin:intro}; it was used only in the statement and proof of Theorem~\ref{thm:Neumann:2:N}.)

The numbers $C$ and $\varepsilon$ denote constants whose value may change from line to line, but which are always positive and depend only on the dimension~$\dmn$, the order $2m$ of any relevant operators, and the numbers $\lambda$ and $\Lambda$ in the bounds~\eqref{eqn:elliptic} (or~\eqref{eqn:elliptic:slices}) and~\eqref{eqn:elliptic:bounded}.

\subsection{Function spaces and boundary data}

Let $\Omega\subseteq\R^n$ or $\Omega\subseteq\R^\dmn$ be a measurable set in Euclidean space. We let $C^\infty_0(\Omega)$ be the space of all smooth functions that are compactly supported in~$\Omega$. We let $L^p(\Omega)$ denote the usual Lebesgue space with respect to Lebesgue measure with norm given by
\begin{equation*}\doublebar{f}_{L^p(\Omega)}=\biggl(\int_\Omega \abs{f(x)}^p\,dx\biggr)^{1/p}.\end{equation*}

If $\Omega$ is a connected open set, then we let the homogeneous Sobolev space $\dot W^p_k(\Omega)$ be the space of equivalence classes of functions $u$ that are locally integrable in~$\Omega$ and have weak derivatives in $\Omega$ of order up to~$k$ in the distributional sense, and whose $k$th gradient $\nabla^k u$ lies in $L^p(\Omega)$. Two functions are equivalent if their difference is a polynomial of order~$k-1$.
We impose the norm 
\begin{equation*}\doublebar{u}_{\dot W^p_k(\Omega)}=\doublebar{\nabla^k u}_{L^p(\Omega)}.\end{equation*}
Then $u$ is equal to a polynomial of order $k-1$ (and thus equivalent to zero) if and only if its $\dot W^p_k(\Omega)$-norm is zero. 
We say that $u\in L^p_{loc}(\Omega)$ or $u\in\dot W^p_{k,loc}(\Omega)$ if $u\in L^p(U)$ or $u\in\dot W^p_{k}(U)$ for any bounded open set $U$ with $\overline U\subset\Omega$.

We will need a number of more specialized norms on functions.
The modified nontangential operator $\widetilde N_+$ was introduced in \cite{KenP93} and (in the half space) is given by
\begin{equation}
\label{dfn:NTM:modified}
\widetilde N_+ H(x) = \sup
\biggl\{\biggl(\fint_{B((y,s),{s}/2)} \abs{H}^2\biggr)^{1/2}:
s>0,\>y\in\R^n,\>
\abs{x-y}< s
\biggr\}
.\end{equation}
We will also use a two-sided nontangential maximal function, which we define as
\begin{equation}
\label{dfn:NTM:twosides}
\widetilde N_* H(x) = \sup
\biggl\{\biggl(\fint_{B((y,s),\abs{s}/2)} \abs{H}^2\biggr)^{1/2}:
s\in\R,\>y\in\R^n,\>%s\neq0,\>
\abs{x-y}< \abs{s}
\biggr\}
.\end{equation}
Finally, we will use the Lusin area integral operator $\mathcal{A}_2^+$ given by
\begin{align}\label{eqn:lusin}
\mathcal{A}_2^+ H(x) &= \biggl( \int_0^\infty \int_{\abs{x-y}<t} \abs{H(y, t)}^2\frac{dy\,dt}{t^\dmn}\biggr)^{1/2}
.\end{align}

\subsubsection{Boundary values and boundary function spaces}

Following \cite{BarHM17pB}, we define the boundary values $\Trace^\pm u$ of a function $u$ defined in $\R^\dmn_\pm$ by 
\begin{equation}
\Trace^\pm  u
=\arr f \quad\text{if}\quad
\lim_{t\to 0^\pm} \doublebar{\partial^\gamma u(\,\cdot\,,t)-f}_{L^1(K)}=0
\end{equation}
for all compact sets $K\subset\R^n$. We define
\begin{equation}
\label{eqn:Dirichlet}
\Tr_j^\pm u= \Trace^\pm \nabla^j u.\end{equation}
We remark that if $\nabla u$ is locally integrable up to the boundary, then $\Trace^\pm u$ exists, and furthermore $\Trace^\pm u$ coincides with the traditional trace in the sense of Sobolev spaces.

We are interested in functions with boundary data in Lebesgue or Sobolev spaces. However, observe that if $j\geq 1$, then the components of $\Tr_j^\pm u$ are derivatives of a common function and so must satisfy certain compatibility conditions. We thus define the following Whitney-Lebesgue, Whitney-Sobolev and Whitney-Besov spaces of arrays that satisfy these conditions.

\begin{defn} \label{dfn:Whitney}
Let 
\begin{equation*}\mathfrak{D}=\{\Tr_{m-1}\varphi:\varphi\text{ smooth and compactly supported in $\R^\dmn$}\}.\end{equation*}

If $1\leq p<\infty$, then we let $\dot W\!A^p_{m-1,0}(\R^n)$ be the completion of the set $\mathfrak{D}$ under the $L^p$ norm. 
We let $\dot W\!A^p_{m-1,1}(\R^n)$ be the completion of $\mathfrak{D}$ under the $\dot W^p_1(\R^n)$ norm, that is, under the norm $\doublebar{\arr f}_{\dot W\!A^p_{m-1,1}(\R^n)}=\doublebar{\nabla_\pureH \arr f}_{L^p(\R^n)}$. 

Finally, we let $\dot W\!A^2_{m-1,1/2}(\R^n)$ be the completion of $\mathfrak{D}$ under the norm in the Besov space $\dot B^{2,2}_{1/2}(\R^n)$; this norm may be written as
\begin{equation*}\doublebar{\arr f}_{\dot B^{2,2}_{1/2}(\R^n)} = \doublebar{\arr f}_{\dot W\!A^2_{m-1,1/2}(\R^n)} = \biggl(\sum_{\abs\gamma=m-1} \int_{\R^n}\abs{\widehat {f_\gamma}(\xi)}^2\abs{\xi}\,d\xi\biggr)^{1/2}.\end{equation*}
\end{defn}

It is widely known that $\arr f\in \dot W\!A^2_{m-1,1/2}(\R^n)$ if and only if $\arr f=\Tr_{m-1}^+ F$ for some $F$ with $\nabla^m F\in L^2(\R^\dmn_+)$.
%See \cite[Lemma~2.9]{BarHM17pB} and the following remarks.

Recall that Theorem~\ref{thm:Neumann:2:N} is concerned with Neumann boundary values $\M_{\mat A}^+u$ of solutions $u$ to $Lu=0$. However, as discussed at the beginning of Section~\ref{sec:potentials:introduction}, Theorem~\ref{thm:Neumann:2:N} follows from Theorem~\ref{thm:potentials} and the proof of \cite[Theorems 1.7 and~1.11]{BarHM18}, and thus we will not use any particular properties of $\M_{\mat A}^+u$ in the proof; in this case we refer the reader to \cite[Section~2.3.2]{BarHM18} for a definition of $\M_{\mat A}^+$.

In the proof of Lemma~\ref{lem:boundary:D} below we will use some properties of $\M_{\mat A}^-$ from \cite{Bar17} and \cite{BarHM17pB}. In these cases we refer the reader to \cite{BarHM17pB} for a definition of $\M_{\mat A}^\pm u$; we remark only that if $u\in \dot W^2_m(\R^\dmn_\pm)$ and $Lu=0$ in $\R^\dmn_\pm$, then the definitions of $\M_{\mat A}^\pm u$ used in \cite{Bar17} and in \cite{BarHM17pB} coincide.

\subsection{The double layer potential and the Newton potential}

In this section we define the double layer potential mentioned in Theorem~\ref{thm:potentials}.

We begin by defining the related Newton potential.
For any $\arr H\in L^2(\R^\dmn)$, by the Lax-Milgram lemma there is a unique function $\Pi^L\arr H$ in $\dot W^2_m(\R^\dmn)$ that satisfies
\begin{equation}\label{eqn:newton}
\langle \nabla^m\varphi, \mat A\nabla^m \Pi^L\arr H\rangle_{\R^\dmn}=\langle \nabla^m\varphi, \arr H\rangle_{\R^\dmn}\end{equation}
for all $\varphi\in \dot W^2_m(\R^\dmn)$. We refer to the operator $\Pi^L$ as the Newton potential.

Now, suppose that $\arr f\in \dot W\!A^2_{m-1,1/2}(\R^n)$. Recall that $\arr f=\Tr_{m-1}^+ F$ for some $F\in \dot W^2_m(\R^\dmn_+)$.
We define
\begin{equation}
\label{dfn:D:newton:+}
\D^{\mat A}\arr f = -\1_+ F + \Pi^L(\1_+ \mat A\nabla^m F)
.\end{equation}
This operator is well-defined, that is, does not depend on the choice of~$F$. See \cite[Lemma~4.2]{Bar17} or \cite[section~2.4]{BarHM17}. Furthermore, it is antisymmetric about exchange of $\R^\dmn_+$ and $\R^\dmn_-$; that is, if $\Tr_{m-1}^-F=\arr f$ and $F\in \dot W^2_m(\R^\dmn_-)$, then
\begin{equation}
\label{dfn:D:newton:-}
\D^{\mat A}\arr f = 
-\Pi^L(\1_-\mat A\nabla^m F) + \1_- F
.\end{equation}
See \cite[formula~(2.27)]{BarHM17} or \cite[formula~(4.8)]{Bar17}. 

\subsection{The single layer potential}
\label{sec:dfn:S}

Let $\arr g$ be a bounded linear operator on the space $\dot W\!A^2_{m-1,1/2}(\R^n)$. Observe that $\arr g$ extends to an operator 
on $\dot B^{2,2}_{1/2}(\R^n)$, and so $\arr g\in{(\dot W\!A^2_{m-1,1/2}(\R^n))^*}$ if and only if there is a representative of $\arr g$ that lies in $\dot B^{2,2}_{-1/2}(\R^n)$, that is, that satisfies
\begin{equation*}
\biggl(\sum_{\abs\gamma=m-1} \int_{\R^n} \frac{1}{\abs\xi}\abs{\widehat{g_\gamma}(\xi)}^2\,d\xi\biggr)^{1/2} =\doublebar{\arr g}_{\dot B^{2,2}_{-1/2}(\R^n)}<\infty.\end{equation*}
The operator $F\mapsto \langle \Tr_{m-1} F,\arr g\rangle_{\R^n}$ is a bounded linear operator on $\dot W^2_m(\R^\dmn)$, and so by the Lax-Milgram lemma there is a unique function $\s^L\arr g\in \dot W^2_m(\R^\dmn)$ that satisfies
\begin{align}
\label{dfn:S}
\langle \nabla^m\varphi, \mat A\nabla^m\s^L\arr g\rangle_{\R^\dmn}&=\langle \Tr_{m-1}\varphi,\arr g\rangle_{\R^n}
&&\text{for all }\varphi\in \dot W^2_m(\R^\dmn)
.\end{align}
See \cite{Bar17}. We remark that this definition coincides with the definition of $\s^L\arr g$ given in \cite{BarHM17,BarHM17pA}. This defines $\s^L$ as a bounded operator $\dot B^{2,2}_{-1/2}(\R^n)\mapsto \dot W^2_m(\R^\dmn)$; by \cite[formula~(4.3)]{BarHM18} (see Section~\ref{sec:S:L2} below) we have that $\s^L$ extends by density to an operator that is bounded $L^2(\R^n)\mapsto \dot W^2_m(\R^n\times(a,b))$ for any numbers $-\infty<a<b<\infty$.

As observed in \cite[formula~(2.23)]{BarHM17pA}, if $\arr g\in \dot B^{2,2}_{-1/2}(\R^n)$ and if $\abs\alpha=m$, then for almost every $(x,t)\in\R^\dmn$ we have that
\begin{align}
\label{eqn:S:fundamental}
\partial^\alpha\s^L\arr g(x,t)
&=
\sum_{\abs{\zeta}=m-1}
\int_{\R^n} \partial_{x,t}^\alpha \partial_{y,s}^\zeta E^L(x,t,y,0)\,g_{\zeta}(y) \,dy
\end{align}
where $E^L$ is the fundamental solution for the operator $L$ constructed in~\cite{Bar16}. 
By the bound \cite[formula~(63)]{Bar16} (reproduced as formula~\eqref{eqn:fundamental:far} below), for almost every $(x,t)\in\R^\dmn_\pm$ we have that $\partial_{x,t}^\alpha E^L(x,t,\,\cdot\,,\,\cdot\,)\in \dot W^2_m(\R^\dmn_\mp)$ and so $v(y)=\partial_{x,t}^\alpha \partial_{y,s}^\zeta E^L(x,t,y,0)$ lies in $\dot B^{2,2}_{1/2}(\R^n)$. Thus, the right hand side converges provided $g_{\zeta}\in \dot B^{2,2}_{-1/2}(\R^n)$.

As in \cite[formula~(2.27)]{BarHM17pA}, if $\abs\gamma=m-1$, then we define 
\begin{align}
\label{eqn:S:variant}
\partial^\gamma \s^L_\nabla\arr h(x,t)
&=\sum_{\abs\beta=m}\int_{\R^n} \partial_{x,t}^\gamma \partial_{y,s}^\beta E^L(x,t,y,0)\,h_\beta(y)\,dy
.\end{align}
We will see (Lemma~\ref{lem:S:variant} below) that if $\arr h\in L^2(\R^n)$, then the integral converges absolutely for almost every $(x,t)\in\R^\dmn$, and the functions $\partial^\gamma \s^L_\nabla\arr h$ given by formula~\eqref{eqn:S:variant} are indeed derivatives of a common $\dot W^2_{m-1,loc}(\R^n)$-function that we may call $\s^L_\nabla\arr h$.

\section{Preliminaries}
\label{sec:preliminaries}

In Section~\ref{sec:S}, we will establish the bounds \eqref{eqn:S:N:intro}, \eqref{eqn:S:N:rough:intro}, \eqref{eqn:S:lusin:intro}, and~\eqref{eqn:S:lusin:rough:intro} on the single layer potential.
In Section~\ref{sec:D}, we will establish the bounds \eqref{eqn:D:N:intro}, \eqref{eqn:D:N:rough:intro}, \eqref{eqn:D:lusin:intro}, and~\eqref{eqn:D:lusin:rough:intro} on the double layer potential. In this section, we will collect some known results and establish some preliminary estimates that will be of use in both Section~\ref{sec:S} and Section~\ref{sec:D}.

\subsection{Regularity results}

We begin by recalling some known regularity results for solutions to elliptic differential equations.

The following lemma is the higher order analogue of the Caccioppoli inequality. It was proven in full generality in \cite{Bar16} and some important preliminary versions were established in \cite{Cam80,AusQ00}.
\begin{lem}[The Caccioppoli inequality]\label{lem:Caccioppoli}
Let $L$ be an operator of the form~\eqref{eqn:divergence} of order~$2m$ associated to coefficients~$\mat A$ that satisfy the bounds \eqref{eqn:elliptic:bounded} and~\eqref{eqn:elliptic}. Let $ u\in \dot W^2_m(B(X_0,2r))$ with $L u=0$ in $B(X_0,2r)$.

Then we have the bound
\begin{equation*}
\fint_{B(X,r)} \abs{\nabla^j  u(x,s)}^2\,dx\,ds
\leq \frac{C}{r^2}\fint_{B(X,2r)} \abs{\nabla^{j-1}  u(x,s)}^2\,dx\,ds
\end{equation*}
for any $j$ with $1\leq j\leq m$.
\end{lem}

We next state the higher order generalization of Meyers's reverse H\"older inequality for gradients. The $k=0$ case of the bound \eqref{eqn:Meyers:lower} was established in \cite{Cam80,AusQ00}. The $k\geq 1$ case was established in \cite{Bar16} and is a relatively straightforward consequence of the $k=0$ case and the Gagliardo-Nirenberg-Sobolev inequality.

\begin{thm}
\label{thm:Meyers} 
Let $L$ be an operator of the form~\eqref{eqn:divergence} of order~$2m$ associated to coefficients~$\mat A$ that satisfy the bounds \eqref{eqn:elliptic:bounded} and~\eqref{eqn:elliptic}. 
Let $X_0\in\R^\dmn$ and let $r>0$. Suppose that $ u\in \dot W^2_m(B(X_0,2r))$ with $L  u=0$ in $B(X_0,2r)$. 

If $k$ is an integer with $0\leq k\leq m$ and $2k<\dmn$, then there is a number $p_k^+=p_{L,k}^+>2\pdmn/(\dmn-2k)$, depending only on the standard constants, such that if $0<p<q<p_k^+$, then
\begin{align}\label{eqn:Meyers:lower}
\biggl(\int_{B(X_0,r)
}\abs{\nabla^{m-k}   u}^{q}\biggr)^{1/{q}}
&\leq 
	\frac{C_{p,q}}{r^{\pdmn/p-\pdmn/{q}}}
	\biggl(\int_{B(X_0,2r)}\abs{\nabla^{m-k}   u}^{p}\biggr)^{1/p} 
\end{align}
for some constant $C_{p,q}$ depending only on $p$, $q$ and the standard constants.

If $0\leq m-k \leq m-\pdmn/2$, then $\nabla^{m-k}  u$ is H\"older continuous and satisfies the bound
\begin{align}\label{eqn:Meyers:lowest}
\sup_{B(X_0,r)} \abs{\nabla^{m-k}   u} 
&\leq 
\frac{C_p}{r^{\pdmn/p}}
\biggl(\int_{B(X_0,2r)}\abs{\nabla^{m-k}   u}^p\biggr)^{1/p} 
\end{align}
for all $0<p\leq \infty$.
\end{thm}

Finally, if $\mat A$ is $t$-independent then solutions to $Lu=0$ have additional regularity. The following lemma was proven in the case $m = 1$ in \cite[Proposition 2.1]{AlfAAHK11} and generalized to the case $m \geq 2$ in \cite[Lemma 3.20]{BarHM17pA}.
\begin{lem}\label{lem:slices}
Let $L$ be an operator of the form~\eqref{eqn:divergence} of order~$2m$ associated to $t$-independent coefficients~$\mat A$ that satisfy the bounds \eqref{eqn:elliptic:bounded} and~\eqref{eqn:elliptic}.
Let $t\in\R$ be a constant, and let $Q\subset\R^n$ be a cube.

If $Lu=0$ in $2Q\times(t-\ell(Q),t+\ell(Q))$, then 
\begin{equation*}\int_Q \abs{\nabla^j \partial_t^k u(x,t)}^p\,dx \leq \frac{C_p}{\ell(Q)} 
\int_{2Q}\int_{t-\ell(Q)}^{t+\ell(Q)} \abs{\nabla^j \partial_s^k u(x,s)}^p\,ds\,dx\end{equation*}
for any $0\leq j\leq m$, any $0< p < p_{m-j}^+$, and any integer $k\geq 0$, where $p^+_{m-j}$ is as in Theorem~\ref{thm:Meyers}.
\end{lem}

\subsection{The fundamental solution}

Recall from formula~\eqref{eqn:S:fundamental} that the single layer potential, originally constructed via the Lax-Milgram lemma, has an explicit representation as an integral operator involving the fundamental solution. We will often make use of this representation; thus, we now state the following result of \cite{Bar16} concerning the fundamental solution for higher order operators.

\begin{thm}[{\cite[Theorem~62 and Lemma~69]{Bar16}}]\label{thm:fundamental}
Let $L$ be an operator of the form~\eqref{eqn:divergence} of order~$2m$ associated to coefficients~$\mat A$ that satisfy the bounds \eqref{eqn:elliptic:bounded} and~\eqref{eqn:elliptic}.
Then there exists a function $E^L(X,Y)$ with the following properties.

Let $s=0$ or $s=1$ and let $q=0$ or $q=1$. 
There is some $\varepsilon>0$ such that if $X_0$, $Y_0\in\R^\dmn$,  if $0<r<R<\abs{X_0-Y_0}/3$, and if either $q=0$ or $\dmn\geq 3$, then
\begin{equation}
\label{eqn:fundamental:far}
\int_{B(Y_0,r)}\int_{B(X_0,R)} \abs{\nabla^{m-s}_X \nabla^{m-q}_Y E^L(X,Y)}^2\,dX\,dY 
\leq C r^{2q}R^{2s} \biggl(\frac{r}{R}\biggr)^\varepsilon
.\end{equation}
If $2q=2=\dmn$ then we instead have the bound
\begin{equation}
\label{eqn:fundamental:far:low}
\int_{B(Y_0,r)}\int_{B(X_0,R)} \abs{\nabla^{m-s}_X \nabla^{m-1}_Y E^L(X,Y)}^2\,dX\,dY \leq C_\delta\, r^{2} R^{2s} \biggl(\frac{R}{r}\biggr)^\delta
\end{equation}
for all $\delta>0$ and some constant $C_\delta$ depending on~$\delta$.

We have the symmetry property % E^L is a solution in both variables
\begin{equation}
\label{eqn:fundamental:symmetric}
\partial_X^\zeta\partial_Y^\xi E^L(X,Y) = \overline{\partial_X^\zeta\partial_Y^\xi E^{L^*}(Y,X)}
\end{equation}
as locally $L^2$ functions, for all multiindices $\zeta$, $\xi$ with $m-1\leq \abs{\zeta}\leq m$, $m-1\leq\abs{\xi}\leq m$. 

Furthermore,
if $\abs{\alpha}=m$ then
\begin{equation}
\label{eqn:fundamental:2}
\partial^\alpha\Pi^L\arr H(X)
= \sum_{\abs{\beta}=m} \int_{\R^\dmn} 	\partial_X^\alpha\partial_Y^\beta E^L(X,Y)\,H_{\beta}(Y)\,dY
\quad\text{}
\end{equation}
for almost every $X\notin\supp \arr H$, and for all $\arr H\in L^2(\R^\dmn)$ whose support is not all of $\R^\dmn$.

%%% This is actually Lemma 69
Finally, if $\widetilde E^L$ is any other  function that satisfies the bounds~\eqref{eqn:fundamental:far},~\eqref{eqn:fundamental:far:low} 
and formula~\eqref{eqn:fundamental:2}, then
\begin{align}
\label{eqn:fundamental:unique}
\nabla_X^{m-q}\nabla_Y^{m-s}\widetilde E^L(X,Y)
=
\nabla_X^{m-q}\nabla_Y^{m-s}E^L(X,Y)
\end{align}
as locally $L^2$ functions provided $0\leq q\leq1$ and $0\leq s\leq 1$.
\end{thm}
Here $\Pi^L$ is the Newton potential defined by formula~\eqref{eqn:newton}. 

We remark that in particular, if  $\xi$ is a multiindex with $m-1\leq\abs\xi\leq m$, and if we let
\begin{equation*}u(z,r)=\partial_{y,s}^\xi E^L(z,r,y,s), \quad v(z,r)=\partial_{x,t}^\xi E^L(x,t,z,r)\end{equation*}
then $u\in \dot W^2_{m,loc}(\R^\dmn\setminus\{(y,s)\})$ and $v\in \dot W^2_{m,loc}(\R^\dmn\setminus\{(x,t)\})$ for almost every $(x,t)\in\R^\dmn$ and $(y,s)\in\R^\dmn$, and furthermore
\begin{equation*}Lu=0 \text{ in } \R^\dmn\setminus\{(y,s)\},
\qquad
L^*v=0 \text{ in } \R^\dmn\setminus\{(x,t)\}.\end{equation*}
In particular, we may apply Lemma~\ref{lem:slices} to the fundamental solution in either the first or second variables.

By uniqueness of the fundamental solution, if 
$\mat A$ is $t$-independent, and if $m-1\leq\abs\zeta\leq m$ and $m-1\leq\abs\xi\leq m$, then
\begin{equation}
\label{eqn:fundamental:shift}
\partial_{x,t}^\xi \partial_{y,s}^\zeta E^L(x,t+r,y,s+r) =\partial_{x,t}^\xi \partial_{y,s}^\zeta E^L(x,t,y,s)
\end{equation}
and so
\begin{equation}
\label{eqn:fundamental:vertical}
\partial_t\partial_{x,t}^\xi \partial_{y,s}^\zeta E^L(x,t,y,s) =-\partial_s\partial_{x,t}^\xi \partial_{y,s}^\zeta E^L(x,t,y,s)
.\end{equation}

\subsection{The lower half space}
\label{sec:lower}

Recall that Theorem~\ref{thm:potentials} involves bounds on the quantities $\widetilde N_+(\nabla^{m-k} u)$ and $\mathcal{A}_2^+(t\nabla^m\partial_t^k u)$, where $k=0$ or $k=1$ and where $u$ denotes various potentials. It is notationally convenient to work only in the upper half space.

However, estimates in terms of the two-sided nontangential maximal operator $\widetilde N_*$ defined in formula~\eqref{dfn:NTM:twosides} will also be of use. In particular, in Lemma~\ref{lem:lusin:+} we will pass from bounds on $\widetilde N_*(\nabla^{m-1} u)$ to bounds on $\mathcal{A}_2^+(t\nabla^m u)$, and in Lemma~\ref{lem:p:range} we will pass from bounds on $\widetilde N_*(\partial_\dmn^m \s^L\arr g)$ to bounds on $\widetilde N_+(\nabla^m \s^L\arr g)$.

We observe that we may easily translate bounds valid in the upper half space to bounds valid in the lower half space, using the following argument.

Let $A_{\alpha\beta}^-=(-1)^{\alpha_\dmn+\beta_\dmn} A_{\alpha\beta}$. Observe that if $\mat A$ is bounded or $t$-independent then so is $\mat A^-$. Let $\varphi$ and $u$ be scalar valued functions defined on $\R^\dmn$ and let $\varphi^-(x,t)=\varphi(x,-t)$, $u^-(x,t)=u(x,-t)$.
A straightforward change of variables argument establishes that
\begin{equation*}\langle \nabla^m\varphi,\mat A\nabla^m u\rangle_{\R^\dmn} = \langle \nabla^m\varphi^-,\mat A^-\nabla^m u^-\rangle_{\R^\dmn}.\end{equation*}

Choosing $u=\varphi$, we see that if $\mat A$ satisfies the ellipticity condition~\eqref{eqn:elliptic} then so does $\mat A^-$. 

Let $\arr H\in L^2(\R^\dmn)$ and let $u=\Pi^L\arr H$. Because $\Pi^L\arr H$ is the unique solution to the problem~\eqref{eqn:newton}, we have that if $H_\alpha^-(x,t)=(-1)^{\alpha_\dmn} H_\alpha(x,-t)$, then
\begin{equation}\Pi^{L}\arr H(x,-t)=\Pi^{L^-}\arr H^-(x,t).\end{equation}
By the definition~\eqref{dfn:D:newton:+} of the double layer potential and formula~\eqref{dfn:D:newton:-},
\begin{equation}\D^{\mat A}\arr f(x,-t)=-\D^{\mat A^-}\arr f^-(x,t)\end{equation}
where if $\arr f=\Tr_{m-1}F$, then $\arr f^-=\Tr_{m-1}F^-$. Similarly, by formula~\eqref{dfn:S}, if $g_\gamma^-(x)=(-1)^{\gamma_\dmn} g_\gamma(x)$, then
\begin{equation}\s^L\arr g(x,-t)=\s^{L^-}\arr g^-(x,t).\end{equation}
We may establish the similar formula 
\begin{equation}\s^L_\nabla\arr h(x,-t)=\s^{L^-}_\nabla\arr h{}^-(x,t),\end{equation}
where $h_\beta^-=(-1)^{\beta_\dmn}h_\beta$, using either uniqueness of the fundamental solution, or using formulas~\eqref{eqn:S:S:vertical} and~\eqref{eqn:S:S:horizontal} below.

Thus, we may easily pass from bounds in the upper half space to bounds in the lower half space.

\subsection{Nontangential bounds}
\label{sec:N:general}

In Sections~\ref{sec:S} and~\ref{sec:D} we will use the following two lemmas to establish nontangential bounds.

\begin{lem} \label{lem:N:1}
If $F\in L^2_{loc}(\R^\dmn_+)$ and $x_0\in \R^n$, then
\begin{equation*}
\widetilde N_+ F(x_0)
\leq 
C\sup_{ t_0>0} \biggl(\fint_{Q(x_0,t_0)}\fint_{t_0/6}^{t_0/2}\abs{F}^2\biggr)^{1/2}
\end{equation*}
where $Q(x_0,t_0)$ is the cube in $\R^n$ with midpoint $x_0$ and side length ${t_0}$.
\end{lem}

\begin{proof}
Recall from the definition~\eqref{dfn:NTM:modified} that
\begin{align*}
\widetilde N_+ F(x_0) 
&= \sup
\biggl\{\biggl(\fint_{B((y,t_0/3),{t_0}/6)} \abs{F}^2\biggr)^{1/2}:
\abs{x-y}<t_0/3
\biggr\}.\end{align*}
But $B((y,t_0/3),{t_0}/6) \subset Q(x_0,t_0)\times (t_0/6, t_0/2)$, and so
\begin{align*}
\widetilde N_+ F(x_0) 
&\leq \sup_{t_0>0}
\biggl(\frac{6^\dmn}{\omega_\dmn t_0^\dmn}\int_{Q(x_0,t_0)}\int_{t_0/6}^{t_0/2} \abs{F}^2\biggr)^{1/2}
\end{align*}
where $\omega_\dmn$ is the volume of the unit ball in $\R^\dmn$, as desired
\end{proof}

The following lemma is very useful for bounding solutions in cubes, and in particular in $Q(x_0,t_0)\times (t_0/6,7t_0/6)$ or in $Q(x_0,t_0)\times(-t_0/2,t_0/2)$.

\begin{lem}\label{lem:iterate} Let $L$ be an operator of the form~\eqref{eqn:divergence} of order~$2m$ associated to $t$-independent coefficients~$\mat A$ that satisfy the bounds \eqref{eqn:elliptic:bounded} and~\eqref{eqn:elliptic}.

Let $Q\subset\R^n$ be a cube and let $\widetilde Q=Q\times(s_0-\ell(Q)/2,s_0+\ell(Q)/2)$ be a cube in $\R^\dmn$.

Suppose that $u\in \dot W^2_m(2\widetilde Q)$ and that $Lu=0$ in $2\widetilde Q$. Let $0\leq j\leq m$, and let $s_0-\ell(Q)/2\leq\tau\leq s_0+\ell(Q)/2$. Then
\begin{align*}\fint_{\widetilde Q} \abs{\nabla^j u(x,t)}^2\,dt\,dx
&\leq
C\ell(Q)^2\biggl(\fint_{2\widetilde Q}\abs{\partial_t^{j+1} u(x,t)}\,dt\,dx\biggr)^2
+C \biggl(\fint_{2Q} \abs{\nabla^{j} u(x,\tau)}\,dx\biggr)^2
.\end{align*}
\end{lem}

\begin{proof}
Let $0\leq k\leq j$. Let $\varepsilon>0$ be a small positive number and let $\widetilde Q_k=(1+k\varepsilon)\widetilde Q$.
By Theorem~\ref{thm:Meyers},
\begin{equation*}
\fint_{\widetilde Q_{k}} \abs{\nabla^{j-k}\partial_t^k u(x,t)}^2\,dt\,dx
\leq
C_\varepsilon \biggl(\fint_{\widetilde Q_{k+1/2}} \abs{\nabla^{j-k}\partial_t^ku(x,t)}\,dt\,dx\biggr)^2
.\end{equation*}
If $(x,t)\in\widetilde Q_{k+1/2}$, then
\begin{align*}
\abs{\nabla^{j-k}\partial_t^k u(x,t)}^2
&\leq
	\abs{
	\nabla^{j-k}\partial_t^k u(x,t)-\nabla^{j-k}\partial_t^k u(x,\tau)
	}
	+ \abs{\nabla^{j-k}\partial_t^k u(x,\tau)}
\\&\leq
	C\int_{s_0-\ell(Q_{k+1/2})/2}^{s_0+\ell(Q_{k+1/2})/2}
	\abs{
	\nabla^{j-k}\partial_t^{k+1} u(x,t)
	}\,dt
	%\\&\qquad
	+ \abs{\nabla^{j-k}\partial_t^k u(x,\tau)}
.
\end{align*}
Thus, by H\"older's inequality
\begin{align*}
\fint_{\widetilde Q_k} \abs{\nabla^{j-k}\partial_t^k u(x,t)}^2\,dt\,dx
&
\leq 
{C_\varepsilon}{\ell(Q_{k+1/2})^2}
\fint_{\widetilde Q_{k+1/2}} 
	\abs{\nabla^{j-k}\partial_t^{k+1} u(x,t)}^2\,dt\,dx
\\&\qquad+
C_\varepsilon
\biggl(\fint_{Q_{k+1/2}} \abs{\nabla^{j-k}\partial_t^k u(x,0)}\,dx\biggr)^2
.\end{align*}
If $k\leq j-1$, then by the Caccioppoli inequality,
\begin{align*}
\fint_{\widetilde Q_k} \abs{\nabla^{j-k}\partial_t^k u(x,t)}^2\,dt\,dx
&
\leq 
{C_\varepsilon}
\fint_{\widetilde Q_{k+1}} 
	\abs{\nabla^{j-k-1}\partial_t^{k+1} u(x,t)}^2\,dt\,dx
\\&\qquad+
C_\varepsilon
\biggl(\fint_{Q_{k+1}} \abs{\nabla^{j-k}\partial_t^k u(x,\tau)}\,dx\biggr)^2
.\end{align*}
Iterating, we see that
\begin{align*}
\fint_{\widetilde Q} \abs{\nabla^{j} u(x,t)}^2\,dt\,dx
&
\leq 
{C_\varepsilon}{\ell(Q_{j+1/2})^2}
\fint_{\widetilde Q_{j+1/2}} 
	\abs{\partial_t^{j+1} u(x,t)}^2\,dt\,dx
\\&\qquad+
C_\varepsilon
\biggl(\fint_{Q_{j+1/2}} \abs{\nabla^{j}u(x,0)}\,dx\biggr)^2
.\end{align*}
A final application of Theorem~\ref{thm:Meyers} yields that
\begin{align*}
\fint_{\widetilde Q} \abs{\nabla^{j} u(x,t)}^2\,dt\,dx
&
\leq 
{C_\varepsilon}{\ell(Q_{j+1/2})^2}
\biggl(\fint_{\widetilde Q_{j+1}} 
	\abs{\partial_t^{j+1} u(x,t)}\,dt\,dx\biggr)^2
\\&\qquad+
C_\varepsilon
\biggl(\fint_{Q_{j+1/2}} \abs{\nabla^{j}u(x,0)}\,dx\biggr)^2
.\end{align*}
Letting $\varepsilon=1/(j+1)$ and so $ Q_{j+1} =2 Q$ completes the proof.
\end{proof}

\subsection{Area integral bounds}
\label{sec:lusin:general}

We will use the following lemma to establish the area integral bounds in 
Theorem~\ref{thm:lusin:intro}.

\begin{lem}
\label{lem:lusin:+} 
Let $\arr u\in L^2_{loc}(\R^\dmn_\pm)$ satisfy ${\mathcal{A}_2^+ (t\arr u)}\in L^2(\R^n)$.
Suppose that there is a nonnegative real-valued function $\varphi$ defined on $\R^n$, and a family of functions $\arr u_Q$ indexed by cubes $Q\subset\R^n$, such that if $Q\subset\R^n$ is a cube, then
\begin{equation*}\doublebar{\mathcal{A}_2^+ (t\arr u_Q)}_{L^2(\R^n)}^2
+ \int_{(3/2)Q}\int_0^{\ell(Q)/4} \abs{\arr u(x,t) - \arr u_Q(x,t)}^2\,t\,dt\,dx
\leq \doublebar{\varphi}_{L^2(4Q)}^2.\end{equation*}

Then there is some $\varepsilon>0$ such that
\begin{equation*}\doublebar{\mathcal{A}_2^+(t\,\arr u)}_{L^p(\R^n)} 
\leq C_p\doublebar{\varphi}_{L^p(\R^n)}\end{equation*}
for any $2\leq p<2+\varepsilon$.

In particular, let $L$ be an operator of the form~\eqref{eqn:divergence} of order~$2m$ associated to $t$-independent coefficients~$\mat A$ that satisfy the bounds \eqref{eqn:elliptic:bounded} and~\eqref{eqn:elliptic}, and let $u\in \dot W^2_{m,loc}(\R^\dmn_+\cup\R^\dmn_-)$. If for each cube $Q\subset\R^n$ there is a function $u_Q$ such that $u-u_Q\in \dot W^2_m(3Q\times (-\ell(Q),\ell(Q)) )$ and $L(u-u_Q)=0$ in $3Q\times(-\ell(Q),\ell(Q))$, and if
\begin{equation*}\doublebar{\mathcal{A}_2^+ (t\nabla^m u_Q)}_{L^2(\R^n)}^2
+\doublebar{\widetilde N_*(\nabla^{m-1} u_Q)}_{L^2(\R^n)}^2
\leq \doublebar{\psi}_{L^2(4Q)}^2,\end{equation*}
then there is some $\varepsilon>0$ such that
\begin{equation*}\doublebar{\mathcal{A}_2^+(t\,\nabla^m u)}_{L^p(\R^n)} 
\leq C_p\doublebar{\psi}_{L^p(\R^n)}+C_p\doublebar{\widetilde N_*(\nabla^{m-1} u)}_{L^p(\R^n)}
\end{equation*}
for any $2\leq p<2+\varepsilon$.
\end{lem}

We will use the following lemma.
\begin{lem}[{\cite[Lemma~3.2]{Iwa98}}] \label{lem:iwaniec} Suppose that $g$, $h\in L^q(\R^n)$ are nonnegative real-valued functions, $1<q<\infty$, and that for some $C_0>0$ and for all cubes $Q\subset\R^n$,
\begin{equation*}\biggl(\fint_Q g^q\biggr)^{1/q}\leq C_0\fint_{4Q}g+\biggl(\fint_{4Q} h^q\biggr)^{1/q}.\end{equation*}
Then there exist numbers $s>q$ and $C>0$ depending only on $n$, $q$ and $C_0$ such that if $h\in L^s(\R^n)$, then
\begin{equation*}\int_{\R^n} g^s\leq C\int_{\R^n} h^s.\end{equation*}
\end{lem}
We remark that the assumption $h\in L^q(\R^n)$ is not necessary; it suffices to require $h\in L^q_{loc}(\R^n)$. To see this, we may, for example, use a local version of this lemma (e.g., \cite[Proposition~6.1]{Iwa98}) in larger and larger localized regions.

\begin{proof}[Proof of Lemma~\ref{lem:lusin:+}]

We begin with the special case where $Lu=0$.
By the Caccioppoli inequality,
\begin{multline*}
\int_{(3/2)Q}\int_0^{\ell(Q)/4} \abs{\nabla^m (u-u_Q)(x,t)}^2\,dt\,dx
\\\leq 
\frac{C}{\ell(Q)^2}\int_{3Q} \int_{-\ell(Q)}^{\ell(Q)} \abs{\nabla^{m-1}(u-u_Q)(y,t)}^2\,dy\,dt.
\end{multline*}
It is straightforward to bound the right hand side by $\widetilde N_*(\nabla^{m-1}(u-u_Q))$, and so
\begin{multline*}
\frac{\ell(Q)}{4}\int_{(3/2)Q}\int_0^{\ell(Q)/4} \abs{\nabla^m (u-u_Q)(x,t)}^2\,dt\,dx
\\\leq 
C\int_{3Q} (\widetilde N_*(\nabla^{m-1} u)^2+\widetilde N_*(\nabla^{m-1} u_Q)^2)
.\end{multline*}
By assumption, and because $0<t<\ell(Q)/4$ in the region of integration, we have that
\begin{multline*}
\doublebar{\mathcal{A}_2^+(t\nabla^m u_Q)}_{L^2(\R^n)}^2+
\int_{(3/2)Q}\int_0^{\ell(Q)/4} \abs{\nabla^m (u-u_Q)(x,t)}^2t\,dt\,dx
\\\leq 
C\int_{3Q}( \widetilde N_*(\nabla^{m-1} u)^2)
+C\int_{4Q} \psi^2
.\end{multline*}
Choosing $\varphi^2=C\psi^2+C\widetilde N_*(\nabla^{m-1} u)^2$, $\arr u=\nabla^m u$ and $\arr u_Q=\nabla^m u_Q$, we may reduce to the general case.
%Then 
%\begin{multline*}
%\doublebar{\mathcal{A}_2^+(t\arr u_Q)}_{L^2(\R^n)}^2+
%\int_{(3/2)Q}\int_0^{\ell(Q)/4} \abs{\arr u(x,t)-\arr u_Q(x,t)}^2\,t\,dt\,dx
%\\\leq 
%\doublebar{\mathcal{A}_2^+(t\nabla^m u_Q)}_{L^2(\R^n)}^2+
%\frac{\ell(Q)}{4}
%\int_{(3/2)Q}\int_0^{\ell(Q)/4} \abs{\nabla^m (u-u_Q)(x,t)}^2\,dt\,dx
%.\end{multline*}
%The right hand side is at most $\int_{4Q} \varphi^2$, and so we may reduce to the general case.
%

We now turn to the general case.
Let $Q\subset\R^n$ be a cube. 
By definition of $\mathcal{A}_2^+$,
\begin{align*}\int_Q \mathcal{A}_2^+(t\,\arr u)(x)^2\,dx
&=
	\int_Q\int_0^\infty \int_{\abs{x-y}<t} \abs{\arr u(y,t)}^2\frac{1}{t^{n-1}}\,dy\,dt\,dx
.\end{align*}
We consider the cases $t>\ell(Q)/4$ and $t\leq \ell(Q)/4$ separately, so
\begin{align*}\int_Q \mathcal{A}_2^+(t\,\arr u)(x)^2\,dx
&\leq
	\int_Q\int_0^{\ell(Q)/4} \int_{\abs{x-y}<t} \abs{\arr u(y,t)}^2\frac{1}{t^{n-1}}\,dy\,dt\,dx
	\\&\qquad+
	\int_Q\int_{\ell(Q)/4}^\infty \int_{\abs{x-y}<t} \abs{\arr u(y,t)}^2\frac{1}{t^{n-1}}\,dy\,dt\,dx
.\end{align*}
The first term satisfies
\begin{multline*}
\int_Q\int_0^{\ell(Q)/4} \int_{\abs{x-y}<t} \abs{\arr u(y,t)}^2\frac{1}{t^{n-1}}\,dy\,dt\,dx
\\\leq
	2\int_Q\int_0^{\ell(Q)/4} \int_{\abs{x-y}<t} \abs{\arr u_Q(y,t)}^2\frac{1}{t^{n-1}}\,dy\,dt\,dx
	\\+
	2\int_Q\int_0^{\ell(Q)/4} \int_{\abs{x-y}<t} \abs{\arr u(y,t)- \arr u_Q(y,t)}^2\frac{1}{t^{n-1}}\,dy\,dt\,dx
.\end{multline*}
But
\begin{equation*}\int_Q\int_0^{\ell(Q)/4} \int_{\abs{x-y}<t} \abs{\arr u_Q(y,t)}^2\frac{1}{t^{n-1}}\,dy\,dt\,dx
\leq
\int_{\R^n} \mathcal{A}_2^+(t\arr u_Q)^2
\leq \int_{4Q} \varphi^2.\end{equation*}
We have that
\begin{multline*}\int_Q\int_0^{\ell(Q)/4} \int_{\abs{x-y}<t} \abs{\arr u(y,t)- \arr u_Q(y,t)}^2\frac{1}{t^{n-1}}\,dy\,dt\,dx
\\\leq C_n\int_{(3/2)Q} \int_0^{\ell(Q)/4} \abs{\arr u(y,t)- \arr u_Q(y,t)}^2\,t\,dy\,dt.\end{multline*}
By assumption the right hand side is bounded.
Thus,
\begin{align*}\int_Q \mathcal{A}_2^+(t\,\arr u)(x)^2\,dx
&\leq
	C \int_{4Q}\varphi^2
	+
	\int_Q\int_{\ell(Q)/4}^\infty \int_{\abs{x-y}<t} \abs{\arr u(y,t)}^2\frac{1}{t^{n-1}}\,dy\,dt\,dx
.\end{align*}

Suppose that $x\in Q$, that $t>0$, and that $\abs{x-y}<t$. Then $\dist(y,(3/2)Q)\leq \max(0,t-\ell(Q)/4)$, and so
\begin{multline*}\int_Q \int_{\ell(Q)/4}^\infty \int_{\abs{x-y}<t} \abs{\arr u(y,t)}^2\,\frac{1}{t^{n-1}}\,dy\,dt\,dx
\\\leq
\abs{Q}\int_{\ell(Q)/4}^\infty \int_{\dist(y,(3/2)Q)<t-\ell(Q)/4} \abs{\arr u(y,t)}^2\,\frac{1}{t^{n-1}}\,dy\,dt
.\end{multline*}
Let $\mathcal{G}$ be a grid of $(3N/2)^\dmnMinusOne$ cubes contained in $(3/2)Q$ with side-length $\ell(Q)/N$ and pairwise-disjoint interiors, for $N$ a large even integer to be chosen momentarily. Then 
\begin{multline*}
\abs{Q}\int_{\ell(Q)/4}^\infty \int_{\dist(y,(3/2)Q)<t-\ell(Q)/4} \abs{\arr u(y,t)}^2\,\frac{1}{t^{n-1}}\,dy\,dt
\\\leq
\abs{Q}
\sum_{R\in\mathcal{G}}
\int_{\ell(Q)/4}^\infty \int_{\dist(y,R)<t-\ell(Q)/4} \abs{\arr u(y,t)}^2\,\frac{1}{t^{n-1}}\,dy\,dt
.\end{multline*}
If $z\in R$ and $\dist(y,R)<t-\ell(Q)/4$, then $\abs{z-y}<t-\ell(Q)/4+\ell(R)\sqrt{n}=t+\ell(Q)(\sqrt{n}/N-1/4)$.
Choosing $N\geq 4\sqrt{n}$, we see that for any $z\in R$,
\begin{multline*}\int_{\ell(Q)/4}^\infty \int_{\dist(y,R)<t-\ell(Q)/4} \abs{\arr u(y,t)}^2\,\frac{1}{t^{n-1}}\,dy\,dt
\\\leq 
\int_{\ell(Q)/4}^\infty \int_{\abs{z-y}<t} \abs{\arr u(y,t)}^2\,\frac{1}{t^{n-1}}\,dy\,dt
\leq \mathcal{A}_2^+(t\,\arr u)(z)^2.
\end{multline*}
Averaging over all $z\in R$, we see that
\begin{multline*}
\abs{Q}
\sum_{R\in\mathcal{G}}
\int_{\ell(Q)/4}^\infty \int_{\dist(y,R)<t-\ell(Q)/4} \abs{\arr u(y,t)}^2\,\frac{1}{t^{n-1}}\,dy\,dt
\\\leq
	\abs{Q}\sum_{R\in\mathcal{G}} \biggl(\fint_R\mathcal{A}_2^+(t\,\arr u)(z)\,dz\biggr)^2
\leq
	\abs{Q} (3N/2)^n \biggl(\frac{\abs{2Q}}{\abs{R}}\fint_{2Q}\mathcal{A}_2^+(t\,\arr u)(z)\,dz\biggr)^2
\end{multline*}
and so
\begin{equation*}\int_Q \int_{\ell(Q)/4}^\infty \int_{\abs{x-y}<t} \abs{\arr u(y,t)}^2\,\frac{1}{t^{n-1}}\,dy\,dt\,dx
\leq
	C\abs{Q}\biggl(\fint_{2Q}\mathcal{A}_2^+(t\,\arr u)(z)\,dz\biggr)^2
.\end{equation*}

Thus,
\begin{align*}\int_Q \mathcal{A}_2^+(t\,\arr u)(x)^2\,dx
&\leq
	C \int_{4Q}\varphi^2
	+C\abs{Q}\biggl(\fint_{2Q}\mathcal{A}_2^+(t\,\arr u)(z)\,dz\biggr)^2
.\end{align*}
By Lemma~\ref{lem:iwaniec}, there is some $p>2$ depending on $n$ and $C$ such that
\begin{align*}\int_{\R^n} \mathcal{A}_2^+(t\,\arr u)(x)^p\,dx
&\leq
	C \int_{\R^n}\varphi^p
\end{align*}
as desired.
\end{proof}

\section{The single layer potential}\label{sec:S}

In this section we will establish the nontangential estimates (\ref{eqn:S:N:intro}--\ref{eqn:S:N:rough:intro}) and area integral estimates (\ref{eqn:S:lusin:intro}--\ref{eqn:S:lusin:rough:intro}) on the single layer potential (and modified single layer potential). %Specifically, the nontangential bounds~(\ref{eqn:S:N:intro}--\eqref{eqn:S:N:rough:intro}) will be established in Lem\-ma~\ref{lem:S:N}. The area integral bounds~(\ref{eqn:S:lusin:intro}--\ref{eqn:S:lusin:rough:intro}) will be established in Lemma~\ref{lem:S:lusin}, using Lemma~\ref{lem:lusin:+} and the nontangential estimates (\ref{eqn:S:N:intro}--\ref{eqn:S:N:rough:intro}).

We will begin (Sections~\ref{sec:S:L2} and~\ref{sec:S:variant}) by showing that $\s^L$ and $\s^L_\nabla$ are well defined operators from $L^2(\R^n)$ to $\dot W^2_{m,loc}(\R^\dmn_\pm)$ and $\dot W^2_{m-1,loc}(\R^\dmn_\pm)$, respectively, and recalling or establishing some bounds on $\s^L\arr g$ and $\s^L_\nabla\arr h$ in the cases $\arr g$, $\arr h\in L^2(\R^n)$. In particular, we will show that the boundary operators $\Tr_m^\pm\s^L$ and $\Tr_{m-1}^\pm\s^L_\nabla$ are bounded from $L^2(\R^n)$ to itself.

In Section~\ref{sec:S:trace}, we will show that if the order~$2m$ of the operator~$L$ is high enough, then the boundary operators $\Tr_m^\pm\s^L$ and $\Tr_{m-1}^\pm\s^L_\nabla$ are also bounded from $L^p(\R^n)$ to itself, for $p$ near but not necessarily equal to~$2$. In Section~\ref{sec:S:high} we will pass to the case of operators~$L$ of lower order, and finally in Section~\ref{sec:S:N} will pass from boundary estimates to nontangential (and area integral) estimates.

\subsection{$\s^L$ as an operator on $L^2(\R^n)$}
\label{sec:S:L2}

Recall from the definition~\eqref{dfn:S} that the single layer potential $\s^L$ was originally defined as an operator from $(\dot W\!A^2_{m-1,1/2}(\R^n))^*$ (or $\dot B^{2,2}_{-1/2}(\R^n)$) to $\dot W^2_m(\R^\dmn)$.
Suppose that $\arr g\in \dot B^{2,2}_{-1/2}(\R^n)\cap L^2(\R^n)$. Then by \cite[formula~(4.5)]{BarHM18}, we have that
\begin{align}
\label{eqn:S:L2}
\sup_{t\neq 0} \doublebar{\nabla^m \s^L\arr g(\,\cdot\,,t)}_{L^2(\R^n)} &\leq C\doublebar{\arr g}_{L^2(\R^n)}
.\end{align}
Because $\dot B^{2,2}_{-1/2}(\R^n)\cap L^2(\R^n)$ is dense in $L^2(\R^n)$, we have that $\s^L\arr g$ extends to an operator that is bounded from $L^2(\R^n)$ to $\dot W^2_{m}(\R^n\times(a,b))$ for any $-\infty<a<b<\infty$. 

We have some further properties of $\s^L\arr g$ for $\arr g\in L^2(\R^n)$. 

By Lemma~\ref{lem:slices} and the bound~\eqref{eqn:fundamental:far}, if $\arr g\in L^2(\R^n)$ is compactly supported then the integral in formula~\eqref{eqn:S:fundamental} for $\nabla^m \s^L\arr g(x,t)$ converges absolutely for almost every $(x,t)\in \R^\dmn_\pm$; by density the formula is valid for such~$\arr g$.

By density, we have that $L(\s^L\arr g)=0$  in the weak sense in $\R^\dmn_\pm$ for any $\arr g\in L^2(\R^n)$. 
% The integral in formula~\eqref{eqn:S:fundamental} does NOT converge absolutely for $\arr g\in L^2(\R^n)$.
By \cite[Theorem~5.3]{BarHM17pB}, and 
by the bounds~\eqref{eqn:S:lusin:-} and \eqref{eqn:S:L2}, 
\begin{align}
\label{eqn:S:limit:L2}
\lim_{t\to \pm \infty} \doublebar{\nabla^m \s^L\arr g(\,\cdot\,,t)}_{L^2(\R^n)}
&=0
\quad\text{for all }\arr g\in L^2(\R^n)
.\end{align}
Furthermore, 
the operators $\Tr_m^\pm\s^L$ 
are bounded from $L^2(\R^n)$ to itself 
and satisfy
\begin{align}
\label{eqn:S:boundary:L2}
\lim_{t\to 0^\pm} \doublebar{\nabla^m \s^L\arr g(\,\cdot\,,t)-\Tr_m^\pm\s^L\arr g}_{L^2(\R^n)}
&=0
\quad\text{for all }\arr g\in L^2(\R^n)
.\end{align}

\subsection{The modified single layer potential $\s^L_\nabla\arr h$}
\label{sec:S:variant}

The gradient $\nabla^{m-1}\s^L_\nabla$ of the modified single layer potential was defined by formula~\eqref{eqn:S:variant} as an integral operator. 
We begin this section by showing that there is a function $\s^L_\nabla\arr h$ in $\dot W^2_{m-1,loc}(\R^\dmn_\pm)$ whose gradient $\nabla^{m-1}\s^L_\nabla\arr h$ is given by formula~\eqref{eqn:S:variant}, and that $L(\s^L_\nabla\arr h)=0$ in $\R^\dmn_\pm$.

\begin{lem}\label{lem:S:variant}
Let $L$ be an operator of the form~\eqref{eqn:divergence} of order~$2m$ associated to $t$-independent coefficients~$\mat A$ that satisfy the bounds \eqref{eqn:elliptic:bounded} and~\eqref{eqn:elliptic}.

If $\abs\gamma=m-1$, and if $\arr h\in L^2(\R^n)$, then the integral in the definition~\eqref{eqn:S:variant} of $\partial^\gamma\s^L_\nabla\arr h(x,t)$ converges absolutely for almost every $(x,t)\in\R^\dmn_\pm$. Furthermore, if $K\subset\R^\dmn_\pm$ is compact, then $\partial^\gamma\s^L_\nabla$ is bounded $L^2(\R^n)\mapsto L^2(K)$.

If $\abs\beta=m$ and $\beta_\dmn\geq 1$, and if $h\in \dot B^{2,2}_{-1/2}(\R^n)$ or $h\in L^2(\R^n)$, then the function $\nabla^{m-1}\s^L_\nabla (h\arr e_\beta)$ given by formula~\eqref{eqn:S:variant} satisfies
\begin{equation}\label{eqn:S:S:vertical}\nabla^{m-1}\s^L_\nabla (h\arr e_\beta)(x,t) = -\nabla^{m-1}\partial_t\s^L(h\arr e_\zeta)(x,t)\quad\text{where }\beta=\zeta+\vec e_\dmn.\end{equation}

If $h\in L^2(\R^n)\cap \dot B^{2,2}_{1/2}(\R^n)$, and if $\beta_\dmn<\abs\beta=m$, then the gradient $\nabla^{m}\s^L_\nabla (h\arr e_\beta)$ of the function $\nabla^{m-1}\s^L_\nabla (h\arr e_\beta)$ given by formula~\eqref{eqn:S:variant} satisfies
\begin{equation}\label{eqn:S:S:horizontal}
\nabla^m\s^L_\nabla (h\arr e_\beta)(x,t) = -\nabla^m\s^L(\partial_{x_j} h\arr e_\zeta)(x,t)\quad\text{where }\beta=\zeta+\vec e_j\end{equation}
where $j$ is any number with $1\leq j\leq n$ and with $\vec e_j\leq \beta$.

Thus by density, if $\arr h\in L^2(\R^n)$, then there is a function $\s^L_\nabla\arr h\in \dot W^2_{m-1,loc}(\R^n)$ such that, if $\abs\gamma=m-1$, then $\partial^\gamma \s^L_\nabla\arr h$ is given by formula~\eqref{eqn:S:variant}. Furthermore, 
\begin{equation*}L(\s^L_\nabla\arr h)=0\quad\text{in}\quad\R^\dmn_\pm.\end{equation*}
\end{lem}

\begin{proof} 
By Lemma~\ref{lem:slices} and the bound~\eqref{eqn:fundamental:far} or~\eqref{eqn:fundamental:far:low}, if $Q\subset\R^n$ is a cube of side length $\ell>0$, then
\begin{multline}
\label{eqn:E:cube}
\int_{Q} \int_\ell^{2\ell} \int_{\R^n} \abs{\partial_{x,t}^\gamma\nabla_{y,s}^m E^L(x,t,y,0)}^2\,dy\,dt\,dx 
\\+ \int_{Q} \int_{-2\ell}^{-\ell} \int_{\R^n} \abs{\partial_{x,t}^\gamma\nabla_{y,s}^m E^L(x,t,y,0)}^2\,dy\,dt\,dx
\leq C\ell
.\end{multline}
In particular, $\partial_{x,t}^\gamma\nabla_{y,s}^m E^L(x,t,\,\cdot\,,0)\in L^2(\R^n)$ for almost every $(x,t)\in\R^\dmn_\pm$. A straightforward covering argument establishes the local boundedness of $\partial^\gamma \s^L_\nabla$.

We now turn to formula~\eqref{eqn:S:S:vertical}.
Choose some $\beta$ with $\abs\beta=m$ and $\beta_\dmn\geq 1$.
Let $\zeta+\vec e_\dmn=\beta$. If $h\in \dot B^{2,2}_{-1/2}(\R^n)$, then by formula~\eqref{eqn:S:fundamental} the function $-\partial_t\s^L(h\arr e_\zeta)$ satisfies
\begin{align*}
-\partial^\gamma \partial_t\s^L(h\arr e_\zeta)(x,t)
&=-\int_{\R^n} \partial_{x,t}^\gamma \partial_t\partial_{y,s}^\zeta E^L(x,t,y,0)\,h(y)\,dy
.\end{align*}
and by formula~\eqref{eqn:fundamental:vertical}
\begin{align*}
-\partial^\gamma \partial_t\s^L(h\arr e_\zeta)(x,t)
&=\int_{\R^n} \partial_{x,t}^\gamma \partial_{y,s}^\beta E^L(x,t,y,0)\,h(y)\,dy
.\end{align*}
Thus, formula~\eqref{eqn:S:S:vertical} is valid for all $h\in \dot B^{2,2}_{-1/2}(\R^n)$; because $\partial^\gamma\s^L_\nabla$ and $\nabla^m\s^L$ are both bounded $L^2(\R^n)\mapsto L^2_{loc}(\R^\dmn_\pm)$, by density formula~\eqref{eqn:S:S:vertical} is valid for all $h\in L^2(\R^n)$.

Finally, we turn to formula~\eqref{eqn:S:S:horizontal}.
If $\beta_\dmn<\abs\beta=m$, then there is some $j$ with $1\leq j\leq n$ such that $\vec e_j\leq \beta$. Let $\zeta+\vec e_j=\beta$. If $h\in \dot B^{2,2}_{1/2}(\R^n)$, then the (formal) derivative $\partial_{j}h$ lies in $\dot B^{2,2}_{-1/2}(\R^n)$, and     
\begin{align*}
-\partial^\alpha \s^L(\partial_{j}h\arr e_\zeta)(x,t)
&=-\int_{\R^n} \partial_{x,t}^\alpha \partial_{y,s}^\zeta E^L(x,t,y,0)\,\partial_{y_j}h(y)\,dy
.\end{align*}
Recall from the remarks following formula~\eqref{eqn:S:fundamental} that for almost every $(x,t)\in\R^\dmn_\pm$, the right hand side converges provided $h\in \dot B^{2,2}_{1/2}(\R^n)$.
But if $h\in L^2(\R^n)$, then 
\begin{align*}
\int_{\R^n} \partial_{x,t}^\alpha \partial_{y,s}^\beta E^L(x,t,y,0)\,h(y)\,dy
\end{align*}
converges absolutely for almost every $(x,t)\in\R^\dmn_\pm$. Thus, we may integrate by parts to see that
\begin{align*}
-\partial^\alpha \s^L(\partial_{j}h\arr e_\zeta)(x,t)
&=\int_{\R^n} \partial_{x,t}^\alpha \partial_{y,s}^\beta E^L(x,t,y,0)\,h(y)\,dy
.\end{align*}
If $\abs\gamma=m-1$ and $1\leq k\leq \dmn$, then
\begin{align*}
-\partial_k\partial^\gamma \s^L(\partial_{j}h\arr e_\zeta)(x,t)
&=\partial_k\int_{\R^n} \partial_{x,t}^\gamma \partial_{y,s}^\beta E^L(x,t,y,0)\,h(y)\,dy
.\end{align*}
Thus, formula~\eqref{eqn:S:S:horizontal} is valid. This completes the proof.
\end{proof}

We now establish bounds similar to the bounds~\eqref{eqn:S:L2}, \eqref{eqn:S:limit:L2}, and~\eqref{eqn:S:boundary:L2}.

\begin{lem}
Let $L$ be an operator of the form~\eqref{eqn:divergence} of order~$2m$ associated to $t$-independent coefficients~$\mat A$ that satisfy the bounds \eqref{eqn:elliptic:bounded} and~\eqref{eqn:elliptic}.

For all $\arr h\in L^2(\R^n)$, we have that
\begin{align}
\label{eqn:S:rough:L2}
\sup_{t\neq 0} \doublebar{\nabla^{m-1} \s^L_\nabla \arr h(\,\cdot\,,t)}_{L^2(\R^n)} &\leq C\doublebar{\arr h}_{L^2(\R^n)}
,\\
\lim_{t\to\pm\infty} \doublebar{\nabla^{m-1} \s^L_\nabla \arr h(\,\cdot\,,t)}_{L^2(\R^n)} &=0
.\end{align}
Furthermore, the boundary operator $\Tr_{m-1}^\pm \s^L_\nabla$ 
is bounded from $L^2(\R^n)$ to itself 
and satisfies
\begin{align}
\label{eqn:S:rough:boundary:L2}
\lim_{t\to 0^\pm} \doublebar{\nabla^{m-1} \s^L_\nabla \arr h(\,\cdot\,,t)-\Tr_{m-1}^\pm \s^L_\nabla\arr h}_{L^2(\R^n)} 
&=0
\end{align}
for all $\arr h\in L^2(\R^n)$.
\end{lem}

\begin{proof}
By formula~\eqref{eqn:S:S:vertical}, if $\arr h=h\arr e_\beta$ for some $\beta$ with $\beta_\dmn\geq 1$, then the theorem follows from the bounds~(\ref{eqn:S:L2}--\ref{eqn:S:boundary:L2}).

More generally, by \cite[Theorem~5.1]{BarHM17pB} and the bound~\eqref{eqn:S:lusin:rough:-}, if $\arr h\in L^2(\R^n)$, then there are two polynomials $P_\pm$ of degree $m-1$ that satisfy
\begin{align*}
\sup_{\pm t> 0} \doublebar{\nabla^{m-1} \s^L_\nabla\arr h(\,\cdot\,,t)-\nabla^{m-1}P_\pm}_{L^2(\R^n)} &\leq C\doublebar{\arr h}_{L^2(\R^n)}
,\\
\doublebar{\Tr_{m-1}^\pm \s^L_\nabla\arr h-\nabla^{m-1}P_\pm}_{L^2(\R^n)} &\leq C\doublebar{\arr h}_{L^2(\R^n)}
,\\
\lim_{t\to \pm\infty} \doublebar{\nabla^{m-1} \s^L_\nabla\arr h(\,\cdot\,,t)-\nabla^{m-1}P_\pm}_{L^2(\R^n)} &=0
,\\
\lim_{t\to 0^\pm} \doublebar{\nabla^{m-1} \s^L_\nabla \arr h(\,\cdot\,,t)-\Tr_{m-1}^\pm \s^L_\nabla\arr h}_{L^2(\R^n)} 
&=0
.\end{align*}
We need only show that $\nabla^{m-1}P_\pm=0$.

We will consider only $P_+$.
Let $Q$ be a cube in $\R^n$ of side length $t$. Then
\begin{align*}\abs{\nabla^{m-1}P_+}^2 
&= \fint_Q \abs{\nabla^{m-1}P_+}^2
\\&\leq 2\fint_Q \abs{\nabla^{m-1}\s^L_\nabla\arr h(x,t)-\nabla^{m-1} P_+}^2\,dx + 2\fint_Q \abs{\nabla^{m-1}\s^L_\nabla\arr h(x,t)}^2\,dx.\end{align*}
By the given bound on ${\nabla^{m-1}\s^L_\nabla\arr h(x,t)-\nabla^{m-1} P_+}$,
\begin{align*}\abs{\nabla^{m-1}P_+}^2 
&\leq \frac{C}{t^n}\doublebar{\arr h}_{L^2(\R^n)}^2 + 2\fint_Q \abs{\nabla^{m-1}\s^L_\nabla\arr h(x,t)}^2\,dx.\end{align*}
By the bound~\eqref{eqn:E:cube} and Lemma~\ref{lem:slices},
\begin{align*}\abs{\nabla^{m-1}P_+}^2 
&\leq \frac{C}{t^n}\doublebar{\arr h}_{L^2(\R^n)}.\end{align*}
Letting $t\to\infty$, we see that $\nabla^{m-1}P_+=0$, as desired.
\end{proof}

\subsection{Boundary values and operators of high order}
\label{sec:S:trace}

In this section, we will show that if $2m\geq n+3$, then the boundary operators $\Tr_m^\pm\s^L$ and 
$\Tr_{m-1}^\pm \s^L_\nabla$ are bounded on $L^p(\R^n)$ for some values of $p\neq 2$. We will also establish some preliminary nontangential estimates. In Section~\ref{sec:S:high} we will show how to generalize to operators of low order, and in Section~\ref{sec:S:N} we will pass to nontangential and area integral estimates.

We begin with the purely vertical derivatives.

\begin{lem} \label{lem:p:range:vertical}
Let $L$ be an operator of the form~\eqref{eqn:divergence} of order~$2m$, with $2m\geq n+3$, associated to $t$-independent coefficients~$\mat A$ that satisfy the bounds \eqref{eqn:elliptic:bounded} and~\eqref{eqn:elliptic}. Suppose that $2\leq p<\infty$.
Then for all $\arr g\in L^p(\R^n)\cap L^2(\R^n)$ and all $\arr h\in L^p(\R^n)\cap L^2(\R^n)$, we have that
\begin{align*}
\doublebar{\Trace^\pm \partial_t^{m-1} \s^{L}_\nabla \arr h}_{L^p(\R^n)}\leq C\doublebar{\arr h}_{L^p(\R^n)}
,\\
\doublebar{\Trace^\pm \partial_t^{m} \s^{L} \arr g}_{L^p(\R^n)}\leq C\doublebar{\arr g}_{L^p(\R^n)}
.\end{align*}
\end{lem}

\begin{proof}%[Proof of Lemma~\ref{lem:p:range:vertical}]
By  the bounds~\eqref{eqn:S:boundary:L2} and~\eqref{eqn:S:rough:boundary:L2}, we have that $\Trace^\pm \partial_t^{m-1} \s^{L}_\nabla \arr h$ and $\Trace^\pm \partial_t^{m} \s^{L} \arr g$ do exist as $L^2$ functions.
The bound on $\s^{L}$ follows from the bound on $\s^{L}_\nabla$ by formula~\eqref{eqn:S:S:vertical}.
By formula~\eqref{eqn:S:fundamental} for $\s^{L^*}$, the definition~\eqref{eqn:S:variant} of~$\s^{L}_\nabla$, and the symmetry relations~\eqref{eqn:fundamental:symmetric} and~\eqref{eqn:fundamental:shift} for the fundamental solution,
we have the duality relation
\begin{equation}
\label{eqn:S:dual}
\langle \arr g, \nabla^{m-1}\s_\nabla^{L}\arr h(\,\cdot\,,t)\rangle_{\R^n} = 
\langle \nabla^m \s^{L^*}\arr g(\,\cdot\,,-t),\arr h\rangle_{\R^n}
\end{equation}
for all $t\neq 0$. Taking limits, we see that it suffices to show that the bound
\begin{equation*}\doublebar{\Tr_m^\mp\s^{L^*}(g\arr e_\perp)(\,\cdot\,,t)}_{L^q(\R^n)}\leq C\doublebar{g}_{L^q(\R^n)}\end{equation*}
is valid for all $1<q\leq 2$. Here $\arr e_\perp=\arr e_{(m-1)\vec e_\dmn}$, so $\langle \arr e_\perp,\nabla^{m-1}\varphi\rangle=\partial_\dmn^{m-1} \varphi$ for all functions $\varphi$ with weak derivatives of order up to~$m-1$.

Let $\abs\alpha=m$, and let
\begin{equation*}T_\alpha g=\Trace^\mp\partial^\alpha\s^{L^*}(g\arr e_\perp).\end{equation*}
Again by formula~\eqref{eqn:S:boundary:L2}, $T_\alpha$ is a well-defined, bounded operator on $L^2(\R^n)$.

We now show that $T_\alpha$
satisfies a {weak} bound on $L^1(\R^n)$; by interpolation this operator satisfies a strong bound on $L^q(\R^n)$ for any $q$ in the range $1<q<2$. 

Let $g\in L^1(\R^n)$.
Fix some number $\mu>0$. 
We seek to show that 
\begin{equation*}\abs{\{x:\abs{T_\alpha g(x)}>\mu\}}<\frac{C\doublebar{g}_{L^1(\R^n)}}{\mu}.\end{equation*}

We apply a standard Calder\'on-Zygmund decomposition to~$g$. That is, there exists a collection $\{Q_i\}$ of closed cubes with pairwise-disjoint interiors, a bounded function $s$, and unbounded functions $u_i$ such that
\begin{equation*}g = s+\sum_i u_i,\end{equation*}
such that each $u_i$ is supported in $Q_i$, and such that the following bounds are valid:
\begin{align*}
\doublebar{s}_{L^\infty(\R^n)}&\leq \mu
,\qquad
\int_{Q_i} u_i = 0
,\qquad
\int_{Q_i} \abs{u_i} \leq 2\mu \abs{Q_i}
,\qquad
\sum_i \abs{Q_i} \leq \frac{2^n}{\mu} \int_{\R^n} \abs{g}
.\end{align*}

As usual, if $\abs{T_\alpha g(x)}>\mu$ then either $\abs{T_\alpha s(x)}>\mu/2$ or $\abs{T_\alpha u(x)}>\mu/2$, where $u=\sum_i u_i$.
Notice that
\begin{equation*}\doublebar{s}_{L^2(\R^n)} \leq \biggl( \int_{\R^n\setminus \cup_i Q_i} \abs{s}^2 + \sum_i \abs{Q_i} \mu^2 \biggr)^{1/2}.\end{equation*} 
For almost every $x\notin \cup_i Q_i$, we have that $s(x)=g(x)$ and $\abs{s(x)}\leq \mu$; thus
\begin{equation*}\doublebar{s}_{L^2(\R^n)} \leq \biggl( \int_{\R^n\setminus \cup_i Q_i} \abs{g}\mu + \sum_i \abs{Q_i} \mu^2 \biggr)^{1/2}
\leq
C \mu^{1/2} \doublebar{g}_{L^1(\R^n)}^{1/2}
.\end{equation*} 
Applying boundedness of $T_\alpha$ on $L^2(\R^n)$, we see that
\begin{equation*}\abs{\{x\in\R^n: \abs{T_\alpha s(x)}>\mu/2\}} 
\leq 4 \frac{\doublebar{T_\alpha s}_{L^2(\R^n)}^2}{\mu^2} 
\leq C \frac{\doublebar{s}_{L^2(\R^n)}^2}{\mu^2} 
\leq C^2 \frac{\doublebar{g}_{L^1(\R^n)}}{\mu} 
\end{equation*}
as desired.

We now turn to the set $\abs{\{x\in\R^n: \abs{T_\alpha u(x)}>\mu/2\}} $.
We have that $\sum_i \abs{8Q_i}\leq C {\doublebar{g}_{L^1(\R^n)}}/{\mu} $, and so we will consider only the set
\begin{equation*}\{x\in\R^n: \abs{T_\alpha u(x)}>\mu/2\}\setminus \bigcup_i 8Q_i.\end{equation*}

If $x\notin Q_i$, then by formula~\eqref{eqn:S:fundamental},
\begin{equation*}T_\alpha u_i(x)=\int_{Q_i} (\partial_{x,t}^\alpha \partial_s^{m-1} E^{L^*}(x,t,y,s)-\partial_{x,t}^\alpha \partial_s^{m-1} E^{L^*}(x,t,y_0,s))\big\vert_{s=t=0}\,u_i(y)\,dy\end{equation*}
for any $y_0$; in particular, we choose $y_0$ to be the midpoint of~$Q_i$.
Let $A_j=2^{j+1}Q_i\setminus 2^j Q_i$. Suppose that $j\geq 3$. Then 
\begin{equation*}\int_{A_j} \abs{T_\alpha u_i(x)}\,dx
\leq
	\int_{Q_i} \abs{u_i(y)}\int_{A_j}\abs{\partial_{x,t}^\alpha  \partial_s^{m-1} (E^{L^*}(x,0,y,0)- E^{L^*}(x,0,y_0,0))}\,dx\,dy
.\end{equation*}

Let $w(y,s)=\partial_{x,t}^\alpha \partial_s^{m-1} E^{L^*}(x,t,y,s)$. We observe that $L^* w=0$ away from the point $(x,t)$.
If $2m\geq n+3$, then by Theorem~\ref{thm:Meyers}, $\nabla w$ is continuous and pointwise bounded away from $(x,t)$, and so if $j\geq 3$ then
\begin{multline*}
\int_{A_j}\abs{\partial_{x,t}^\alpha  \partial_s^{m-1} (E^{L^*}(x,0,y,0)- E^{L^*}(x,0,y_0,0))}\,dx
\\\leq 
	C\ell(Q_i)\int_{A_j}
	\fint_{2^{j-2}\widetilde Q_i}\abs{\partial_{x,t}^\alpha \nabla_{y,s}^m E^{L^*}(x,0,y,s)}\,dy\,ds\,dx
\end{multline*}
where $\widetilde Q_i = Q_i\times (-\ell(Q_i)/2,\ell(Q_i)/2)$ is a cube in $\R^\dmn$.
We change the order of integration and apply Lemma~\ref{lem:slices} to the function $v(x,t)=\nabla_{y,s}^m E^{L^*}(x,t,y,s)$ to see that 
\begin{multline*}
\int_{A_j}\abs{\partial_{x,t}^\alpha  \partial_s^{m-1} (E^{L^*}(x,0,y,0)- E^{L^*}(x,0,y_0,0))}\,dx
\\\leq 
	\frac{C}{2^j}
	\fint_{2^{j-2}\widetilde Q_i}\int_{\widetilde A_{j,1}}\abs{\partial_{x,t}^\alpha \nabla_{y,s}^m E^{L^*}(x,t,y,s)}\,dx\,dt\,dy\,ds
\end{multline*}
where $\widetilde A_{j,1}=(A_j\cup A_{j-1}\cup A_{j+1})\times (-2^j\ell(Q),2^j\ell(Q))$.

By H\"older's inequality and the bound~\eqref{eqn:fundamental:far},
\begin{equation*}
\int_{A_j}\abs{\partial_{x,t}^\alpha  \partial_s^{m-1} (E^{L^*}(x,0,y,0)- E^{L^*}(x,0,y_0,0))}\,dx
\leq 
	C2^{-j} 
\end{equation*}
and so
\begin{equation*}\int_{\R^n\setminus 8Q_i} \abs{T_\alpha u_i(x)}\,dx
\leq
	C\int_{Q_i} \abs{u_i(y)}\,dy
.\end{equation*}
Thus,
\begin{equation*}\int_{\R^n\setminus\cup_i 8Q_i} \abs{T_\alpha u} \leq C \int_{\R^n} \abs{g}\end{equation*}
and so
\begin{equation*}\abs{\{x:\abs{T_\alpha u(x)}>\mu/2\}}\leq \frac{C}{\mu}\int_{\R^n} \abs{g}\end{equation*}
as desired.
\end{proof}

We will now establish nontangential estimates on the purely vertical derivatives of the single layer potential. We observe that the conditions of the following lemma are met when $2m\geq n+3$, $k=m$ and $2\leq q<p<\infty$. We will later apply the lemma in the $k=0$ and $q<2$ cases.

\begin{lem}\label{lem:nontangential}
Let $L$ be an operator of the form~\eqref{eqn:divergence} of order~$2m$, with $2m\geq n+1$, associated to $t$-independent coefficients~$\mat A$ that satisfy the bounds \eqref{eqn:elliptic:bounded} and~\eqref{eqn:elliptic}.
Let $p_j^+=p_{L^*,j}^+$ be as in Theorem~\ref{thm:Meyers} with $L$ replaced by~$L^*$, and let $1/p_{j}^++1/p_{j}^-=1$.

Let $0\leq k\leq m$, let $\gamma$ be a multiindex with $\abs\gamma=m-1$, and let ${q}$ and ${p}$ satisfy   $p_{j}^-<{q}<{p}\leq\infty$, where $j=\gamma_\dmn+1$.
Suppose that 
the boundary value operator
\begin{equation*}g\mapsto \Tr_{m-k}^+ \partial_\dmn^k \s^L( g\arr e_\gamma),\end{equation*}
which is well defined for all $g\in L^2(\R^n)$,
extends by density to an operator that is bounded $L^{{q}}(\R^n)\mapsto L^{{q}}(\R^n)$  and $L^{{p}}(\R^n)\mapsto L^{{p}}(\R^n)$.

Then we have the bound
\begin{equation*}\doublebar{\widetilde N_+ (\nabla^{m-k} \partial_\dmn^k \s^L( g\arr e_\gamma))}_{L^{{p}}(\R^n)}\leq C\doublebar{g}_{L^{{p}}(\R^n)}\end{equation*}
for all $g\in L^{{p}}(\R^n)\cap L^2(\R^n)$,
for some constant $C$ depending only on ${q}$, ${p}$ and the standard parameters.

Similarly, if $1\leq k\leq m$, if $\abs\alpha=m$, and if
\begin{equation*}h\mapsto \Tr_{m-k}^+ \partial_\dmn^{k-1} \s^L_\nabla(h\arr e_\alpha)(\,\cdot\,,t)\end{equation*}
is bounded $L^{{q}}(\R^n)\mapsto L^{{q}}(\R^n)$ and $L^{{p}}(\R^n)\mapsto L^{{p}}(\R^n)$ for some $p_{j}^-<{q}<{p}\leq\infty$, where $j=\alpha_\dmn$, then
\begin{equation*}\doublebar{\widetilde N_+ (\nabla^{m-k} \partial_\dmn^{k-1} \s^L_\nabla(h\arr e_\alpha))}_{L^{{p}}(\R^n)}\leq C\doublebar{h}_{L^{{p}}(\R^n)}\end{equation*}
for all $h\in L^{{p}}(\R^n)\cap L^2(\R^n)$.
\end{lem}

\begin{proof} 
Let $\zeta=\alpha-j\vec e_\dmn$ or $\zeta=\gamma-(j-1)\vec e_\dmn$, so $\alpha = (\zeta,j)$ or $\gamma=(\zeta,j-1)$. By formulas~\eqref{eqn:S:fundamental} and~\eqref{eqn:S:variant} for $\s^L$ and $\s^L_\nabla$, and by formula~\eqref{eqn:fundamental:vertical}, we have that if $f\in L^2(\R^n)$ is compactly supported, then for almost every $(x,t)\in\R^\dmn_\pm$,
\begin{align}
\label{eqn:S:nontangential:proof}
-\nabla^{m-k} \partial_t^k \s^L(f\arr e_\gamma)(x,t)
&= (-1)^j\int_{\R^n} \nabla_{x,t}^{m-k} \partial_t^{k+j-1} \partial_y^\zeta E^L(x,t,y,0)\,f(y)\,dy,
\\
\label{eqn:S:rough:nontangential:proof}
\nabla^{m-k} \partial_t^{k-1} \s^L_\nabla(f\arr e_\alpha)(x,t)
&= (-1)^j\int_{\R^n} \nabla_{x,t}^{m-k} \partial_t^{k+j-1} \partial_y^\zeta E^L(x,t,y,0)\,f(y)\,dy.
\end{align}
Choose some $g\in L^{p}(\R^n)\cap L^2(\R^n)$, and let either $u=-\partial_\dmn^k\s^L(g\arr e_\gamma)$ or $u=\partial_\dmn^{k-1}\s^L_\nabla(g\arr e_\alpha)$. 

In either case we wish to bound $\widetilde N_+(\nabla^{m-k}u)$. %Observe that $u\in \dot W^2_{m,loc}(\R^\dmn_+)$ and $Lu=0$ in $\R^\dmn_+$.
Let $x_0\in\R^n$. 
By Lemma \ref{lem:N:1}, 
\begin{equation*}
\widetilde N_+ (\nabla^{m-k} u)(x_0)
\leq 
C\sup_{ t_0>0} \biggl(\fint_{Q(x_0,t_0)}\fint_{t_0/6}^{t_0/2}\abs{\nabla^{m-k} u}^2\biggr)^{1/2}
.\end{equation*}
Choose some $t_0>0$ and let $Q=Q(x_0,t_0)$. Let $\widetilde Q =Q\times(-\ell(Q)/2,\ell(Q)/2) =Q\times(-t_0/2,t_0/2)$. Let $u_Q =-\partial_\dmn^k\s^L(\1_{4Q} g\arr e_\gamma)$ or $u_Q=\partial_\dmn^{k-1}\s^L_\nabla(\1_{4Q}g\arr e_\alpha)$.
Observe that
\begin{equation*}\biggl(\fint_{Q}\fint_{t_0/6}^{t_0/2}\abs{\nabla^{m-k} u}^2\biggr)^{1/2}\!\!
\leq
\biggl(\fint_{Q}\fint_{t_0/6}^{t_0/2}\abs{\nabla^{m-k} u_Q}^2\biggr)^{1/2}
\!\!+
\biggl(3\fint_{\widetilde Q}\abs{\nabla^{m-k} (u-u_Q)}^2\biggr)^{1/2}\llap{.}\end{equation*}

By formulas~\eqref{eqn:S:nontangential:proof} and~\eqref{eqn:S:rough:nontangential:proof} and by H\"older's inequality,
{\multlinegap=0pt\begin{multline*}
\fint_{Q}\fint_{t_0/6}^{t_0/2}\abs{\nabla^{m-k} u_Q}^2
\\\begin{aligned}
&\leq
	\fint_{Q}\fint_{t_0/6}^{t_0/2} 
	\biggl(\int_{4Q} 
	\abs{\nabla_{x,t}^{m-k} \partial_t^{k+j-1} \partial_y^\zeta E^L(x,t,y,0)}^{q'}\,dy \biggr)^{2/q'}
	\!dt\,dx\,\biggl(\int_{4Q} \abs{g}^q \biggr)^{2/q}
.\end{aligned}\end{multline*}}%
If $q>p_j^-$, then $q'<p_j^+$ and so we may use 
Lemma~\ref{lem:slices}, Theorem~\ref{thm:Meyers} and the bound \eqref{eqn:fundamental:far} to bound the integral of $E^L$. Thus,
\begin{equation*}
\biggl(\fint_{Q}\fint_{t_0/6}^{t_0/2}\abs{\nabla^{m-k} u_Q}^2\biggr)^{1/2}
\leq C\biggl(\fint_{4Q} \abs{g}^{{q}}\biggr)^{1/{q}}
\leq C \mathcal{M}(\abs{g}^{{q}})(x_0)^{1/{q}}
.\end{equation*}

By Lemma~\ref{lem:iterate},
\begin{align*}
\biggl(\fint_{\widetilde Q}\abs{\nabla^{m-k} (u-u_Q)}^2\biggr)^{1/2}
&\leq
C\ell(Q)\fint_{2\widetilde Q}\abs{\partial_t^{m-k+1} (u-u_Q)(x,t)}\,dt\,dx
\\&\qquad+
C\fint_{2Q} \abs{\Tr_{m-k}^+u_Q}
+
C\fint_{2Q} \abs{\Tr_{m-k}^+u}
.\end{align*}
The last term is at most $C\mathcal{M}(\Tr_{m-k}^+u)$.
By H\"older's inequality,
\begin{equation*}\fint_{2Q} \abs{\Tr_{m-k}^+  u_Q}
\leq \biggl(\fint_{2Q} \abs{\Tr_{m-k}^+u_Q}^{{q}}\biggr)^{1/{q}}.\end{equation*}
By assumption,
\begin{equation*}
\biggl(\fint_{2Q} \abs{\Tr_{m-k}^+u_Q}^{{q}}\biggr)^{1/{q}}
\leq C\biggl(\fint_{4Q} \abs{g}^{{q}}\biggr)^{1/{q}}
\leq C\mathcal{M}(\abs{g}^{{q}})(x_0)^{1/{q}}.\end{equation*}

%Thus,
%\begin{align*}\biggl(\fint_{Q(x_0,t_0)}\fint_{t_0/6}^{t_0/2}\abs{\nabla^{m-k} u}^2\biggr)^{1/2}
%&\leq C\mathcal{M}(\abs{g}^{{q}})(x_0)^{1/{q}}
%+C\mathcal{M}(\Tr_{m-k}^+u)(x_0)
%\\&\qquad+C\ell(Q)\fint_{2\widetilde Q}\abs{\partial_t^{m-k+1} (u-u_Q)(x,t)}\,dt\,dx.
%\end{align*}

Finally, we consider the term involving $\partial_t^{m-k+1} (u-u_{Q})$. 
By formula~\eqref{eqn:S:nontangential:proof} or~\eqref{eqn:S:rough:nontangential:proof}, for almost every $(x,t)\in\R^\dmn_\pm$ we have that
\begin{equation*}\partial_t^{m-k+1}(u-u_Q)(x,t)=
(-1)^j\int_{\R^n\setminus 4Q} \partial_t^{m+j} \partial_{y}^\zeta E^L(x,t,y,0)\,g(y)\,dy.\end{equation*}
Let $A_\ell=2^{\ell+1}Q\setminus 2^\ell Q$. Then 
\begin{equation*}
\partial_t^{m-k+1}(u-u_Q)(x,t) 
= 
	(-1)^j\sum_{\ell=2}^\infty\int_{A_\ell} \partial_t^{m+j} \partial_{y}^\zeta E^L(x,t,y,0)\,g(y)\,dy
.\end{equation*}
Let 
\[u_\ell(x,t) = \int_{A_\ell} \partial_t^{m+j} \partial_{y}^\zeta E^L(x,t,y,0)\,g(y)\,dy.\]
Observe that $Lu_\ell=0$ away from $A_\ell\times\{0\}$.
If $2m\geq \dmn$ and $\ell\geq 2$, then by Theorem~\ref{thm:Meyers},
\[\sup_{(x,t)\in 2\widetilde Q}\abs{u_\ell(x,t)}
\leq C \biggl(\fint_{2^{\ell-1/2}\widetilde Q} \abs{u_\ell}^2\biggr)^{1/2}
.\]
As before, by H\"older's inequality, Lemma~\ref{lem:slices}, Theorem~\ref{thm:Meyers} and the bound~\eqref{eqn:fundamental:far}, if ${q}'<p_j^+$, then
\[\biggl(\fint_{2^{\ell-1/2}\widetilde Q} \abs{u_\ell}^2\biggr)^{1/2}
\leq
\frac{C}{2^\ell t_0 }  \biggl(\fint_{A_\ell} \abs{g}^q\biggr)^{1/q}
\leq \frac{C}{2^\ell t_0 }  \mathcal{M}(\abs{g}^q)(x_0)^{1/q}.
\]

Thus,
\begin{align*}
\biggl(\fint_{Q}\fint_{t_0/6}^{t_0/2}\abs{\nabla^{m-k} u}^2\biggr)^{1/2}
&\leq
	C\mathcal{M}(\abs{g}^{{q}})(x_0)^{1/{q}}
	+
	C\mathcal{M}({\Tr_{m-k}^+ u})(x_0)
\end{align*}
and so by Lemma~\ref{lem:N:1},
\begin{align*}
\widetilde N_+(\nabla^{m-k}u)(x_0)
&\leq
	C\mathcal{M}(\abs{g}^{{q}})(x_0)^{1/{q}}
	+
	C\mathcal{M}({\Tr_{m-k}^+ u})(x_0)
.\end{align*}
By assumption, if $g\in L^{{p}}(\R^n)$ then ${\Tr_{m-k}^+u} \in L^{{p}}(\R^n)$. We have that ${p}>1$ and so $\mathcal{M}$ is bounded on $L^{{p}}(\R^n)$, and so $\mathcal{M}({\Tr_{m-k}^+ u})\in L^{{p}}(\R^n)$. Furthermore, if $g\in L^{{p}}(\R^n)$, then $\abs{g}^{{q}}\in L^{{p}/{q}}(\R^n)$. If ${p}>{q}$, then ${p}/{q}>1$ and so $\mathcal{M}$ is bounded on $L^{{p}/{q}}(\R^n)$. Thus, the right-hand side is in $L^{{p}}(\R^n)$ and the proof is complete.
\end{proof}

We now extend from boundary values of the purely vertical derivatives to boundary values of the full gradient.

\begin{lem}\label{lem:p:range}
Let $L$ be an operator of the form~\eqref{eqn:divergence} of order~$2m$ associated to $t$-independent coefficients~$\mat A$ that satisfy the bounds \eqref{eqn:elliptic:bounded} and~\eqref{eqn:elliptic}.

Then there is some $p>2$ such that, if $\arr g\in L^p(\R^n)\cap L^2(\R^n)$, then
\begin{equation*}
\doublebar{\Tr_m^+\s^L\arr g}_{L^p(\R^n)}
\leq 
C\doublebar{\arr g}_{L^p(\R^n)}
+ C\doublebar{\widetilde N_*(\partial_\dmn^m \s^L\arr g)}_{L^p(\R^n)}
\end{equation*}
whenever the right hand side is finite.

Similarly, there is some $p>2$ such that, if $\arr h\in L^p(\R^n)\cap L^2(\R^n)$, then
\begin{equation*}
\doublebar{\Tr_{m-1}^+\s^L_\nabla\arr h}_{L^p(\R^n)}
\leq 
C\doublebar{\arr h}_{L^p(\R^n)}
+ C\doublebar{\widetilde N_*(\partial_\dmn^{m-1}\s^L_\nabla\arr g)}_{L^p(\R^n)}
\end{equation*}
whenever the right hand side is finite.
\end{lem}

\begin{proof}%[Proof of Lemma~\ref{lem:p:range}]
We follow the proof of a similar inequality in \cite[pp.~17--19]{HofMitMor15}. Choose some $\arr g\in L^2(\R^n)\cap L^p(\R^n)$, and let $u=s^L\arr g$ or $u=\s^L_\nabla \arr g$. Assume that $\widetilde N_*(\partial_\dmn^\ell u)\in L^p(\R^n)$, where $\ell=m$ or $\ell=m-1$. We wish to show that for some $p>2$ we have that $\Tr_\ell^+u\in L^p(\R^n)$.

As in the proof of Lemma~\ref{lem:lusin:+}, we will use Lemma~\ref{lem:iwaniec}.
For each cube $Q\subset\R^n$, let $\arr g_Q=\arr g\1_{4Q}$ and $\arr g=\arr g_Q+\arr g_{Q,f}$, and let $u_Q=\s^L\arr g_Q$ or $\s^L_\nabla\arr g_Q$. Then
\begin{equation*}\fint_Q \abs{\Tr_\ell^+ u}^2
\leq
2\fint_Q \abs{\Tr_\ell^+ u_Q}^2
+2\fint_Q \abs{\Tr_\ell^+ (u-u_Q)}^2.\end{equation*}
By the bounds \eqref{eqn:S:boundary:L2} and~\eqref{eqn:S:rough:boundary:L2}, 
\begin{equation*}\fint_Q \abs{\Tr_\ell^+ u}^2
\leq
C\fint_{4Q} \abs{\arr g}^2
+2\fint_Q \abs{\Tr_\ell^+ (u-u_Q)}^2
.\end{equation*}
Observe that $L(u-u_Q)=0$ in a neighborhood of $Q\times\{0\}$, and so we may write
\begin{equation*}\fint_Q \abs{\Tr_\ell^+ (u-u_Q)}^2
= \fint_Q \abs{\nabla^{\ell} (u-u_Q)(x,0)}^2\,dx.\end{equation*}
By Lemma~\ref{lem:slices},
\begin{equation*}
\fint_Q \abs{\Tr_\ell^+ (u-u_Q)}^2
\leq C \fint_{(3/2)Q} \fint_{-\ell(Q)/2}^{\ell(Q)/2} \abs{\nabla^{\ell} (u-u_Q)(x,t)}^2\,dt\,dx
.\end{equation*}
By Lemma~\ref{lem:iterate} and H\"older's inequality, 
\begin{align*}
\fint_Q \abs{\Tr_\ell^+ (u-u_Q)}^2
&\leq
{C}{\ell(Q)^2}\fint_{2Q}\fint_{-\ell(Q)}^{\ell(Q)}\abs{\partial_t^{\ell+1} (u-u_Q)(x,t)}^2\,dt\,dx
\\&\qquad+C\biggl(\fint_{2Q} \abs{\nabla^{\ell }(u-u_Q)(x,0)}\,dx\biggr)^2
.\end{align*}
By the Caccioppoli inequality,
\begin{multline*}\ell(Q)^2\fint_{2Q}\fint_{-\ell(Q)}^{\ell(Q)}\abs{\partial_t^{\ell+1} (u-u_Q)(x,t)}^2\,dt\,dx
\\\leq
C\fint_{3Q}\fint_{-(3/2)\ell(Q)}^{(3/2)\ell(Q)}\abs{\partial_t^{\ell} (u-u_Q)(x,t)}^2\,dt\,dx
\end{multline*}

Now, 
\begin{equation*}\fint_{3Q} \fint_{-(3/2)\ell(Q)}^{(3/2)\ell(Q)} \abs{\partial_t^\ell u_Q(x,t)}^2 \,dt\,dx
\leq
\frac{C}{\abs{Q}}\sup_{t\neq 0} \doublebar{\partial_t^\ell u_Q(\,\cdot\,,t)}_{L^2(\R^n)}^2
\end{equation*}
which by the bound~\eqref{eqn:S:L2} or~\eqref{eqn:S:rough:L2} is at most $C\abs{Q}^{-1}\doublebar{\arr g_Q}_{L^2(\R^n)}^2$. Thus,
\begin{align*}
\fint_Q \abs{\Tr_\ell^+ (u-u_Q)}^2
&\leq 
	C\fint_{3Q} \fint_{-(3/2)\ell(Q)}^{(3/2)\ell(Q)} \abs{\partial_\dmn^\ell u(x,t)}^2 \,dt\,dx
	+C \fint_{4Q}\abs{\arr g(x)}^2\,dx
	\\&\qquad
	+C\biggl(\fint_{2Q} \abs{\nabla^{\ell}(u-u_Q)(x,0)}\,dx\biggr)^2
.\end{align*}

An elementary argument shows that
\begin{equation*}\fint_{3Q} \fint_{-(3/2)\ell(Q)}^{(3/2)\ell(Q)} \abs{\partial_\dmn^\ell u(x,t)}^2 \,dt\,dx
\leq
C\fint_{3Q} \widetilde N_*(\partial_\dmn^\ell u)(x)^2\,dx.\end{equation*}

By H\"older's inequality,
\begin{equation*}
\fint_{2Q} \abs{\nabla^{\ell}(u-u_Q)(x,0)}\,dx
\leq 
	\fint_{2Q} \abs{\Tr_\ell^+ u(x)}\,dx
	+\biggl(\fint_{2Q} \abs{\Tr_\ell^+u_Q(x)}^2\,dx\biggr)^{1/2}
\end{equation*}
which by the bound~\eqref{eqn:S:boundary:L2} or~\eqref{eqn:S:rough:boundary:L2} is at most
\begin{equation*}\fint_{2Q} \abs{\Tr_\ell^+ u(x)}\,dx
	+C\biggl(\fint_{4Q} \abs{\arr g(x)}^2\,dx\biggr)^{1/2}.\end{equation*}

Thus, we see that
\begin{align*}
\fint_Q \abs{\Tr_\ell^+ u}^2
&\leq
	C\fint_{4Q} \abs{\arr g}^2	
	+C\fint_{3Q} \widetilde N_*(\partial_\dmn^\ell u)(x)^2 \,dx
	%\\&\qquad
	+C\biggl(\fint_{2Q} \abs{\Tr_\ell^+ u(x)}\,dx\biggr)^{2}
.\end{align*}

We will use Lemma~\ref{lem:iwaniec}.
Let $g=\abs{\Tr_\ell^+ u}$, let  $h=\abs{\arr g}+\widetilde N_*(\partial_\dmn^\ell u)$ and let $q=2$. 
Then there is some $p>2$ such that
\begin{equation*}\int_{\R^n} \abs{\Tr_\ell^+ u(x)}^p\,dx
\leq C\int_{\R^n} \abs{\arr g(x)}^p+\widetilde N_*(\partial_\dmn^\ell u)(x)^p\,dx\end{equation*}
as desired.
\end{proof}

\subsection{Reduction to operators of high order}\label{sec:S:high}

The following formulas were established in \cite{Bar16,BarHM17,BarHM17pA} and inspired by an argument in \cite[Section~2.2]{AusHMT01}; we will use them to pass from the case $2m\geq n+3$ to the general case.

Choose some large number $M$. There are constants $\kappa_\zeta$ such that
\begin{equation*}\Delta^M = \sum_{\abs{\zeta}=M} \kappa_\zeta\, \partial^{2\zeta}.\end{equation*}
In fact, $\kappa_\zeta = M!/\zeta!=M!/(\zeta_1!\zeta_2!\dots\zeta_\dmn!)$, and so we have that $\kappa_\zeta\geq 1$ for all $\abs{\zeta}=M$.

Define the differential operator $\widetilde L=\Delta^M \! L\Delta^M$; that is, $\widetilde L$ is the operator of order $4M+2m$ associated to coefficients $\mat{\widetilde A}$ that satisfy
\begin{equation}\label{eqn:A:high}
\langle \nabla^{m+2M}\varphi,\mat{\widetilde A}\nabla^{m+2M}\psi\rangle
=\langle \nabla^{m}\Delta^M\varphi,\mat{A}\nabla^{m}\Delta^M\psi\rangle
\end{equation}
for all nice test functions $\varphi$ and~$\psi$. 
Observe that $\widetilde{\mat A}$ is $t$-independent and satisfies the bounds \eqref{eqn:elliptic:bounded} and~\eqref{eqn:elliptic}. A precise formula for~$\widetilde A$ may be found in \cite[formula~11.1]{BarHM17}.

Let $\widetilde g_\varepsilon(x) = \sum_{\gamma+2\xi=\varepsilon} \kappa_\xi\,g_\gamma(x)$.  By \cite[formula~(11.2)]{BarHM17}, if $\abs\alpha=m$ then
\begin{align}\label{eqn:S:high}
\partial^\alpha \s^L \arr g(x,t)
&=
\sum_{\abs{\zeta}=M}
\kappa_\zeta \partial^{\alpha+2\zeta} \s^{\widetilde{\mat A}} \arr{ \widetilde{g}}(x,t)
=
\partial^{\alpha} \Delta^M\s^{\widetilde{\mat A}} \arr{ \widetilde{g}}(x,t)
.\end{align}

Similarly, let $\widetilde h_\varepsilon = \sum_{\alpha+2\xi=\varepsilon}\kappa_\xi h_\alpha$.
If $\abs\gamma=m-1$, then \cite[formula~(3.10)]{BarHM17pA},
\begin{align}\label{eqn:S:variant:high}
\partial^\gamma \s^L_\nabla \arr h(x,t)
&=
	\sum_{\abs{\zeta}=M}
	\kappa_\zeta \partial^{\gamma+2\zeta} \s^{\widetilde{L}}_\nabla \arr{\widetilde{h}}(x,t)
=
	\partial^{\gamma} \Delta^M\s^{\widetilde{L}}_\nabla \arr{\widetilde{h}}(x,t)
.\end{align}

\subsection{Nontangential and area integral estimates}
\label{sec:S:N}

We now establish the nontangential bounds~\eqref{eqn:S:N:intro} and~\eqref{eqn:S:N:rough:intro}  on the single layer potential.

\begin{lem}\label{lem:S:N} Let $L$ be an operator of the form~\eqref{eqn:divergence} of order~$2m$ associated to $t$-independent coefficients~$\mat A$ that satisfy the bounds \eqref{eqn:elliptic:bounded} and~\eqref{eqn:elliptic}.

Then the bounds~\eqref{eqn:S:N:intro} and~\eqref{eqn:S:N:rough:intro} are valid. That is, there is some number $\varepsilon>0$ such that the bounds
\begin{equation*}
\doublebar{\widetilde N_*(\nabla^m \s^L\arr g)}_{L^p(\R^n)} 
\leq C_p\doublebar{\arr g}_{L^p(\R^n)}
\end{equation*}
and
\begin{equation*}
\doublebar{\widetilde N_*(\nabla^{m-1} \s^L_\nabla\arr h)}_{L^p(\R^n)} 
\leq C_p\doublebar{\arr h}_{L^p(\R^n)}
\end{equation*}
are valid whenever $2-\varepsilon<p<2+\varepsilon$.
\end{lem}

\begin{proof} 
Let $M$ be large enough that $2\widetilde m = 2m+4M\geq n+3$, and let $\widetilde L$ be the operator of order $2\widetilde m$ associated to the coefficients $\widetilde{\mat A}$ given by formula~\eqref{eqn:A:high}.

By Lemmas \ref{lem:p:range:vertical} and~\ref{lem:nontangential} (with $k=\widetilde m$ or $k=\widetilde m-1$), and by Section~\ref{sec:lower},
we have that the bounds
\begin{equation*}
\doublebar{\widetilde N_*(\partial_\dmn^{\widetilde m}\s^{\widetilde L}\arr {\widetilde g})}_{L^p(\R^n)}\leq C_p\doublebar{\arr {\widetilde g}}_{L^p(\R^n)}
,\qquad
\doublebar{\widetilde N_*(\partial_\dmn^{\widetilde m-1}\s^{\widetilde L}_\nabla\arr {\widetilde h})}_{L^p(\R^n)}\leq C_p\doublebar{\arr {\widetilde h}}_{L^p(\R^n)}
\end{equation*}
are valid for all $2<p<\infty$ and all $\arr{\widetilde g}$, $\arr{\widetilde h}\in L^p(\R^n)\cap L^2(\R^n)$.

Thus, by Lemma~\ref{lem:p:range}, there is some $\tilde p>2$ such that if $p=\tilde p$, then
\begin{equation}
\label{eqn:S:trace}
\doublebar{\Tr^\pm_{\widetilde m}\s^{\widetilde L}\arr {\widetilde g}}_{L^p(\R^n)}\leq C\doublebar{\arr {\widetilde g}}_{L^p(\R^n)}
,\quad
\doublebar{\Tr^\pm_{\widetilde m-1}\s^{\widetilde L}_\nabla\arr {\widetilde h}}_{L^p(\R^n)}\leq C\doublebar{\arr {\widetilde h}}_{L^p(\R^n)}
.\end{equation}
By interpolation, the inequalities~\eqref{eqn:S:trace} are valid for all $p$ with $2\leq p\leq \tilde p$. 

The adjoint operator $L^*$ to~$L$ is also of the form~\eqref{eqn:divergence}, of order~$2\widetilde m$, and associated to $t$-independent coefficients~$\mat{\widetilde A}\vphantom{\mat A}^*$ that satisfy the bounds \eqref{eqn:elliptic:bounded} and~\eqref{eqn:elliptic}. Thus, there is some $\tilde p_*>0$ such that the inequalities~\eqref{eqn:S:trace} are valid, with $L$ replaced by~$L^*$, for all $p$ with $2\leq p\leq \tilde p_*$. 

By the duality relation~\eqref{eqn:S:dual}, the inequalities~\eqref{eqn:S:trace} (with the original~$L$) are valid for all $\tilde p_*'\leq p\leq \tilde p$. 

By Lemma~\ref{lem:nontangential} (with $k=0$), we have that if $\max(\tilde p_*',p_0^-)< p\leq \tilde p$, then
\begin{equation*}
\doublebar{\widetilde N_*(\nabla^{\widetilde m}\s^{\widetilde L}\arr {\widetilde g})}_{L^p(\R^n)}\leq C_p\doublebar{\arr {\widetilde g}}_{L^p(\R^n)}
,\qquad
\doublebar{\widetilde N_*(\nabla^{\widetilde m-1}\s^{\widetilde L}_\nabla\arr {\widetilde h})}_{L^p(\R^n)}\leq C_p\doublebar{\arr {\widetilde h}}_{L^p(\R^n)}
\end{equation*}
An application of formulas~\eqref{eqn:S:high} and~\eqref{eqn:S:variant:high} completes the proof.
\end{proof}

As an immediate corollary we have area integral estimates.

\begin{lem}\label{lem:S:lusin}
Let $L$ be an operator of the form~\eqref{eqn:divergence} of order~$2m$ associated to $t$-independent coefficients~$\mat A$ that satisfy the bounds \eqref{eqn:elliptic:bounded} and~\eqref{eqn:elliptic}. 

Then the bounds~\eqref{eqn:S:lusin:intro} and~\eqref{eqn:S:lusin:rough:intro} are valid. That is, there is some number $\varepsilon>0$ such that the bounds
\begin{align}
\label{eqn:S:lusin}
\doublebar{\mathcal{A}_2^+(t\,\nabla^m \partial_t\s^L\arr g)}_{L^p(\R^n)} 
&\leq C_p\doublebar{\arr g}_{L^p(\R^n)}
,\\
\label{eqn:S:lusin:rough}
\doublebar{\mathcal{A}_2^+(t\,\nabla^m \s^L_\nabla\arr h)}_{L^p(\R^n)} 
&\leq C_p\doublebar{\arr h}_{L^p(\R^n)}
\end{align}
are valid whenever $2-\varepsilon<p<2+\varepsilon$.
\end{lem}

\begin{proof}
The case $2-\varepsilon<p\leq 2$ is known (see formulas \eqref{eqn:S:lusin:-} and~\eqref{eqn:S:lusin:rough:-} above). The $p>2$ case follows from Lemma~\ref{lem:lusin:+} with $u=\partial_t\s^L\arr g$ or $u=\s^L_\nabla\arr h$, $u_Q=\partial_t\s^L(\1_{4Q}\arr g)$ or $u_Q=\s^L_\nabla(\1_{4Q}\arr h)$; by the bounds~\eqref{eqn:S:lusin:-} and~\eqref{eqn:S:lusin:rough:-} and Lemma~\ref{lem:S:N}, the conditions of the lemma are satisfied with $\psi=C\abs{\arr h}$ or $\psi=C\abs{\arr g}$.
\end{proof}

\section{The double layer potential}
\label{sec:D}

In this section we will establish the nontangential estimates~(\ref{eqn:D:N:intro}--\ref{eqn:D:N:rough:intro}) and the area integral estimates~(\ref{eqn:D:lusin:intro}--\ref{eqn:D:lusin:rough:intro}) on the double layer potential. 

We will begin (Section~\ref{sec:D:boundary}) by showing that the boundary values $\Tr_m\D^{\mat A}\arr \varphi$ and $\Tr_{m-1}\D^{\mat A}\arr f$ lie in $L^p(\R^n)$, for $p$ near 2 and for appropriate inputs $\arr\varphi$ and $\arr f$. 
We will then (Section~\ref{sec:D:N:high}) establish the nontangential estimate~\eqref{eqn:D:N:intro} on $\nabla^m\D^{\mat A}\arr\varphi$ in the special case where $2m\geq\dmn$. In Section~\ref{sec:D:high} we will extend to the case $2m<\dmn$. Finally, in Section~\ref{sec:D:final}, we will complete the proof of Theorem~\ref{thm:potentials} by establishing the bounds \eqref{eqn:D:N:rough:intro} and~(\ref{eqn:D:lusin:intro}--\ref{eqn:D:lusin:rough:intro}).

\subsection{Boundary values of the double layer potential}
\label{sec:D:boundary}

We begin by bounding the boundary values of the double layer potential. 

\begin{lem}\label{lem:boundary:D} Let $L$ be an operator of the form~\eqref{eqn:divergence} of order~$2m$ associated to $t$-independent coefficients~$\mat A$ that satisfy the bounds \eqref{eqn:elliptic:bounded} and~\eqref{eqn:elliptic}.
Then there is an $\varepsilon>0$ such that if $2-\varepsilon<p<2+\varepsilon$, then
\begin{align}
\label{eqn:D:boundary:1}
\doublebar{\Tr_{m-1}^+ \D^{\mat A}\arr f}_{L^p(\R^n)} 
&\leq C_p\doublebar{\arr f}_{\dot W\!A^p_{m-1,0}(\R^n)}
,\\
\label{eqn:D:boundary:2}
\doublebar{\Tr_{m}^+ \D^{\mat A}\arr \varphi}_{L^p(\R^n)}
&\leq C_p\doublebar{\arr \varphi}_{\dot W\!A^p_{m-1,1}(\R^n)}
\end{align}
whenever $\arr f\in\dot W\!A^2_{m-1,1/2}(\R^n)\cap\dot W\!A^p_{m-1,0}(\R^n)$ and whenever 
$\arr\varphi=\Tr_{m-1}\Phi$ for some $\Phi\in C^\infty_0(\R^\dmn)$.
\end{lem}

\begin{proof}
By \cite[formulas~(5.4) and~(5.6)]{Bar17}, we have the duality relation
\begin{equation}\label{eqn:S:D:dual}\langle \arr g, \Tr_{m-1}^+ \D^{\mat A}\arr f\rangle_{\R^n} 
= 
	-\langle \M_{\mat A^*}^- \s^{L^*}\arr g,\arr f\rangle_{\R^n} 
\end{equation}
for all $\arr f\in \dot W\!A^2_{m-1,1/2}(\R^n)$ and all $\arr g\in \dot B^{2,2}_{-1/2}(\R^n)$. Here $\M_{\mat A^*}^-$ represents the Neumann boundary operator of \cite{Bar17}; if $u\in \dot W^2_m(\R^\dmn_-)$ then the definition of $\M_{\mat A^*}^- u$ in \cite{Bar17} coincides with that in \cite{BarHM17pB}.

Recall from the definition~\eqref{dfn:S} of $\s^{L^*}$ that if $\arr g\in \dot B^{2,2}_{-1/2}(\R^n)$ then $\s^{L^*}\arr g\in \dot W^2_m(\R^\dmn)$.
By \cite[Theorem~6.2]{BarHM17pB}, if $\s^{L^*}\arr g\in \dot W^2_m(\R^\dmn_-)$ and $1<p<\infty$, then
\begin{multline*}
\abs{\langle \arr g, \Tr_{m-1}^+ \D^{\mat A}\arr f\rangle_{\R^n} } = \abs{\langle \M_{\mat A^*}^- \s^{L^*}\arr g,\arr f\rangle_{\R^n} }
\\\leq
C\doublebar{\arr f}_{\dot W\!A^{p}_{m-1,0}(\R^n)} 
\bigl(\doublebar{\mathcal{A}_2^-(t\nabla^m \partial_t\s^{L^*}\arr g)}_{L^{p'}(\R^n)}+\doublebar{\widetilde N_-(\nabla^m \s^{L^*}\arr g)}_{L^{p'}(\R^n)}\bigr).
\end{multline*}
Here $\mathcal{A}_2^- $ and $\widetilde N_-$ are defined analogously to $\mathcal{A}_2^+ $ and $\widetilde N_+$ in the lower half space.

By Lemmas~\ref{lem:S:N} and~\ref{lem:S:lusin}, if $\arr g\in L^{p'}(\R^n)\cap \dot B^{2,2}_{-1/2}(\R^n)$ for some $2-\varepsilon<{p'}<2+\varepsilon$, then
\begin{equation*}\abs{\langle \arr g, \Tr_{m-1}^+ \D^{\mat A}\arr f\rangle_{\R^n}}
\leq
C\doublebar{\arr f}_{\dot W\!A^{p}_{m-1,0}(\R^n)} \doublebar{\arr g}_{L^{p'}(\R^n)}.
\end{equation*}
Because $L^{p'}(\R^n)\cap \dot B^{2,2}_{-1/2}(\R^n)$ is dense in $L^{p'}(\R^n)$, the bound~\eqref{eqn:D:boundary:1} is valid.

We now turn to the bound~\eqref{eqn:D:boundary:2}. We wish to bound $\Trace^+ \partial^\alpha \D^{\mat A}\arr \varphi$ for all $\abs\alpha=m$. We will need separate arguments for the cases $\alpha_\dmn<m$ and $\alpha_\dmn>0$.

We begin with the case $\alpha_\dmn<\abs\alpha=m$; then there is some $j$ with $1\leq j\leq n$ such that $\alpha_j\geq1$, and so $\alpha=\gamma+\vec e_j$ for some multiindex $\gamma$ with nonnegative entries. Integrating by parts, we have that if $h\in C^\infty_0(\R^n)$ then
\begin{equation*}\langle h, \Trace^+ \partial^\alpha\D^{\mat A}\arr \varphi\rangle_{\R^n} 
= -\langle \partial_{x_j} h, \Trace^+ \partial^\gamma\D^{\mat A}\arr \varphi\rangle_{\R^n} .\end{equation*}
By formula~\eqref{eqn:S:D:dual}, if $\arr \varphi\in \dot W\!A^2_{m-1,1/2}(\R^n)$ then
\begin{equation*}\langle h, \Trace^+ \partial^\alpha\D^{\mat A}\arr \varphi\rangle_{\R^n} 
= \langle \M_{\mat A^*}^- \s^{L^*}(\partial_{x_j} h\arr e_\gamma),\arr \varphi\rangle_{\R^n}.\end{equation*}
The function $\partial_{x_j} h$ is in $\dot B^{2,2}_{-1/2}(\R^n)$, and so $\s^{L^*}(\partial_j h\arr e_\gamma)\in \dot W^2_m(\R^\dmn)$. By \cite[Theorem~6.1]{BarHM17pB}, if $1<p'<\infty$ then
\begin{equation*}\abs{\langle \M_{\mat A^*}^- \s^{L^*}(\partial_j h\arr e_\gamma),\arr \varphi\rangle_{\R^n} }
\leq
C\doublebar{\arr \varphi}_{\dot W\!A^{p}_{m-1,1}(\R^n)} 
\doublebar{\mathcal{A}_2^-(t\nabla^m \s^{L^*}(\partial_j h\arr e_\gamma))}_{L^{p'}(\R^n)}.
\end{equation*}
By formula~\eqref{eqn:S:S:horizontal}, $\nabla^m\s^{L^*}(\partial_{j} h \arr e_\gamma) = -\nabla^m\s^{L^*}_\nabla (h\arr e_\alpha) $. Thus,
\begin{equation*}\abs{\langle h, \Trace^+ \partial^\alpha\D^{\mat A}\arr \varphi\rangle_{\R^n} }
\leq
C\doublebar{\arr \varphi}_{\dot W\!A^p_{m-1,1}(\R^n)} 
\doublebar{\mathcal{A}_2^-(t\nabla^m \s^{L^*}_\nabla( h\arr e_\alpha))}_{L^{p'}(\R^n)}.
\end{equation*}
By Lemma~\ref{lem:S:lusin}, if $h\in L^{p'}(\R^n)$ for some $2-\varepsilon<{p'}<2+\varepsilon$, then
\begin{equation*}\abs{\langle h, \Trace^+ \partial^\alpha\D^{\mat A}\arr \varphi\rangle_{\R^n} }
\leq
C\doublebar{\arr \varphi}_{\dot W\!A^p_{m-1,1}(\R^n)} 
\doublebar{h}_{L^{p'}(\R^n)}.
\end{equation*}
By duality and by density, we have that 
\begin{equation*}\doublebar{\Trace^+\partial^\alpha\D^{\mat A}\arr \varphi}_{L^p(\R^n)}
\leq C_p\doublebar{\arr \varphi}_{\dot W\!A^p_{m-1,1}(\R^n)}
\end{equation*}
whenever $\alpha_\dmn<m$.

Finally, we turn to $\partial_\dmn^m \D^{\mat A}\arr \varphi$. In fact, we will bound $\partial^\alpha \D^{\mat A}\arr \varphi$ for any $\alpha$ with $\abs\alpha=m$ and $\alpha_\dmn>0$. Recall that $\arr \varphi=\Tr_{m-1} \Phi$ for some $\Phi\in \dot W^2_m(\R^\dmn)$. As in the proof of \cite[formula~(6.3)]{BarHM17pA}, for almost every $(x,t)\in\R^\dmn_+$, by formulas~\eqref{dfn:D:newton:-} and \eqref{eqn:fundamental:2} we have that 
\begin{align}
\label{eqn:D:fundamental}
\partial^\alpha \D^{\mat A} \arr \varphi(x,t)
&=
- \!\!\sum_{\substack{\abs{\xi}=\abs{\beta}=m}} \int_{\R^\dmn_-} \partial_{x,t}^\alpha \partial_{y,s}^\xi E^L(x,t,y,s)\,A_{\xi\beta}(y) \, \partial^\beta \Phi(y,s)\,ds\,dy.
\end{align}
Let $\gamma=\alpha-\vec e_\dmn$. By assumption, $\gamma\in(\N_0)^\dmn$ is a multiindex with nonnegative entries.
By formula~\eqref{eqn:fundamental:vertical},
\begin{align*}
\partial^\gamma\partial_t \D^{\mat A} \arr \varphi(x,t)
&=
\sum_{\substack{\abs{\xi}=\abs{\beta}=m}} \int_{\R^\dmn_-} \partial_{x,t}^\gamma \partial_s \partial_{y,s}^\xi E^L(x,t,y,s)\,A_{\xi\beta}(y) \, \partial^\beta \Phi(y,s)\,ds\,dy
.\end{align*}
If $\Phi\in C^\infty_0(\R^\dmn)$, then we may integrate by parts in~$s$ to see that
\begin{align*}
\partial^\gamma\partial_t \D^{\mat A} \arr \varphi(x,t)
&=
-\sum_{\substack{\abs{\xi}=\abs{\beta}=m}} \int_{\R^\dmn_-} \partial_{x,t}^\gamma \partial_{y,s}^\xi E^L(x,t,y,s)\,A_{\xi\beta}(y) \, \partial^\beta \partial_s \Phi(y,s)\,ds\,dy
\\&\qquad
-\sum_{\substack{\abs{\xi}=\abs{\beta}=m}} \int_{\R^n} \partial_{x,t}^\gamma \partial_{y,s}^\xi E^L(x,t,y,0)\,A_{\xi\beta}(y) \, \partial^\beta \Phi(y,0)\,dy.
\end{align*}
By formulas~\eqref{dfn:D:newton:-} and~\eqref{eqn:S:variant}, we have that if $\Phi\in C^\infty_0(\R^\dmn)$ and $\abs\gamma=m-1$, then
\begin{multline}
\label{eqn:D:D:S}
\partial^\gamma\partial_t \D^{\mat A} (\Tr_{m-1}\Phi)(x,t)
\\=
\partial^\gamma\D^{\mat A}(\Tr_{m-1}\partial_\dmn \Phi)(x,t)
-\partial^\gamma \s^L_\nabla(\mat A\Tr_m \Phi)(x,t)
.\end{multline}
Let $\Psi(x,t)=\Phi(x,t)-\eta(t)\,t^m\,\partial_\dmn^m \Phi(x,0)/m!$ for a smooth cutoff function~$\eta$ equal to $1$ near $t=0$. Then $\Tr_{m-1}\Psi=\Tr_{m-1}\Phi$ and $\partial_\dmn^m \Psi(x,0)=0$, and so
\begin{align*}\doublebar{\Tr_{m-1}^+\partial_\dmn \Psi}_{\dot W\!A^p_{m-1,0}(\R^n)}
+\doublebar{\mat A\Tr_m^+ \Psi}_{L^p(\R^n)}
&\leq C\doublebar{\arr \varphi}_{\dot W\!A^p_{m-1,1}(\R^n)}
.\end{align*}
Thus, by the bounds \eqref{eqn:D:boundary:1} and~\eqref{eqn:S:trace}, we have that
\begin{equation*}\doublebar{\Trace^+\partial^\alpha \D^{\mat A}\arr \varphi}_{L^p(\R^n)}
\leq C_p\doublebar{\arr \varphi}_{\dot W\!A^p_{m-1,1}(\R^n)}
\end{equation*}
whenever $\alpha_\dmn>0$ and $p$ is sufficiently close to~2.
This completes the proof.
\end{proof}

\subsection{Nontangential estimates for operators of high order}
\label{sec:D:N:high}
In this section we will establish the bound~\eqref{eqn:D:N:intro} in the special case $2m\geq\dmn$. In Section~\ref{sec:D:high} we will pass to the case of lower order operators, and in Section~\ref{sec:D:final} we will establish the bounds \eqref{eqn:D:N:rough:intro}, \eqref{eqn:D:lusin:intro} and~\eqref{eqn:D:lusin:rough:intro}.

\begin{lem}\label{lem:D:N:high}
Let $L$ be an operator of the form~\eqref{eqn:divergence} of order~$2m\geq\dmn$ associated to $t$-independent coefficients~$\mat A$ that satisfy the bounds \eqref{eqn:elliptic:bounded} and~\eqref{eqn:elliptic}.
Then the bound~\eqref{eqn:D:N:intro} is valid; that is, there is some $\varepsilon>0$ such that if $2-\varepsilon< p<2+\varepsilon$, then
\begin{align*}
\doublebar{\widetilde N_+(\nabla^m\D^{\mat A}\arr \varphi)}_{L^p(\R^n)} \leq C_p \doublebar{\arr \varphi}_{\dot W\!A^p_{m-1,1}(\R^n)}
\end{align*}
for any $\arr\varphi=\Tr_{m-1} \Phi$ for some $\Phi$ smooth and compactly supported.
\end{lem}

The remainder of this section will be devoted to a proof of this lemma.

We will apply Lemma~\ref{lem:N:1} to $\nabla^mu$, where $u=\D^{\mat A}\arr\varphi$. Let $Q=Q(x_0,t_0)\subset\R^n$ be a cube of side length $t_0$ and with midpoint~$x_0$.
As in the proof of Lemma~\ref{lem:nontangential}, by Lemma~\ref{lem:iterate}, if $u_Q$ satisfies $L(u-u_Q)=0$ in $2Q\times (-\ell(Q),\ell(Q))$, then
\begin{align}
\label{eqn:D:N:proof}
\biggl(\fint_{Q}\fint_{t_0/6}^{t_0/2}\abs{\nabla^m u}^2\biggr)^{1/2}
&\leq
\biggl(\fint_{Q}\fint_{t_0/6}^{t_0/2}\abs{\nabla^m u_Q}^2\biggr)^{1/2}
\\&\qquad\nonumber
+C\ell(Q)\fint_{2 Q}\fint_{-\ell(Q)}^{\ell(Q)} \abs{\partial_t^{m+1} (u-u_Q)}
\\&\qquad\nonumber
+C\fint_{2Q} \abs{\Tr_m^+ u_Q}
+C\fint_{2Q} \abs{\Tr_m^+ u}
.\end{align}
The final term is at most $C\mathcal{M}(\Tr_m^+ u)(x_0)=C\mathcal{M}(\Tr_m^+ \D^{\mat A}\arr\varphi)(x_0)$, which we may control using Lem\-ma~\ref{lem:boundary:D} and boundedness of the maximal operator.
We will bound the remaining terms much as in the proof of Lem\-ma~\ref{lem:nontangential}. 
Our first step is to construct an appropriate $u_Q$.

\begin{defn}\label{dfn:D:local} Suppose that $\arr\varphi=\Tr_{m-1}\Phi$ for some $\Phi\in C^\infty_0(\R^\dmn)$, and let $R\subset\R^n$ be a cube. We define $\arr\varphi_R$ as follows.

Let $\rho_R:\R^n\mapsto [0,\infty)$ be smooth, supported in $(4/3)R$ and identically equal to 1 in~$R$, and let $\eta:\R\mapsto[0,\infty)$ be smooth, supported in $(-2,2)$ and equal to $1$ in $(-1,1)$.

Let $\Phi_R(x,t)=\rho_R(x) \eta(2t/\ell(R))(\Phi(x,t)-P_R(x,t))+P_R(x,t)$, where $P_R$ is the polynomial of degree $m-1$ that satisfies $\int_{(4/3)R} \nabla^k \Phi(x,0)-\nabla^k P_R(x,0)\,dx=0$ for any $0\leq k\leq m-1$. Observe that $\nabla^m \Phi_R=0$ outside of $(4/3)R\times (-\ell(R),\ell(R))$ and that $\Phi_R=\Phi$ in $R\times (-\ell(R)/2,\ell(R)/2)$.

Let $\arr\varphi_R=\Tr_{m-1} \Phi_R$. Observe that $\arr\varphi_R=\arr\varphi$ in $R$ and $\arr\varphi_R$ is constant outside $(4/3)R$.

By the Poincar\'e inequality, $\arr\varphi_R\in \dot W\!A^p_{m-1,1}(\R^n)$ for any $1\leq p<\infty$, with $\doublebar{\arr\varphi_R}_{\dot W\!A^p_{m-1,1}(\R^n)} \leq C\doublebar{\nabla_\pureH \arr\varphi}_{L^p((4/3)R)}$. Furthermore, by formula~\eqref{dfn:D:newton:-} for the double layer potential,
\begin{equation*}\D^{\mat A} \arr\varphi -\D^{\mat A}\arr\varphi_R =
\1_-\Phi-\1_-\Phi_R-\Pi^L(\1_-\mat A \nabla^m \Phi) 
+\Pi^L(\1_-\mat A \nabla^m \Phi_R)\end{equation*}
and so $\D^{\mat A} \arr\varphi -\D^{\mat A}\arr\varphi_R\in \dot W^2_{m}(R\times(-\ell(R)/2,\ell(R)/2))$ with $L(\D^{\mat A} \arr\varphi -\D^{\mat A}\arr\varphi_R)=0$ in $R\times (-\ell(R)/2,\ell(R)/2)$.

\end{defn}
We will use this definition again in the proof of Corollary~\ref{cor:D:lusin:+}.

Let $u_Q=\D^{\mat A}\arr\varphi_{8Q}$, so $\arr\varphi=\arr\varphi_{8Q}$ in $8Q$ and $L(u-u_Q)=0$ in $8Q\times(-4\ell(Q),4\ell(Q))$.

Let $q\geq1$. We will impose further conditions on $q$ throughout the proof. By H\"older's inequality,
\begin{equation*}\fint_{2Q} \abs{\Tr_m^+ u_Q(x)}\,dx\leq 
\biggl(\fint_{2Q} \abs{\Tr_m^+ u_Q(x)}^{q}\,dx\biggr)^{1/q}\end{equation*}
and by Lemma~\ref{lem:boundary:D} and the definition~\ref{dfn:D:local} of~$\arr\varphi_{8Q}$, if $\abs{q-2}$ is small enough then
\begin{align*}\fint_{2Q} \abs{\Tr_m^+ \D^{\mat A}\arr\varphi_{8Q}}
&\leq 
C\biggl(\frac{1}{\abs{2Q}}\int_{\R^n} \abs{\nabla_\pureH \arr\varphi_{8Q}}^{q}\biggr)^{1/q}
\leq C\mathcal{M}(\abs{\nabla_\pureH \arr\varphi}^{q})(x_0)^{1/q}
.\end{align*}

To contend with the remaining terms in the bound~\eqref{eqn:D:N:proof}, we will need a decomposition $\arr\varphi=\sum_{j=0}^\infty \arr\varphi_j$ and functions $\Psi_j$ such that $\Tr_{m-1} \Psi_j=\arr\varphi_j$.

Let $\Phi_0=\Phi_{8Q}$ and $\arr\varphi_0=\arr\varphi_{8Q}=\Tr_{m-1}\Phi_0$, and for each $j\geq 1$, let $\Phi_j=\Phi_{2^{j+3}Q}-\Phi_{2^{j+2}Q}$ and $\arr \varphi_j=\Tr_{m-1}\Phi_j = \arr\varphi_{2^{j+3}Q}-\arr\varphi_{2^{j+2}Q}$, where $\Phi_R$ and $\arr\varphi_R$ are as in Definition~\ref{dfn:D:local}.

Then $\arr\varphi=\sum_{j=0}^\infty\arr\varphi_j$,  $\nabla_\pureH\arr\varphi_j=$ outside of $(4/3)2^{j+3}Q\subset2^{j+4}Q$, and \begin{equation*}\doublebar{\nabla_\pureH \arr\varphi_j}_{L^p(\R^n)}\leq C_p\doublebar{\nabla_\pureH \arr\varphi_j}_{L^p(2^{j+4}Q)}\end{equation*}
for any $1\leq p<\infty$. Furthermore, if $j\geq 1$ then $\arr\varphi_j=0$ in $2^{j+2}Q$.

We will need extensions $\Psi_j$ with $\arr\varphi_j=\Tr_{m-1}\Psi_j$.
We have that $\arr\varphi_j=\Tr_{m-1}\Phi_j$; however, we will need $\Psi_j$ to satisfy some bounds in terms of the norms of the boundary values $\arr\varphi_j$, and so we cannot use the obvious extensions $\Psi_j=\Phi_j$.

We define extensions $\Psi_j$ as follows. Let $\varphi_{j,k}(x)=\partial_\dmn^k\Phi_j(x,0)$; we have that $\abs{\nabla_\pureH^{m-k}\varphi_{j,k}(x)}\leq \abs{\arr\varphi_k(x)}$.
Let $\theta:\R^n\mapsto \R$ be smooth, nonnegative, and satisfy the conditions $\int_{\R^n} \theta=1$, $\theta(x)=0$ whenever $\abs{x}>1$, and $\int_{\R^n}x^\zeta \theta(x)\,dx=0$ for all multiindices $\zeta\in(\N_0)^n$ with $1\leq \abs\zeta\leq m-1$. Let $\theta_t(x)=t^{-n}\theta(x/t)$.
Define
\begin{equation*}
H_j(x,t)=\sum_{k=0}^{m-1} \frac{1}{k!} t^k \varphi_{j,k}*\theta_t(x) 
= \sum_{k=0}^{m-1} \frac{1}{k!} t^k \int_{\R^n} \varphi_{j,k}(x-ty)\,\theta(y)\,dy
.\end{equation*}
By the proof of \cite[Lemma~3.3]{BarHM17}, we have that $\Tr_{m-1} H_j=\arr\varphi_j$.
Furthermore, if $x\in \R^n\setminus 2^{j+4}Q$ and $\abs{t}<\dist(x,2^{j+4}Q)$, then $\nabla^{m} H_j(x,t)=0$. Finally, if $j\geq 1$, if $x\in 2^{j+2}Q$, and if $\abs{t}<\dist(x,\R^n\setminus 2^{j+2}Q)$, then $\nabla^{m-1} H_j(x,t)=0$. 
 
Observe that if $\zeta\in (\N_0)^n$ is a multiindex, then $\doublebar{(\partial_\pureH^\zeta \theta)_t}_{L^1(\R^n)}$ is bounded, uniformly in~$t$, and so convolution with $(\partial_\pureH^\zeta \theta)_t$ represents a bounded operator on $L^q(\R^n)$ for any $1\leq q\leq\infty$. Using this fact, it is elementary to show that 
\begin{equation*}
\sup_{t\in\R}\doublebar{\nabla^m H_j(\,\cdot\,,t)}_{L^{q}(\R^n)}\leq C\doublebar{\nabla_\pureH \arr\varphi_j}_{L^{q}(\R^n)}\leq C\doublebar{\nabla_\pureH \arr\varphi}_{L^{q}(2^{j+4}Q)}
.\end{equation*}
This is not true of the function~$\Phi_j$.

However, observe that $\nabla^m H_j$ is not compactly supported. 
Let 
\begin{equation*}\Psi_j(x,t)= (H_j(x,t)-P_j(x,t))\,\eta\biggl(\frac{t}{2^{j+2}\ell(Q)}\biggr)+P_j(x,t),\end{equation*} where $\eta(t)=1$ if $\abs{t}\leq1$ and $\eta(t)=0$ if $\abs{t}\geq2$, and where $P_j(x,t)$ is the polynomial of degree $m-1$ with 
\begin{equation*}\int_{2^{j+5}Q} \int_{-2^{j+3}\ell(Q)}^{-2^{j+2}\ell(Q)} \nabla^k (H_j(x,t)-P_j(x,t))\,dt\,dx=0\end{equation*}
for all $0\leq k\leq m-1$.
Observe that $H_j(x,t)=\Psi_j(x,t)$ whenever $\abs{t}<{-2^{j+2}\ell(Q)}$.
By the Poincar\'e inequality, $\int_{\R^\dmn_-}\abs{\nabla^m \Psi_j}^{q}\leq C 2^j\ell(Q) \doublebar{\nabla_\pureH \arr\varphi}_{L^{q}(2^{j+4}Q)}^{q}$ for any $1\leq q<\infty$.

We now return to the terms in the bound~\eqref{eqn:D:N:proof}.
By Theorem~\ref{thm:Meyers},
\begin{align*}
\biggl(\fint_{Q}\fint_{t_0/6}^{t_0/2}\abs{\nabla^m u_Q}^2\biggr)^{1/2}
&\leq \frac{C}{\ell(Q)^{\pdmn/q}} \doublebar{\nabla^mu_Q}_{L^{q}(\R^\dmn)}
.\end{align*}
By formula~\eqref{dfn:D:newton:-}, if $\abs\alpha=m$, $x\in\R^n$ and $t>0$, then 
\begin{equation*}\partial^\alpha u_Q(x,t)=\partial^\alpha \D^{\mat A} \arr\varphi_0(x,t)=\partial^\alpha \Pi^L(\1_-\mat A\nabla^m \Psi_0)(x,t).\end{equation*} 
By \cite[Lemma~43]{Bar16}, if $p_{L^*,0}^-<q<p_{L,0}^+$, where $p_{L,0}^+$ is as in Theorem~\ref{thm:Meyers} and where $1/p_{L^*,0}^++1/p_{L^*,0}^-=1$, then $\nabla^m \Pi^L$ is bounded $L^{q}(\R^\dmn)\mapsto L^{q}(\R^\dmn)$, and so 
\begin{align*}
\biggl(\fint_{Q}\fint_{t_0/6}^{t_0/2}\abs{\nabla^m u_Q}^2\biggr)^{1/2}
&\leq  \frac{C}{\ell(Q)^{\pdmn/q}} \doublebar{\nabla^m \Psi_0}_{L^{q}(\R^\dmn)}
.\end{align*}
By our bounds on $\Psi_0$,
\begin{align*}\biggl(\fint_{Q}\fint_{t_0/6}^{t_0/2}\abs{\nabla^m u_Q}^2\biggr)^{1/2}
&\leq \frac{C}{\ell(Q)^{n/q}} \doublebar{\nabla_\pureH\arr\varphi_0}_{L^{q}(\R^n)}
\leq \frac{C}{\ell(Q)^{n/q}} \doublebar{\nabla_\pureH\arr\varphi}_{L^{q}(8Q)}
\\&\leq C \mathcal{M} (\abs{\nabla_\pureH \arr\varphi}^{q})(x_0)^{1/q}
.\end{align*}

Finally, let $u_f=u-u_Q=\D^{\mat A}(\arr\varphi-\arr\varphi_0)$. By formula~\eqref{dfn:D:newton:-} for the double layer potential, and because $\arr\varphi=\sum_{j=0}^\infty \arr\varphi_j$, we have that
\begin{align*}
\abs{\partial_t^{m+1} u_f(x,t)}
&=
	\abs[Big]{\partial_t^{m+1} \Pi^L\Bigl(\sum_{j=1}^\infty \1_-\mat A\nabla^m \Psi_j\Bigr)(x,t)}
\\&\leq
	\sum_{j=1}^\infty\abs{\partial_t^{m+1} \Pi^L( \1_-\mat A\nabla^m \Psi_j)(x,t)}
.\end{align*}

Let $x\in 2Q$ and let $-\ell(Q)<t<\ell(Q)$. Recall that if $j\geq 1$, then $\nabla^m \Psi_j=0$ in $\{(y,s):\abs{s}<\dist(y,\R^n\setminus 2^{j+2}Q)\}$, and so $\1_-\mat A\nabla^m \Psi_j=0$ in 
$2^{j+1}Q\times ({-2^{j}\ell(Q)},\infty)$. Thus, $L(\Pi^L( \1_-\mat A\nabla^m \Psi_j))=0$ in this set. If $m$ is large enough, then by the bound~\eqref{eqn:Meyers:lowest},
\begin{align*}
\abs{\partial_t^{m+1} u_f(x,t)}
&\leq
	\sum_{j=1}^\infty
	C\fint_{2^{j+1/3}Q} \fint_{-2^{j-2/3}\ell(Q)}^{2^{j-2/3}\ell(Q)} \abs{\partial_t^{m+1} \Pi^L( \1_-\mat A\nabla^m \Psi_j)(y,s)}\,ds\,dy
.\end{align*}
By H\"older's inequality and the Caccioppoli inequality,
\begin{multline*}
\abs{\partial_t^{m+1} u_f(x,t)}
\\\leq
	\sum_{j=1}^\infty
	\frac{C}{2^j\ell(Q)}
	\biggl(\fint_{2^{j+2/3}Q} \fint_{-2^{j-1/3}\ell(Q)}^{2^{j-1/3}\ell(Q)} \abs{\partial_t^{m} \Pi^L( \1_-\mat A\nabla^m \Psi_j)(y,s)}^2\,ds\,dy\biggr)^{1/2}
\end{multline*}
and by Theorem~\ref{thm:Meyers}, if $q>0$ then
\begin{align*}
\abs{\partial_t^{m+1} u_f(x,t)}
&\leq
	\sum_{j=1}^\infty
	\frac{C}{2^j\ell(Q)}\biggl(\fint_{2^{j+1}Q} \fint_{-2^{j}\ell(Q)}^{2^{j}\ell(Q)} \abs{\partial_t^{m} \Pi^L( \1_-\mat A\nabla^m \Psi_j)}^{q}\biggr)^{1/q}
.\end{align*}
If $p_{L^*,0}^-<q<p_{L,0}^+$, then again by boundedness of $\nabla^m\Pi^L$ and our bounds on $\nabla^m \Psi_j$,
\begin{align*}
\abs{\partial_t^{m+1} u_f(x,t)}
&\leq
	\sum_{j=1}^\infty
	\frac{C}{(2^j\ell(Q))^{n/q+1}} \doublebar{\nabla_\pureH \arr\varphi}_{L^{q}(2^{j+4}Q)}
\leq
	\frac{C}{\ell(Q)}\mathcal{M}(\abs{\nabla_\pureH\arr\varphi}^{q})(x_0)^{1/q}
.\end{align*}
Thus, by Lemma~\ref{lem:N:1},
\begin{equation*}\widetilde N_+(\nabla^m \D^{\mat A}\arr\varphi)(x_0) \leq C \mathcal{M}(\abs{\nabla_\pureH\arr\varphi}^{q})(x_0)^{1/q}
+\mathcal{M}(\Tr_m^+\D^{\mat A}\arr\varphi)\end{equation*}
for any $q$ sufficiently close to~$2$. Choosing $q<p$, we have that by boundedness of the maximal operator~$\mathcal{M}$ and by Lemma~\ref{lem:boundary:D},
\begin{equation*}\doublebar{\widetilde N_+(\nabla^m \D^{\mat A}\arr\varphi)}_{L^p(\R^n)}\leq C \doublebar{\nabla_\pureH \arr\varphi}_{L^p(\R^n)}\end{equation*}
as desired.

\subsection{Reduction to operators of high order}
\label{sec:D:high}

We must now extend to the case of operators of lower order. Recall formulas~\eqref{eqn:S:high} and~\eqref{eqn:S:variant:high}. Our goal is to establish an analogous formula for~$\D^{\mat A}$. That is, we wish to find an operator $\mathcal{O}$ such that
\begin{equation*}\D^{\mat A} \arr\varphi = \Delta^M \D^{\mat{\widetilde A}} (\mathcal{O}\arr\varphi),\end{equation*}
where $\mat{\widetilde A}$ is given by formula~\eqref{eqn:A:high}. We remark that we will need to take somewhat more care in this case, as the natural domain of $\D^{\mat{\widetilde A}}$ is not $\dot B^{2,2}_{1/2}$ but instead a closed proper subset $\smash{\dot W\!A^2_{m-1,1/2}}$.

Let $m\geq 1$ and $M\geq 1$ be integers.
Let $\arr \varphi$ be an array indexed by multiindices of length $m-1$. We define $\mathcal{O}\arr\varphi$ as follows. 

If $\delta\in (\N_0)^\dmn$ is a multiindex with $\abs\delta=m+2M-1$, then there is some nonnegative integer $\ell=\delta_\dmn$ and some multiindex $\xi\in (\N_0)^n$ with $\abs\xi=2M+m-1-\ell$ such that $\delta=(\xi,\ell)$. We define
\begin{align*}
(\mathcal{O}\arr\varphi)_{(\xi,\ell)} &= 0 , & \ell&<2M, \\
(\mathcal{O}\arr\varphi)_{(\xi,\ell)}
&= 
\varphi_{(\xi,\ell-2M)} 
-
\sum_{k=1}^{M}
\sum_{{\abs\zeta=k }}\tilde\kappa_\zeta^M (\mathcal{O}\arr\varphi)_{(\xi+2\zeta,\ell-2k)}
,& 2M\leq \ell&\leq 2M+m-1
\end{align*}
where $\tilde\kappa_\zeta^M=\kappa_{(\zeta,M-k)}=M!/\zeta!(M-k)!$ whenever $\abs\zeta=k$, and where $\kappa_\xi$ and $\zeta!$ are as in Section~\ref{sec:S:high}. 

There are then constants $\mu_{\gamma,\delta}$ depending only on $\gamma$, $\delta$, $m$, $M$ and the dimension $\dmn$ such that
\begin{equation*}(\mathcal{O}\arr\varphi)_\delta=\sum_{\abs\gamma=m-1} \mu_{\gamma,\delta}\varphi_\gamma\end{equation*}
for all $\abs\delta=m+2M-1$. As such, $\mathcal{O}$ is bounded on $L^p(\R^n)$ and $\dot W^p_1(\R^n)$ for any $1\leq p\leq \infty$.

We now show that if $\arr\varphi$ is in the domain of $\D^{\mat A}$, then $\mathcal{O}\arr\varphi$ is in the domain of~$\D^{\mat{\widetilde A}}$.

\begin{lem} Let $m\geq 1$ and let $M\geq 1$.

If $\arr\varphi=\Tr_{m-1}F$ for some $F\in C^\infty_0(\R^\dmn)$, then $\mathcal{O}\arr\varphi=\Tr_{m+2M-1}H$  for some $H\in C^\infty_0(\R^\dmn)$.

If $\arr\varphi=\Tr_{m-1}F$ for some $F\in \dot W^2_m(\R^\dmn)$, then $\mathcal{O}\arr\varphi=\Tr_{m+2M-1}H$  for some $H\in \dot W^2_{m+2M}(\R^\dmn)$.
\end{lem}

\begin{proof}
Let $F_j = \Trace^+\partial_\dmn^j F$. 
If $F\in C^\infty_0(\R^\dmn)$ then $F_j\in C^\infty_0(\R^n)$. If $F\in \dot W^2_m(\R^\dmn)$, then
$\partial_\dmn^j F \in \dot W^2_{m-j}(\R^\dmn)$ for all $0\leq j\leq m-1$, and so $F_j= \Trace^+ \partial_\dmn^j F$ lies in the space $\dot B^{2,2}_{m-j-1/2}(\R^n)$. 

Observe that if $\abs\gamma=m$ and $\gamma=(\xi,j)$ for some $0\leq j\leq m-1$ and some $\xi\in (\N_0)^n$, then $\varphi_\gamma=\partial_\pureH^\xi F_j$.

\textbf{Claim.} There exist functions $\Phi_\ell$, in either $C^\infty_0(\R^n)$ or $\dot B^{2,2}_{m+2M-\ell-1/2}(\R^n)$, such that 
\begin{equation}\label{eqn:O:proof}
(\mathcal{O}\arr\varphi)_{(\xi,\ell)} = \partial_\pureH^\xi \Phi_\ell
.\end{equation}
We will prove this by induction on~$\ell$.

If $\ell<2M$ (and in particular if $\ell=0$ or $\ell=1$), let $\Phi_\ell=0$. By definition of $\mathcal{O}$, formula~\eqref{eqn:O:proof} is valid for all $\ell<2M$.

If $\ell\geq 2M$, then by the induction hypothesis
\begin{align*}
(\mathcal{O}\arr\varphi)_{(\xi,\ell)}
&= 
\partial_\pureH^\xi F_{\ell-2M} 
-
\sum_{k=1}^{M}
\sum_{{\abs\zeta=k }} \widetilde\kappa_\zeta^M
\partial^{\xi+2\zeta} \Phi_{\ell-2k}
.\end{align*}
Recall from Section~\ref{sec:S:high} that $\Delta^M =\sum_{\abs\delta=M} \kappa_\delta \partial^{2\delta}$. 
Observe that $\Delta^M=(\Delta_\pureH+\partial_\dmn^2)^M$, where $\Delta_\pureH$ denotes the Laplacian in $\R^n$ or in the $n$ horizontal variables in $\R^\dmn$, and so by definition of $\widetilde \kappa_\zeta^M$ and the binomial theorem,
\begin{equation*}\sum_{k=0}^M \sum_{\abs\zeta=k} \widetilde\kappa_\zeta^M \partial_\pureH^{2\zeta}\partial_\dmn^{2M-2k}
=\sum_{k=0}^M \binom{M}{k}\Delta_\pureH^{k}\partial_\dmn^{2M-2k}. \end{equation*}
Therefore,
\begin{equation*}\sum_{\abs\zeta=k}\widetilde\kappa_\zeta^M \partial^{2\zeta} = \binom{M}{k} \Delta_\pureH^k\end{equation*}
and so
\begin{align*}
(\mathcal{O}\arr\varphi)_{(\xi,\ell)}
&= 
\partial_\pureH^\xi F_{\ell-2M} 
-
\sum_{k=1}^{M}
\partial^{\xi} \Delta_\pureH^k\Phi_{\ell-2k}
.\end{align*}
Taking $\Phi_\ell=F_{\ell-2M} 
-
\sum_{k=1}^{M}
\Delta_\pureH^k\Phi_{\ell-2k}$, we see that the claim is valid.

We now must assemble the function $H$ from the functions $\Phi_{\ell}$.

If $F\in C^\infty_0(\R^\dmn)$, let $\eta$ be a smooth cutoff function, and let
\begin{equation*}H(x,t)=\sum_{\ell=0}^{m+2M-1} \frac{1}{\ell!}t^\ell\,\eta(t)\, \Phi_\ell(x).\end{equation*}
If $F\in \dot W^2_m(\R^\dmn)$, and so $\Phi_\ell \in \dot B^{2,2}_{m+2M-\ell-1/2}(\R^n)$ for all $0\leq \ell\leq m+2M-1$, it is well known that there is a function $H\in \dot W^2_{m+2M}(\R^\dmn)$ such that $\partial_\dmn^\ell H(x,0)=\Phi_\ell(x)$ for all such~$\ell$. 
For example, as in Lemma~\ref{lem:D:N:high} we may let 
\begin{equation*}H(x,t)=\sum_{\ell=0}^{m+2M-1} \frac{1}{\ell!}t^\ell\,\Phi_\ell*\rho_t(x)\end{equation*}
where $\rho_t(x)=t^{-n}\rho(x/t)$ for some function $\rho$ that is smooth, compactly supported, and satisfies $\int_{\R^n}\rho=1$, $\int_{\R^n}x^\gamma\rho(x)\,dx=0$ for all $\gamma$ with $1\leq \abs\gamma\leq m-1$. An elementary argument involving the Fourier transform completes the proof.
\end{proof}

We have now shown that $\mathcal{O}\arr \varphi$ is the trace of some function~$H$. We now make explicit the relationship between $\arr\varphi$ and~$H$.

\begin{lem}\label{lem:high:3} Let $m\geq 1$ and let $M\geq 1$.
Let $\arr\varphi$ and $H$ be as in the previous lemma.
Then $\arr\varphi=\Tr_{m-1}\Delta^M H$.
\end{lem}

\begin{proof}
Let $\gamma=(\xi,j)$ for some $0\leq j\leq m-1$ and some $\abs\xi=m-1-j$. Recall that $\widetilde \kappa_{\vec 0}^M = \kappa_{(\vec 0,M)} = 1$. We then have that
\begin{align*}\varphi_\gamma(x)
&=\varphi_{(\xi,j)}(x) = \sum_{k=0}^M \sum_{\abs\zeta=k} \widetilde \kappa_\zeta^M (\mathcal O \arr\varphi)_{\xi+2\zeta,j+2M-2k}(x)
.\end{align*}
By definition of~$H$,
\begin{align*}\varphi_\gamma(x)
&= \sum_{k=0}^M \sum_{\abs\zeta=k} \widetilde \kappa_\zeta^M 
\partial_\pureH^{\xi}\partial_\dmn^j \partial_\pureH^{2\zeta}\partial_\dmn^{2M-2k} H(x,0)
\end{align*}
and by definition of $\widetilde\kappa_\zeta^M$,
\begin{align*}\varphi_\gamma(x)
&= 
\partial_\pureH^\xi \partial_\dmn^j \Delta^M H(x,0) 
=\partial^\gamma\Delta^M H(x,0)
\end{align*}
as desired.
\end{proof}

Finally, we establish the analogue to formulas~\eqref{eqn:S:high} and~\eqref{eqn:S:variant:high} for the double layer potential.

\begin{lem}\label{lem:D:high} Let $L$ be an operator of the form~\eqref{eqn:divergence} of order~$2m$ associated to $t$-independent coefficients~$\mat A$ that satisfy the bounds \eqref{eqn:elliptic:bounded} and~\eqref{eqn:elliptic}. Let $M\geq 1$ and let $\widetilde{\mat A}$ be as in formula~\ref{eqn:A:high}.

Let $H\in \dot W^2_{m+2M}(\R^\dmn)$. Then
\begin{equation*}\D^{\mat A} (\Tr_{m-1} \Delta^M H) = \Delta^M \D^{\widetilde{\mat A}} (\Tr_{m+2M-1} H).\end{equation*}

In particular, by Lemma~\ref{lem:high:3}, if $\arr\varphi\in \dot W\!A^2_{m-1,1/2}(\R^n)$, then
\begin{equation*}\D^{\mat A} \arr\varphi = \Delta^M \D^{\widetilde{\mat A}} (\mathcal{O}\arr\varphi).\end{equation*}
\end{lem}

\begin{proof}
Recall that by formula~\eqref{dfn:D:newton:-} for the double layer potential
\begin{equation*}\D^{\mat A}(\Tr_{m-1} \Delta^M H) = \1_-\Delta^M H -\Pi^L(\1_-\mat A\nabla^m\Delta^M H)\end{equation*}
and
\begin{equation*}\Delta^M \D^{\widetilde{\mat A}} (\Tr_{m+2M-1} H)
= \1_-\Delta^M H - \Delta^M \Pi^{\widetilde L} (\1_-\mat{\widetilde A}\nabla^{m+2M} H).\end{equation*}
Thus, we need only show that $\Delta^M \Pi^{\widetilde L} (\1_-\mat{\widetilde A}\nabla^{m+2M} H)=\Pi^L(\1_-\mat A\nabla^m\Delta^M H)$.

By the definition~\eqref{eqn:newton} of the Newton potential, we have that 
\begin{equation*}u=\Pi^L(\1_-\mat A\nabla^m\Delta^M H)\end{equation*}
is the unique element of $\dot W^2_m(\R^\dmn)$ that satisfies
\begin{equation*}\langle \nabla^m\varphi,\mat A\nabla^m u\rangle_{\R^\dmn} = \langle \nabla^m\varphi, \mat A\nabla^m \Delta^M H\rangle_{\R^\dmn_-}\end{equation*}
for all $\varphi\in \dot W^2_m(\R^\dmn)$.

Choose some such~$\varphi$. Then there is a $\Phi\in \dot W^2_{m+2M}(\R^\dmn)$ such that $\Delta^M\Phi=\varphi$. 
Let $v=\Delta^M \Pi^{\widetilde L} (\1_-\mat{\widetilde A}\nabla^{m+2M} H)$. Then
\begin{equation*}\langle \nabla^m\varphi, \mat A\nabla^m v\rangle_{\R^\dmn} 
= \langle \nabla^m \Delta^M \Phi, \mat A\nabla^m \Delta^M \Pi^{\widetilde L} (\1_-\mat{\widetilde A}\nabla^{m+2M} H)\rangle_{\R^\dmn}.\end{equation*}
It is clear from the definition of~$\widetilde L$ in Section~\ref{sec:S:high} that 
\begin{equation*}\langle \nabla^m\varphi, \mat A\nabla^m v\rangle_{\R^\dmn} 
= \langle \nabla^{m+2M} \Phi, \mat {\widetilde A}\nabla^{m+2M} \Pi^{\widetilde L} (\1_-\mat{\widetilde A}\nabla^{m+2M} H)\rangle_{\R^\dmn}.\end{equation*}
Again by formula~\eqref{eqn:newton},
\begin{equation*}\langle \nabla^m\varphi, \mat A\nabla^m v\rangle_{\R^\dmn} 
= \langle \nabla^{m+2M} \Phi, \mat{\widetilde A}\nabla^{m+2M} H\rangle_{\R^\dmn_-}\end{equation*}
and again by the definition of~$\widetilde L$,
\begin{equation*}\langle \nabla^m\varphi, \mat A\nabla^m v\rangle_{\R^\dmn} 
= \langle \nabla^{m}\Delta^M \Phi, \mat{A}\nabla^{m} \Delta^M H\rangle_{\R^\dmn_-}= \langle \nabla^{m}\varphi, \mat{A}\nabla^{m} \Delta^M H\rangle_{\R^\dmn_-}.\end{equation*}
This equation is valid for all $\varphi\in \dot W^2_m(\R^\dmn)$, and so $u=v$, as desired.
\end{proof}

\subsection{Nontangential and area integral estimates}
\label{sec:D:final}

By Lemmas~\ref{lem:D:N:high} and~\ref{lem:D:high}, we have that the bound \eqref{eqn:D:N:intro} is valid; that is, if $L$ and $\mat A$ are as in Theorem~\ref{thm:potentials}, then
there is some $\varepsilon>0$ such that if $2-\varepsilon< p < 2+\varepsilon$, then
\begin{align}
\label{eqn:D:N}
\doublebar{\widetilde N_+(\nabla^m\D^{\mat A}\arr \varphi)}_{L^p(\R^n)} \leq C_p \doublebar{\arr \varphi}_{\dot W\!A^p_{m-1,1}(\R^n)}
\end{align}
for any $\arr\varphi$ that satisfies $\arr\varphi=\Tr_{m-1} \Phi$ for some $\Phi$ smooth and compactly supported. By density, we may extend $\D^{\mat A}$ to an operator from ${\dot W\!A^p_{m-1,1}(\R^n)}$ to $\dot W^2_{m,loc}(\R^\dmn_+)$ that satisfies this bound.

Using this bound, it is straightforward to establish the bounds \eqref{eqn:D:N:rough:intro}--\eqref{eqn:D:lusin:rough:intro}.

\begin{cor}\label{cor:D:lusin:+} Let $L$ and $\mat A$ be as in Theorem~\ref{thm:potentials}. Then the bound~\eqref{eqn:D:N:rough:intro} is valid; that is, there is some $\varepsilon>0$ such that 
\begin{align}
\label{eqn:D:N:rough}
\doublebar{\widetilde N_+(\nabla^{m-1}\D^{\mat A}\arr f)}_{L^p(\R^n)} \leq C_p \doublebar{\arr f}_{\dot W\!A^p_{m-1,0}(\R^n)}
&&\text{if }2-\varepsilon\leq p<2+\varepsilon,
\end{align}
whenever $\arr f=\Tr_{m-1} F$ for some $F\in C^\infty_0(\R^\dmn)$.

Furthermore, there is some $\varepsilon>0$ such that the bounds~\eqref{eqn:D:lusin:intro} and~\eqref{eqn:D:lusin:rough:intro} are valid; that is,
\begin{align}
\label{eqn:D:lusin:+}
\doublebar{\mathcal{A}_2^+(t\nabla^m\partial_t\D^{\mat A}\arr \varphi)}_{L^p(\R^n)} &\leq C_p \doublebar{\arr \varphi}_{\dot W\!A^p_{m-1,1}(\R^n)}
&&\text{if }2\leq p<2+\varepsilon
,\\
\label{eqn:D:lusin:rough:+}
\doublebar{\mathcal{A}_2^+(t\nabla^m\D^{\mat A}\arr f)}_{L^p(\R^n)} &\leq C_p \doublebar{\arr f}_{\dot W\!A^p_{m-1,0}(\R^n)}
&&\text{if }2\leq p<2+\varepsilon,
\end{align}
whenever $\arr f=\Tr_{m-1} F$ and $\arr \varphi=\Tr_{m-1} \Phi$ for some $F$, $\Phi\in C^\infty_0(\R^\dmn)$.
\end{cor}

\begin{proof}
We will use Lemma~\ref{lem:lusin:+} to establish the bound \eqref{eqn:D:lusin:+}.
Let $u=\partial_\dmn\D^{\mat A}\arr \varphi$ and $u_Q=\partial_\dmn\D^{\mat A}\arr \varphi_{3Q}$, where $\arr \varphi_R$ is as in Definition~\ref{dfn:D:local}, and let $\psi=C\abs{\arr \varphi}$. By the bounds~\eqref{eqn:D:N} and~\eqref{eqn:D:lusin:2}, the conditions of Lemma~\ref{lem:lusin:+} are satisfied, and so the bound \eqref{eqn:D:lusin:+} is valid.

By formula~\eqref{eqn:D:D:S}, if $F$, $\Phi\in C^\infty_0(\R^n)$ and $\Tr_{m-1} F=\Tr_{m-1}\partial_\dmn\Phi$, then
\begin{equation*}\D^{\mat A} (\Tr_{m-1} F)= \s^L_\nabla(\mat A\Tr_m \Phi)+\partial_t\D^{\mat A}(\Tr_{m-1}\Phi).\end{equation*}
As in the proof of \cite[formula~(6.3)]{BarHM17pA}, given $\Tr_{m-1}F$ we may find an appropriate $\Phi$ such that 
\begin{equation*}\doublebar{\mat A\Tr_m \Phi}_{L^p(\R^n)}+ \doublebar{\Tr_{m-1}\Phi}_{\dot W\!A^p_{m-1,1}(\R^n)}
\leq 
\doublebar{\Tr_{m-1}F}_{\dot W\!A^p_{m-1,0}(\R^n)}.\end{equation*}
Thus, the bound \eqref{eqn:D:N:rough} follows from Lemma \ref{lem:S:N} and the bound~\eqref{eqn:D:N}, and the bound \eqref{eqn:D:lusin:rough:+} follows from Lemma~\ref{lem:S:lusin} and the bound~\eqref{eqn:D:lusin:+}.
\end{proof}

% Change to \bibliographystyle{amsplain} for numbers
\bibliographystyle{amsalpha}
\bibliography{bibli}
\end{document}